\documentclass[a4paper,12pt,reqno]{amsart}
\usepackage[top=25truemm,bottom=25truemm,left=20truemm,right=20truemm]{geometry}
\usepackage{xcolor}
\usepackage{graphicx}
\usepackage{float}
\usepackage{subcaption}
\usepackage{amssymb,amsmath}
\numberwithin{equation}{section}
\usepackage{amsthm}
\theoremstyle{plain}
\newtheorem{mainthm}{Theorem}
\newtheorem{thm}{Theorem}[section]
\newtheorem{prop}[thm]{Proposition}
\newtheorem{lemma}[thm]{Lemma}
\newtheorem{cor}[thm]{Corollary}
\newtheorem{lemdef}[thm]{Lemma and Definition}
\theoremstyle{definition}
\newtheorem{defn}[thm]{Definition}
\newtheorem{rem}[thm]{Remark}

\newcommand{\R}{\mathbb{R}}

\newcommand{\Max}{\operatorname*{Max}}
\newcommand{\bs}[1]{\boldsymbol{#1}}

\begin{document}
\title[Curvature surfaces]{Curvature surfaces in generic conformally flat hypersurfaces arising from Poincar\'{e} metric \\[1mm]
\small{\textnormal{---Extension and Approximation---}}}
\renewcommand{\thefootnote}{\fnsymbol{footnote}}
\footnotetext[0]{to appear in Osaka Journal of Mathematics.}
\renewcommand{\thefootnote}{\arabic{footnote}}

\author[N. Matsuura]{Nozomu Matsuura}
\address[N. Matsuura]{Fukuoka University,
8-19-1 Nanakuma, Jonan-ku Fukuoka 814-0180, Japan}
\email{nozomu@fukuoka-u.ac.jp}

\author[Y. Suyama]{Yoshihiko Suyama}
\address[Y. Suyama]{Fukuoka University,
8-19-1 Nanakuma, Jonan-ku Fukuoka 814-0180, Japan}
\email{suyama@fukuoka-u.ac.jp}


\thanks{This work was partially supported by
JSPS KAKENHI Grant Number JP19K03507.}

\maketitle

\begin{abstract}
We study generic conformally flat (analytic-)hypersurfaces in the Euclidean $4$-space $\mathbb{R}^4$. Such a local-hypersurface is obtained as an evolution of surfaces issuing from a certain surface in $\mathbb{R}^4$, and then, in consequence, the original surface is a (principal-)curvature surface of the hypersurface.
The Poincar\'{e} metric ${\check g}_H$ of the upper half plane leads to a $6$-dimensional set of rational Riemannian metrics $g_0$ of $\mathbb{R}^2$: on a simply connected open set in the regular domain of $g_0$, a curvature surface $f^0$ with the metric $g_0$ is determined, which we denote by $(f^0,g_0)$. 
In this paper, we choose a suitable metric $g_0$ of $\mathbb{R}^2$ determined by ${\check g}_H$ to get nice curvature surfaces (but it also has degenerate and divergent points in $\mathbb{R}^2$), and clarify the structure of the curvature surfaces $(f^0,g_0)$: the curvature surfaces $(f^0,g_0)$ extend analytically to what kind of set in $\mathbb{R}^2$ beyond the regular set of $g_0$, and then the extended surface $(f^0,g_0)$ is defined on a certain open set of $\mathbb{R}^2$ and bounded in $\mathbb{R}^4$; for the extended surface $(f^0,g_0)$, we explicitly catch the set of degenerate points and the limits in $\mathbb{R}^4$ of both ends of every principal curvature line, and then the two limits of every line for one principal curvature are parallel small circles in a standard $2$-sphere $\mathbb{S}^2$.
Then, every principal curvature line in the extended surface $(f^0,g_0)$ is expressed by a frame field of $\mathbb{R}^4$ induced on the surface from a hypersurface and it lies on a standard $2$-sphere $\mathbb{S}^2$ with line-dependent radius.
We also provide a general method of constructing an approximation of such frame fields, and obtain the entire pictures of those lines including degenerate points of $(f^0,g_0)$.
\end{abstract}

\vspace{3mm}


\section{Introduction}\label{sec:intro}

Let us regard the Poincar\'{e} metric ${\check g}_H:=((dx)^2+(dy)^2)/y^2$ as a singular metric on the whole plane $\R^2$.  
We study curvature surfaces of generic conformally flat (analytic-)hypersurfaces in the Euclidean $4$-space $\R^4$, arising from ${\check g}_H$ of $\R^2$.
The metric ${\check g}_H$ of $\R^2$ gives rise to many generic conformally flat local-hypersurfaces, but there is no known explicit property or representation of such hypersurfaces and their curvature surfaces. In this paper, we clarify the structure of the analytically extended curvature surfaces by using a frame field of $\R^4$ induced on the surface from hypersurfaces.
Here, we say that a hypersurface is generic if it has distinct three principal curvatures at each points, and that a surface is curvature if it is woven of the curvature lines for two principal curvatures of some generic conformally flat hypersurface. Indeed, such lines always make a surface as mentioned at the definition of the Guichard net in this section.
For $n$-dimensional hypersurfaces with $n>3$, there are no generic conformally flat hypersurfaces by the result due to Cartan (\cite{ca}, \cite{la}).

Two papers (\cite{he1}, \cite{cpw}) published in 1994 and 1995, respectively, led to the papers
(\cite{bu}--\cite{ct}, \cite{hsuy}, \cite{su4}) giving important various general properties on generic conformally flat local-hypersurfaces. However, there are very few explicit examples of such hypersurfaces realized in $\R^4$: the only known large classes are conformal product hypersurfaces (\cite{he1}, \cite{la}) and hypersurfaces with cyclic Guichard net
(\cite{hs1}, \cite{spt}--\cite{su3});
in addition, there is a hypersurface in \cite{ct2} and the Guichard net with Bianchi-type (\cite{hs2}). This is indicative of the fact that it is difficult in general to explicitly represent such hypersurfaces in $\R^4$.

Now, we explain the Guichard net for a generic conformally flat hypersurface, mentioned above. It was defined in the paper \cite{he1}. 
Let $f$ be a generic conformally flat hypersurface in $\R^4$ defined on a domain $U$ of $\R^3$, and let $\kappa_i \ (i=1,2,3)$ be the principal curvatures of $f$. 
For the sake of simplicity for the description, we assume that $\kappa_3$ is the middle principal curvature: $\kappa_1>\kappa_3>\kappa_2$ or $\kappa_1<\kappa_3<\kappa_2$. 
Then, a principal curvature line coordinate system $(x,y,z)$ taken in the order $\kappa_i \ (i=1,2,3)$ and a function $\varphi=\varphi(x,y,z)\in (0,\pi/2)$ on $U$ are determined\footnote{The interval $(0,\pi/2)$ can be replaced by each interval in $(-\pi,\pi)$ satisfying $\sin\varphi\cos\varphi\neq 0$.} such that the (non-degenerate) Riemannian metric on $U$
\begin{gather}\label{def:g}
g=\cos^2\varphi(dx)^2+\sin^2\varphi(dy)^2+(dz)^2
\end{gather}
is conformally flat and the first fundamental form $\mathrm{I}_f$ of $f$ is given by
\begin{gather}\label{def:I}
\mathrm{I}_f:=P^2g
=P^2(\cos^2\varphi(dx)^2+\sin^2\varphi(dy)^2+(dz)^2)
\end{gather}
with a function $P=P(x,y,z)\neq 0$ on $U$. 
Furthermore, the principal curvatures $\kappa_1$, $\kappa_2$ are determined from $(\varphi,P,\kappa_3)$ as
\begin{gather}\label{def:kappa12}
\kappa_1=P^{-1}\tan\varphi+\kappa_3,\quad
\kappa_2=-P^{-1}\cot\varphi+\kappa_3,
\end{gather}
where $P^{-1}(x,y,z):=(1/P)(x,y,z)$: in this paper, no inverse function appears.
The conformally flat metric $g$ of \eqref{def:g} is called the (principal) Guichard net\footnote{We call the \emph{canonical} principal Guichard net of $f$ only the Guichard net (see \cite[section 2]{hs3}).} of $f$.
Note that the above coordinate system $(x,y,z)$ and the metric $g$ in \eqref{def:g} are defined only on the domain where $f$ is generic.
Conversely, for any conformally flat metric $g$ on $U$ of \eqref{def:g}, there is a generic conformally flat hypersurface with the Guichard net $g$ uniquely up to a conformal transformation of $\R^4$, if $U$ is simply connected (cf.\ \cite{he1}--\cite{hs1}, \cite{cpw}).
Then, in order to realize the hypersurface in $\R^4$, it is necessary to find out the functions $P$ in \eqref{def:I} and $\kappa_3$ in \eqref{def:kappa12} from $\varphi$ such that the Gauss and the Codazzi equations are satisfied (cf.\ \cite{hs3}, \cite[Proposition 2.3]{su4}).

As mentioned at the beginning, it is difficult to find generic conformally flat hypersurfaces, but is there any way to investigate their structure? For the problem, we proposed in papers \cite{bhs} and \cite{su4} to study curvature surfaces for such hypersurfaces, because they reflect the nature of hypersurfaces.
In fact, the principal curvature line of a curvature surface is also that of the hypersurface including the surface and any generic conformally flat local-hypersurface $f(x,y,z)$ is regarded as a one-parameter family of $(x,y)$-curvature surfaces with the parameter $z$. Although the curvature surfaces in this paper are not explicitly represented, their structures can be studied in detail and the surfaces are visualized through our frame field approximation. Furthermore, in the papers (\cite{fsuy}, \cite{krsuy}), various singularities appearing on surfaces in $3$-dimensional spaces were studied and the explicit criteria for them were given. The curvature surfaces would be also interesting for the problem of singularities on surfaces in $\R^4$ (see (A1)--(A4), (B1), (C1) and (C2) in sub-section \ref{subsec:results}).

To paraphrase the fact that any generic conformally flat local-hypersurface is a one-parameter family of curvature surfaces, it is obtained by an evolution of surfaces issuing from a certain surface in $\R^4$, and then, in consequence, the original surface is a curvature surface of the hypersurface. 
The paper \cite{su4} proved that this method is possible.
A certain analytic (local-)surface $\phi$ in the standard unit $3$-sphere $\mathbb{S}^3$ leads to a curvature surface $f^0$ in $\R^4$ and $\phi$ is the restriction to $f^0$ of the unit normal vector $N$ of a hypersurface including $f^0$: $\phi$ gives rise to an orthonormal frame field of $\R^4$, and the frame field induces a curvature surface (cf.\ \cite[Main Theorem 2]{su4}).
In particular, by \cite[Main Theorem 2]{bhs}, a certain family $\mathrm{Met}^0$, consisting of analytic Riemannian metrics $\check g$ on simply connected open sets $V\subset \R^2$ with constant Gauss curvature $-1$, is determined such that each metric $\check g$ on $V$ leads to a one-parameter family of the Guichard net $g^c$ on some domain $V\times I$ with the parameter $c\neq 0$, where $I$ is an open interval.   
Note that the Cauchy-Kovalevskaya theorem for analytic evolution equations is applied to get $g^c$ on $V\times I$ from $\check g$ on $V$, but $g^c$ is not a direct extension to $V\times I$ of $\check g$ on $V$ (see the sub-section \ref{subsec:results}). Now, for each $g^c$ on $V\times I$ arising from $\check g$ on $V$, the existence of a generic conformally flat hypersurface with the Guichard net $g^c$ on $V\times I$ is theoretically guaranteed as mentioned above. 
When we recognize the fact as the existence on curvature surfaces arising from $\check g$ on $V$ of $\mathrm{Met}^0$, we can directly capture these curvature surfaces from $\check g$ on $V$: any $\check g$ on $V$ leads to a $6$-dimensional set of Riemannian metrics $g_0$ on $V$ such that, for each metric $g_0$ on $V$, a surface $\phi$ in $\mathbb{S}^3$ mentioned above is determined and the curvature surface obtained has the metric $g_0$. Here, the $6$-dimension is the dimension of all Riemannian metrics of the hypersurfaces with the Guichard nets $g^c$ ($c\neq 0$), because, for each $c$, the hypersurfaces with the Guichard net $g^c$ are determined by the arbitrariness of conformal transformations in $\R^4$ and $\check g$ on $V$ specifies one surface for each hypersurface (see Section \ref{subsec:existence}, cf.\ \cite[Main Theorem 1]{su4}).

Now, for the Poincar\'{e} metric ${\check g}_H=((dx)^2+(dy)^2)/y^2$ of $\R^2$, a family $\Phi^c=(\varphi(x,y),c \varphi_z(x,y))$ with parameter $c\neq 0$ consisting of two functions on $\R^2$ is determined (see \eqref{def:phi} in the next subsection \ref{subsec:existence}).
For a simply connected open set $V$ of $\R^2$ satisfying $(\varphi_z(\varphi_z)_x(\varphi_z)_y)(x,y)\neq 0$, the metric ${\check g}_H$ on $V$ belongs to the family $\mathrm{Met}^0$.
Moreover, for each $\Phi^c$, a $5$-dimensional set of rational Riemannian metrics $g_0$ on $\R^2$ is determined. We take $\Phi:=\Phi^1$, then all $5$-dimensional set of the metrics $g_0$ of $\R^2$ was solved in \cite[Example 3.2]{su4}.
In this paper, we choose a suitable metric $g_0$ on $\R^2$ among them to get \emph{nice} curvature surfaces, but it also has degenerate and divergent points in $\R^2$. 
Our first aim is to study an analytic extension of curvature surfaces beyond the regular set in $\R^2$ of the metric $g_0$ and further to clarify the structure of the extended surface including the limits in $\R^4$ of both ends of every principal curvature line, since the extended surface is defined on a certain open set of $\R^2$ and bounded in $\R^4$. Then, all extended principal curvature lines in the surface are expressed by the frame field determining curvature surfaces and they lie on some standard $2$-spheres with line-dependent radii.
The second aim is to construct generally an approximation of such frame fields.
Then, the approximation of each principal curvature line also lies on a standard $2$-sphere.
By the approximation, we give several figures for the extended surface: the entire pictures of some principal curvature lines, the image of the set of degenerate points in the surface and so on.

Now, our pair $\Phi=\Phi^1$ locally determines the functions $\varphi(x,y,z)$ in \eqref{def:g}, $P(x,y,z)$ in \eqref{def:I} and $\kappa_3(x,y,z)$ in \eqref{def:kappa12} for some generic conformally flat hypersurfaces, of which fact we explain in the following sub-section \ref{subsec:existence}.
We summarize the results of this paper in sub-section \ref{subsec:results}.

\subsection{Existence of local curvature surfaces}\label{subsec:existence}

We state the setting in this paper, and briefly review the results for generic conformally flat local-hypersurfaces in the papers \cite{bhs} and \cite{su4} in addition to the matters discussed above, which make our problems clearer.
Let $\check g_H$ be the Poincar\'{e} metric on $\R^2$.
From now on, we assume that the domain $U$, where hypersurfaces $f(x,y,z)$ are defined, is given by $U=V\times I\subset \R^2\times \R$ for a simply connected open set $V\subset \R^2=\R^2_{(x,y)}$ and a suitable open interval $I\subset \R=\R_{(z)}$ with $0\in I$.

(R1)\;
For $\check g_H$ on $\R^2$, a pair $\Phi(=\Phi^1)=(\varphi(x,y), \varphi_z(x,y))$ of functions is determined as
\begin{gather}\label{def:phi}
\cos\varphi(x,y)=
\frac{x^2-y^2}{x^2+y^2},\quad
\sin\varphi(x,y)=\frac{2xy}{x^2+y^2},\quad
\varphi_z(x,y)=
\frac{y}{x^2+y^2},
\end{gather}
which are analytic functions on $\R^2$ with the pole at the origin. Here, $\varphi(x,y)$ is defined on $\R^2$ except for the line $x=0$ as $\varphi(x,y)\in(-\pi,\pi)$ for $x>0$ and $\varphi(-x,y)=-\varphi(x,y)$: explicitly, $\varphi(x,y)$ for $x>0$ is determined as follows, when $x^2-y^2<0$,
\begin{equation*}
\varphi(x,y)\in (-\pi,-\pi/2) \ \ \text{if} \ \ y<0 \ \ \text{and} \ \ \varphi(x,y)\in(\pi/2,\pi) \ \ \text{if} \ y>0;
\end{equation*}
when $x^2-y^2=0$, \ $\varphi(x,y)=-\pi/2$ \ if $y<0$ \ and \ $\varphi(x,y)=\pi/2$ \ if $y>0$; and
\begin{equation*}
\varphi(x,y)\in(-\pi/2,\pi/2) \ \ \text{if} \ x^2-y^2>0.
\end{equation*}

(R2)\;
Let $S_1$ be the set defined by
\begin{gather}\label{def:S1}
S_1:=\left\{(x,y)\in \R^2\left|\; (\varphi_z(\varphi_{z})_x(\varphi_{z})_y)(x,y)=0\right.\right\}
=\left\{(x,y)\in \R^2\;|\; xy(x^2-y^2)=0\right\}.
\end{gather}
For a domain $V$ such that $V\subset \R^2\setminus S_1$, the pair $\Phi$ leads to an analytic function $\varphi(x,y,z)$ on $U=V\times I$ uniquely as an evolution in $z$-direction under the initial conditions $\varphi(x,y,0)=\varphi(x,y)$ and $\varphi_z(x,y,0)=\varphi_z(x,y)$, and then $\varphi(x,y,z)$ defines a conformally flat metric $g$ on $V\times I$ in \eqref{def:g} (by replacing $V\times I$ with a sub-domain $V'\times I$ such that $\varphi_z\varphi_{zx}\varphi_{zy}\sin\varphi\cos\varphi\neq 0$ holds on $V'\times I$, if necessary)\footnote{We shall omit this remark from now on.}. Hence, there is a generic conformally flat hypersurface $f_V(x,y,z)$ on $V\times I$ with the Guichard net $g$ from $\Phi|_V$ theoretically. However, our aim is to study the explicit structure on curvature surfaces in $\R^4$. 

Precisely speaking, the existence of $\varphi(x,y,z)$ on $V\times I$ determining a Guichard net is equivalent to that of a pair $(\varphi,\psi)(x,y,z)$ of functions on $V\times I$ adding a certain function $\psi(x,y,z)$, and then $(\varphi,\psi)(x,y,z)$ satisfies a certain system of evolution equations in $z$ (cf.\ \cite[Theorem 1-(3) and (4)]{bhs}). When we find out $(\varphi,\psi)(x,y,z)$ as a solution of the system evolving from $z=0$, four analytic functions $(\varphi,\varphi_z,\psi,\psi_z)(x,y,0)$ on $V$ are required as the initial condition. These initial functions are determined from $\check g_H|_{V}$ (\cite[Theorem 7]{bhs}). However, we have omitted the explanation of $\psi$, since it does not appear in our argument later.

(R3)\;
As the solutions to a certain system of differential equations defined from $\Phi$, a $5$-dimensional set $\{(\bar P, \bar P_z,\bar\kappa_3)\}$ of triplets consisting of three functions on $\R^2$ is determined and each triplet $(\bar P, \bar P_z,\bar\kappa_3)$ leads to a curvature surface $f^0_V(x,y)$ on $V$ in some hypersurface $f_V(x,y,z)$. Here, we choose a triplet $\mathcal{P} = (\bar P(x,y), \bar P_z(x,y),\bar\kappa_3(x,y))$, which is our object, and explain how the curvature surface $f^0_V(x,y)$ is determined from $\mathcal{P}$.    
In order to describe the triplet $\mathcal{P}$ explicitly, we prepare a generalized hypergeometric function $X_0(x)$ on $\R$ of type ${}_1F_2$ (see Section \ref{subsec:choice} for the choice of $X_0(x)$): for a sequence $\{a_k\}_{k=1}^{\infty}$ given by $a_1=1$ and $2(k+1)(4k^2+5/4)a_{k+1}+(2k-1)a_k=0\ (k\geq1)$, $X_0(x)$ is defined by
\begin{gather}\label{def:X0}
X_0=X_0(x) :=5/2+\textstyle\sum_{k=1}^{\infty}a_kx^{2k}
=5/2+x^2-(1/21)x^4+\cdots.
\end{gather}
Note that $X'_0(x)/x$ is also an analytic function on $\R$
satisfying $(X'_0(x)/x)(0)=2$.
The $\mathcal{P}$ is determined from $X_0$ as follows:
\begin{equation}\label{expression:Pbar}
\begin{split}
\bar P^{-1} (x,y) &:= (1/\bar P) (x,y) = \frac{x}{x^2+y^2}h(x,y),\quad
\bar\kappa_3 (x,y) = -\frac{y}{x^2+y^2}\left(2(X_0-xX_0')+\sqrt{5}\right),\\
(\bar P^{-1})_z (x,y) &:= -\frac{\bar P_z}{\bar P^2} (x,y)=\frac{1}{(x^2+y^2)^2} \left((x^2-y^2)\left(X_0+\textstyle\frac{\sqrt{5}}{2}\right)+2xy^2X_0'\right)+\sqrt{5},
\end{split}
\end{equation}
where $h(x,y)=2X_0-xX_0'+y^2 X_0'/x+\sqrt{5}$ and $h(x,y)>\sqrt{5}$ holds on $\R^2$ (Corollary \ref{cor:ineq4h} in \S\ref{sec:metric}).
For $\mathcal{P}$, a rational Riemannian metric $g_0$ on $\R^2$, which is our object, is determined by
\begin{gather}\label{def:g0}
g_0:=\bar P^2\left(\cos^2\varphi(dx)^2+\sin^2\varphi(dy)^2\right)
=\frac{1}{h^2(x,y)}\left(\frac{(x^2-y^2)^2}{x^2}(dx)^2+4y^2(dy)^2\right).
\end{gather}
Then, $S_1$ in \eqref{def:S1} coincides with the singular set in $\R^2$ of $g_0$, and we have $g_0(x,y)=g_0(-x,y)=g_0(x,-y)$ on $\R^2\setminus S_1$ by the property of $h$ (Corollary \ref{cor:ineq4h}).

Next, we review the relation between the triplet $\mathcal{P}$ and the generic conformally flat hypersurfaces $f_V(x,y,z)$ in (R2).
For $\mathcal{P}$, we define two functions $\bar\kappa_1$ and $\bar\kappa_2$ on $\R^2\setminus S_1$ by
\begin{equation}\label{def:kappabar12}
\begin{split}
\bar\kappa_1(x,y)&:=\bar P^{-1}\tan \varphi+\bar\kappa_3=(2y/(x^2-y^2))\big(X_0+\textstyle\frac{\sqrt{5}}{2}\big),\\
\bar\kappa_2(x,y)&:=-\bar P^{-1}\cot\varphi+\bar\kappa_3=y^{-1}\left(((x^2+y^2)/2)(X_0'/x)-\big(X_0+\textstyle\frac{\sqrt{5}}{2}\big)\right).
\end{split}
\end{equation}
Then, $\bar\kappa_1\bar\kappa_2\neq 0$ holds on $\R^2\setminus S$, where $S$ is the set defined by 
\begin{gather}\label{def:S}
S:=S_1\cup S_2\ \ \text{and}\ \ S_2:=\big\{(x,y)\;\big|\; (x^2+y^2)X_0'=(2X_0+\sqrt{5})x\big\}.
\end{gather}
For a domain $V$ such that $V\subset \R^2\setminus S$, we have the following facts (R4) and (R5):

(R4)\;
For the pair $\Phi$ and the triplet $\mathcal{P}$, an analytic surface $\phi(x,y)$ in $\mathbb{S}^3$ with principal line coordinates is determined on $V$.
Let $F^0_{V}(x,y):=\left[\phi,X^0_{\alpha},X^0_{\beta},\xi\right](x,y)$ be the frame field of $\R^4$, where $X^0_{\alpha}$ and $X^0_{\beta}$ are the unit principal directions corresponding to the coordinates $(x,y)$ and $\xi$ is a unit normal vector field of $\phi$. Then, an analytic surface $f^0_V(x,y)$ in $\R^4$ with the metric $g_0$ of \eqref{def:g0} is determined on $V$ as a certain integral surface of $(X^0_{\alpha},X^0_{\beta})$ (see Theorem \ref{thm:F0df0} in the next sub-section \ref{subsec:results}). In brief, $\Phi$ and $\mathcal{P}$ determine the structure equation $dF^0_V=F^0_V\Omega$ of the surface $f^0_V(x,y)$ and its solution $F^0_V(x,y)$ on $V$ satisfying the integrability condition.   

(R5)\;
There is a generic conformally flat hypersurface $f_{V}(x,y,z)$ in (R2) defined on $V\times I$ such that $f_{V}(x,y,z)$ satisfies the following conditions (1) and (2):

(1)\;
$f_{V}(x,y,0)=f^0_{V}(x,y)$ holds on $V$.

(2)\;
For $f_{V}(x,y,z)$, the conformal element $P^2$ of $\mathrm{I}_{f_{_V}}$ in \eqref{def:I}\footnote{$P^{-1}(x,y,z):=(1/P)(x,y,z)$ satisfies the evolution equation \cite[(2.3.6)]{su4} in $z$-direction.\label{footnote:evoleq}} and the principal curvatures $\kappa_i$ satisfy the equations $P(x,y,0)=\bar P(x,y)$, $P_z(x,y,0)=\bar P_z(x,y)$ and $\kappa_i(x,y,0)=\bar\kappa_i(x,y)$.
Let $F_{V}(x,y,z):=\left[N,X_{\alpha}, X_{\beta},X_{\gamma}\right](x,y,z)$ be the orthonormal frame field determined by $f_{V}(x,y,z)$, where $N(x,y,z)$ is normal and $(X_{\alpha}, X_{\beta},X_{\gamma})(x,y,z)$ are the principal directions corresponding to the coordinates $(x,y,z)$.
Actually, $f_{V}(x,y,z)$ is determined as an evolution of surfaces in $z$ issuing from the surface $f^0_{V}(x,y)$ on $z=0$ under the condition $F_{V}(x,y,0)=F^0_{V}(x,y)$ on $V$, and then the condition $\bar\kappa_1\bar\kappa_2\neq 0$ is necessary (\cite[Theorem 3.1]{su4}).
Here, note that the other triplets $\{(\bar P, \bar P_z, \bar\kappa_3)\}$ determined by $\Phi$ determine conformal transformations of $f_{V}(x,y,z)$ arising from $\mathcal{P}$.

Now, by (R4) and (R5), $f^0_{V}(x,y)$ is an analytic curvature surface with the metric $g_0|_{V}$, and $f^0_{V}(x,y)$ and $F^0_{V}(x,y)$ are determined only by $\Phi$ and $\mathcal{P}$.
Furthermore, each coordinate line of $f^0_{V}(x,y)$ is a principal curvature line of $f_{V}(x,y,z)$. We emphasize that the curvature surfaces $f^0_{V}(x,y)$ are defined only on each domain $V\subset \R^2\setminus S$.
Our first aim is to study the existence and the structure of an extended curvature surface $f^0(x,y)$ including the singularity of $g_0$ and further to study the boundaries of the principal curvature lines of an extended surface  $f^0(x,y)$: for example,   
let us take the interval $(0,1]\times\{y\}\subset \R^2$ for each $y$, then the length of $x$-line $f^0(x,y)$ on the interval with respect to $g_0$ is finite if $y=0$, but it diverges to $\infty$ if $y\neq 0$; hence, it would be interesting to study the curvature surface as $x\rightarrow \pm 0$.

\subsection{The results}\label{subsec:results}

We summarize the results of each section. 
Let $D:=\{(x,y)\in \R^2 \;|\; x>0\}$. In Section \ref{sec:extension}, we study an extended curvature surface only on the singular Riemannian space $(D,g_0)$, since the metric $g_0$ on $\R^2$ of \eqref{def:g0} diverges on the line $x=0$ except for the origin. In Section \ref{sec:structure}, we shall study it on the whole $\R^2$.

Now, the function $X_0(x)$ in \eqref{def:X0} determines the properties of the space $(D,g_0)$ and the curvature surfaces. In Section \ref{sec:metric} we study the property of $X_0(x)$ for later use. Then, we also review how the triplet $\mathcal{P}$ of \eqref{expression:Pbar}, which we have chosen, is determined from $\Phi$ of \eqref{def:phi}, since it is necessary for the study of $X_0(x)$. 

In Section \ref{sec:extension}, we firstly verify that $\Phi$ and $\mathcal{P}$ give rise to an analytic curvature surface $f^0(x,y)$ on $D$ with the metric $g_0$. Here, we say that $f^0(x,y)$ is a curvature surface on $D$ if, for a domain $V$ such that $V\subset D\setminus S$, there is a generic conformally flat hypersurface $f_V(x,y,z)$ on $V\times I$ such that $f_{V}(x,y,0)=f^0(x,y)$ holds on $V$.
\begin{mainthm}\label{thm:F0df0}
\begin{enumerate}
\item
The frame field $F^0_V(x,y)$ on $V$ mentioned in (R4) extends analytically to a field $F^0(x,y):=\left[\phi,X^0_{\alpha},X^0_{\beta},\xi\right](x,y)$ on $D$.

\item
An analytic curvature surface $f^0(x,y)$ on $D$ is determined as an integral surface of $(X^0_{\alpha},X^0_{\beta})$ by $df^0= (xh(x,y))^{-1} ((x^2-y^2) X^0_{\alpha}\,dx+ 2xy X^0_{\beta}\,dy)$.
\end{enumerate}
\end{mainthm}
In Theorem \ref{thm:F0df0}, the field $F^0(x,y)$ is uniquely determined up to a transformation $QF^0(x,y)$ by a constant orthogonal matrix $Q$, and the fact (2) implies that the field $F^0(x,y)$ 
is a solution to the structure equation of $f^0(x,y)$ on $D$. 
\begin{mainthm}\label{thm:f0}
The curvature surface $f^0(x,y)$ on $D$ equipped with the frame field $F^0(x,y)=\left[\phi,X^0_{\alpha},X^0_{\beta},\xi\right](x,y)$, given in Theorem \ref{thm:F0df0}, satisfies the following conditions (1)--(4):
\begin{enumerate}
\item
The unit analytic vector ${\bs{u}}(y):=\left(1+y^2\right)^{-1/2}[yX^0_{\beta}-\phi](x,y)$ depends only on $y$.

\item
Any $x$-curve $f^0(x,y)$ with fixed $y$ lies on a standard $2$-sphere $\mathbb{S}^2_y$ of radius\\
$(2\sqrt{5})^{-1}\sqrt{(5+4y^2)/(1+y^2)}$ in an affine hyperplane perpendicular to ${\bs{u}}(y)$.

\item
The unit analytic vector $\tilde{\bs{u}}(x):=(B_2^2+C_2^2)^{-1/2}(x) (-B_2X^0_{\alpha}+C_2\xi)(x,y)$ depends only on $x$, where $B_2(x)=-(1/2)(X_0'(x)/x)-\sqrt{5}$ and $C_2(x)=X_0''(x)-X_0'(x)/x$.

\item
Any $y$-curve $f^0(x,y)$ with fixed $x$ lies on a standard $2$-sphere $\mathbb{S}^2_x$ of radius $(B_2^2+C_2^2)^{-1/2}(x)$ in an affine hyperplane perpendicular to $\tilde{\bs{u}}(x)$.
\end{enumerate}
\end{mainthm}

In Theorem \ref{thm:f0}, when we regard the surface $f^0(x,y)$ as a one-parameter family of $x$-curves, the surface $f^0(x,y)$ is expressed as
\begin{gather}\label{eq:f0asxcurves}
f^0(x,y)=(2\sqrt{5})^{-1}\sqrt{(5+4y^2)(1+y^2)^{-1}}\,{\bs{f}}(x,y)+A(y)
\end{gather}
with the unit analytic vector ${\bs{f}}(x,y):=((1+y^2)(5+4y^2))^{-1/2} (X^0_{\beta}-2(1+y^2)\xi+y\phi) (x,y)$.
Similarly, when we regard the surface $f^0(x,y)$ as a one-parameter family of $y$-curves, the surface $f^0(x,y)$ is expressed as
\begin{gather}\label{eq:f0asycurves}
f^0(x,y)=(B_2^2+C_2^2)^{-1/2}(x)\tilde{\bs{f}}(x,y)+{\tilde A}(x)
\end{gather}
with the unit analytic vector $\tilde{\bs{f}}(x,y):=(B_2^2+C_2^2)^{-1/2}(x)\left(B_2\xi+C_2X^0_{\alpha}\right)(x,y)$. Here, $A(y)$ (resp.\ ${\tilde A}(x)$) is a $\R^4$-valued analytic function of $y$ (resp.\ of $x$) determined uniquely up to a parallel translation, and they are the centers of $\mathbb{S}^2_y$ and $\mathbb{S}^2_x$, respectively. Furthermore, the curvature surface $f^0(x,y)$ on $D$ is bounded in $\R^4$. 

Next, in the singular set $S_1\cap D=\{(x,y)\;|\; y(x^2-y^2)=0\}$ of $(D,g_0)$, the surface $f^0(x,y)$ has the following features:

(A1)\;
Any $x$-curve $f^0(x,y)$ with fixed $y\neq 0$ has a cusp of type $(2,3,4)$ at $x=|y|$.

(A2)\;
Any $y$-curve $f^0(x,y)$ with fixed $x$ has a cusp of type $(2,3,4)$ at $y=0$.

(A3)\;
The curves $f^0(|y|,y)$ of $y$ are the envelopes of the family of $y$-curves $f^0(x,y)$ with fixed $x$.

(A4)\;
We translate the surface $f^0(x,y)$ such that $f^0(p_0)={\bf 0}$ holds at some point $p_0:=(x,0)$, where ${\bf 0}$ is the origin of $\R^4$. Then, there exists a reflection $B$ about a hyperplane $\R^3(\ni {\bf 0})$ such that $(B\circ f^0)(x,y)=f^0(x,-y)$, $(B\circ\phi)(x,y)=-\phi(x,-y)$ and $(B\circ\xi)(x,y)=\xi(x,-y)$ hold.

In Section \ref{sec:structure}, we study the limits of $x$-curves $f^0(x,y)$ as $x\rightarrow 0$ and $y$-curves $f^0(x,y)$ as $y\rightarrow\infty$. Then, we can clarify the structure of the surface $f^0(x,y)$ (and also the space $(D,g_0)$): in the case of $x\rightarrow 0$, we replace $x$ with $u$ given by $x=e^{-u}$ and study the case of $u\rightarrow \infty$.
For the vectors ${\bs{u}}(y)$ and $\tilde{\bs{u}}(x)$ in Theorem \ref{thm:f0}, we have the following theorem (see Definition \ref{lemdef:unifconv} for the uniform convergence in the theorem):

\begin{mainthm}\label{intro:thm:asymp-u-utilde}
There is an orthonormal frame $[{\bs{b}}^{\infty},{\bs{c}}^{\infty},\tilde{\bs{b}}{}^{\infty},\tilde{\bs{c}}^{\infty}]$ of $\R^4$ (consisting of constant vectors) such that it satisfies the following conditions:
\begin{enumerate}
\item
The vector ${\bs{u}}(y)$ moves on the unit circle in a plane spanned by $\tilde{\bs{b}}{}^{\infty}$ and $\tilde{\bs{c}}^{\infty}$. Let ${\bs{v}}_1(y)$ be a unit vector-valued function in $y$ determined by $(d{\bs{u}}/dy)(y)=(y^2/(1+y^2)){\bs{v}}_1(y)$.
Then, ${\bs{u}}(y)$ (resp.\ ${\bs{v}}_1(y)$) uniformly converges to the circle $\tilde{\bs{b}}(y):=\cos y\, \tilde{\bs{b}}{}^{\infty}+\sin y\,\tilde{\bs{c}}^{\infty}$ (resp.\ $\tilde{\bs{c}}(y):=-\sin y\, \tilde{\bs{b}}{}^{\infty}+\cos y\,\tilde{\bs{c}}^{\infty}$) as $y$ tends to $\infty$. In plain words, ${\bs{u}}(y)$ approaches asymptotically the constant velocity circular motion $\tilde{\bs{b}}(y)$ as $y$ tends to $\infty$. 

\item
The vector $\tilde{\bs{u}}(x)$ moves on the unit circle in a plane spanned by ${\bs{b}}^{\infty}$ and ${\bs{c}}^{\infty}$. Let $\tilde{\bs{v}}_1(x)$ be a unit vector of $x$ determined by $(d\tilde{\bs{u}}/du)(e^{-u})= -T(u)\tilde{\bs{v}}_1(e^{-u})$, where $T(u)>0$. Then, $\tilde{\bs{u}}(e^{-u})$ (resp.\ $\tilde{\bs{v}}_1(e^{-u})$) uniformly converges to the circle ${\bs{b}}(u):=\cos\frac{\sqrt{5}u}{2}{\bs{b}}^{\infty}-\sin\frac{\sqrt{5}u}{2}{\bs{c}}^{\infty}$ (resp.\ $-{\bs{v}}_2(u):=\sin\frac{\sqrt{5}u}{2}{\bs{b}}^{\infty}+\cos\frac{\sqrt{5}u}{2}{\bs{c}}^{\infty}$) as $u$ tends to $\infty$.
\end{enumerate}
\end{mainthm}

For the limits of $u$-curves $f^0(e^{-u}, y)$ as $u \to \infty$ and $y$-curves $f^0(x,y)$ as $y \to \infty$, we have the following facts (B1)--(B4):

(B1)\;
Any $u$-curve $f^0(e^{-u},y)$ with $y$ uniformly converges to a small circle of $\mathbb{S}^2_y$ as $u$ tends to $\infty$: only the $u$-curve $f^0(e^{-u},0)$ converges to a point of $\mathbb{S}^2_{y=0}$.

(B2)\;
All $u$-curve $X^0_{\alpha}(e^{-u},y)$ with $y$ uniformly converge to the circle ${\bs{b}}(u)$ as $u$ tends to $\infty$.

(B3)\;
Any $y$-curve $f^0(x,y)$ with $x$ uniformly converges to the point $(B_2^2+C_2^2)^{-1/2}(x)\tilde{\bs{v}}_1(x)+{\tilde A}(x)$ as $y$ tends to $\infty$.

(B4)\;
All $y$-curve $X^0_{\beta}(x,y)$ with $x$ uniformly converge to the circle $\tilde{\bs{b}}(y)$ as $y$ tends to $\infty$.

\noindent
In (B1) and (B2), the convergences for those $u$-curves are also uniform with respect to $y\in \R$. In (B3) and (B4), the convergences for those $y$-curves $f^0(x,y)$ are also uniform in the wider sense with respect to $x\in(0,\infty)$.

Furthermore, we have some other results:
(C1) As $x\rightarrow 0$ and $x\rightarrow\infty$, each $x$-curve $f^0(x,y)$ with $y\neq 0$ converges uniformly to the parallel small circles in $\mathbb{S}^2_y$. Then, the centers of these circles lie on the $t$-line $t{\bs v}_1(y)+A(y)$.
(C2) $\lim_{y\rightarrow\infty}f^0(x,y)=\lim_{y\rightarrow -\infty}f^0(x,y)$ holds for any $x\in (0,\infty)$.
(C3) There is a curvature surface $f^0(x,y)$ such that the curve $A(y)$ in \eqref{eq:f0asxcurves} (resp.\ $\tilde A(x)$ in \eqref{eq:f0asycurves}) lies on the plane $H$ spanned by $\tilde{\bs{b}}{}^{\infty}$ and $\tilde{\bs{c}}{}^{\infty}$ (resp.\ on the plane $\tilde{H}$ spanned by ${\bs{b}}^{\infty}$ and ${\bs{c}}^{\infty}$).

In consequence,
the curvature surface on $D$ given at (C3) has the following structure.
The centers of the circles in (B1) (resp.\ the convergence points in (B3)) lie on the plane $H$ (resp.\ $\tilde{H}$).
At the intersection of $\mathbb{S}^2_y$ and the $t$-line $t{\bs v}_1(y)+A(y)$, the two tangent spaces of $\mathbb{S}^2_y$ are spanned by ${\bs{b}}^{\infty}$ and ${\bs{c}}^{\infty}$.
Similarly, at the two points $\pm(B^2_2+C_2^2)^{-1/2}(x)\tilde{\bs{v}}_1(x)+\tilde{A}(x) \in \mathbb{S}^2_x$, the tangent spaces of $\mathbb{S}^2_x$ are spanned by $\tilde{\bs{b}}{}^{\infty}$ and $\tilde{\bs{c}}^{\infty}$. We could imagine clearly the structure of the curvature surface $f^0(x,y)$ if we would assume $A(y)=\tilde A(x)=\bs{0}$.

At the end in Section \ref{sec:structure}, we study the relationship between two curvature surfaces $f^0(x,y)$ on $D$ and $\hat f^0(x,y)$ on $D(-):=\{(x,y)\;|\;x<0\}$. 

In Section \ref{sec:approx}, for each positive integer $n$, we define an approximation $F^{\delta_n}(x,y)$ of the frame field $F^0(x,y)$ from the structure equation $dF^0=F^0\Omega$.
We construct $F^{\delta_n}(x,y)$ on a compact square $E\subset D$, by regarding $\phi$ (or $F^0(x,y)$) as a (singular) surface in the standard $3$-sphere $\mathbb{S}^3$.
Then, the Gauss and the Codazzi equations for $\phi$ are important in the construction.
Let $E:=[x_0,x_0+a]\times[y_0,y_0+a]$, and $x_0<x_1<\dots <x_n=x_0+a$ and $y_0<y_1<\dots <y_n=y_0+a$ be the divisions of $[x_0,x_0+a]$ and $[y_0,y_0+a]$ of equal length $\delta_n:=a/n$.
Then, for an initial orthogonal matrix at $(x_0,y_0)$, an orthonormal frame field $F^{\delta_n}(x,y)$, independently of the width $\delta_n$, is determined on the lattice in $E$ made from the divisions: precisely, $F^{\delta_n}(x,y)$ is defined on each path (see Definition \ref{def:path}).
The $F^{\delta_n}(x,y)$ is constructed by a kind of polygonal line method: on each edge $[x_i,x_{i+1}]\times \{y_j\}$ or $\{x_i\}\times[y_j,y_{j+1}]$, we approximate $F^0(x,y)$ by a rational curve (not by a line).
We further define $F^{\delta_n}(x,y)$ at all points $(x,y)\in E$ by a little change of the divisions.
Then, $F^{\delta_n}(x,y)$ converges to $F^0(x,y)$ uniformly on $E$ as $n\rightarrow \infty$. By using $F^{\delta_n}(x,y)$, we can draw the curves: coordinate lines, the cusps and the enveloping curves and so on, which are given in Sections \ref{sec:extension} and \ref{sec:structure}, since each coordinate curve is expressed by the frame field $F^0(x,y)$ as in \eqref{eq:f0asxcurves} and \eqref{eq:f0asycurves}. Then, similar to the case of coordinate lines in $f^0(x,y)$, those approximations also lies on $2$-spheres $\mathbb{S}^2$.
 
\section{Choice of a singular metric $g_0$ determining our curvature surfaces}\label{sec:metric}

As mentioned in the introduction, for the Poincar\'{e} metric $\check g_H=(dx^2+dy^2)/y^2$ on $\R^2$, we have taken at \eqref{def:phi} the following pair
$\Phi=(\varphi,\varphi_z)$ of functions on $\R^2$,
\begin{gather*}
\cos\varphi(x,y)=\frac{x^2-y^2}{x^2+y^2},\quad
\sin\varphi(x,y)=\frac{2xy}{x^2+y^2},\quad
\varphi_z(x,y)=\frac{y}{x^2+y^2},
\end{gather*}
and chosen at \eqref{expression:Pbar} the following triplet $\mathcal{P}=(\bar P, \bar P_z, \bar\kappa_3)$ of three functions on $\R^2$,
\begin{gather*}
\begin{split}
\bar P^{-1}&=\displaystyle\frac{x}{x^2+y^2}h(x,y),\quad
\bar\kappa_3=-\frac{y}{x^2+y^2}\left(2(X_0-xX_0')+\sqrt{5}\right),\\
(\bar P^{-1})_z&:=\displaystyle-\frac{\bar P_z}{\bar P^2}=\frac{1}{(x^2+y^2)^2}\left((x^2-y^2)\left(X_0+\textstyle\frac{\sqrt{5}}{2}\right)+2xy^2X_0'\right)+\sqrt{5}.
\end{split}
\end{gather*}
Here $X_0=X_0(x)$ is the hypergeometric function on $\R$ given in \eqref{def:X0} and $h(x,y)$ is the function defined by
\begin{gather}\label{def:h}
h(x,y)=2X_0-xX_0'+y^2 (X_0'/x)+\sqrt{5}.
\end{gather}
The triplet $\mathcal{P}$ of \eqref{expression:Pbar} is selected from the $5$-dimensional set of triplets $\{(\bar P, \bar P_z,\bar\kappa_3)\}$ in \cite[Example 3.2]{su4} determined by $\Phi$ and it leads to a singular metric $g_0$ on $\R^2$ of \eqref{def:g0} determining curvature surfaces. In this section, we firstly explain that the choice of the triplet $\mathcal{P}$ is suitable, and next study the property of $X_0(x)$ for the argument later.

\subsection{Choice of the triplet $\mathcal{P}$ and the singular metric $g_0$}\label{subsec:choice}

Each triplet $(\bar P, \bar P_z,\bar\kappa_3)$ is determined from a pair $(X_1(x),Y(y))$ of solutions to the two differential equations \eqref{def:X1Y} below, of which fact we explain in this sub-section. 
Now, we define two differential equations for $1$-variable functions $X_1=X_1(x)$ of $x$ and $Y=Y(y)$ of $y$ by
\begin{gather}\label{def:X1Y}
xX_1'''-X_1''+\left(x+9/(4x)\right)X_1'-X_1=cx^2,\quad
Y''+Y=cy^2,
\end{gather}
respectively, where $c$ is a constant.
Let us take $c=0$ at the equation for $X_1$, then the point $x=0$ of the equation is a regular singularity and its fundamental solutions are given by the power series $x^r\left[1+\sum_{k=1}^{\infty}\alpha_kx^k\right]$ with $r=0, \ 2\pm \sqrt{-5}/2$. The function $X_0(x)$ of \eqref{def:X0} is a solution corresponding to $r=0$ and it is analytic on $\R$. 
The two solutions $q_i(x)$ $(i=1,2)$ other than $X_0(x)$ are the following complex-valued functions: let $r_1:=2+\sqrt{-5}/2$ and $r_2:=2-\sqrt{-5}/2$, then $q_i(x)$ are given by
\begin{gather*}
q_1(x)=x^2\left[\cos((\sqrt{5}/2)\log |x|)+\sqrt{-1}\sin((\sqrt{5}/2)\log |x|)\right]\left[1+\textstyle\sum_{k=1}^{\infty}\beta_{1,k}x^{2k}\right],\\
q_2(x)=\bar q_1(x)=x^2\left[\cos((\sqrt{5}/2)\log |x|)-\sqrt{-1}\sin((\sqrt{5}/2)\log |x|)\right]\left[1+\textstyle\sum_{k=1}^{\infty}\beta_{2,k}x^{2k}\right],
\end{gather*}
where $\beta_{i,1}=-(r_i-1)/[(r_i+2)(r_i^2+5/4)]$ and $\beta_{i,k+1}=-(r_i+2k-1)/[(r_i+2k+2)((r_i+2k)^2+5/4)]\ \beta_{i,k}$ for $k\geq 1$. The $X_0(x)$ and a linearly independent pair of real-valued functions determined from $q_i(x)$ are a base of real-valued solutions to the equation. 
Next, the general solution to the equation for $Y$ involves the linear combination of $\cos y$ and $\sin y$ as terms. In Proposition \ref{prop:class} below, we shall take a particular solution $Y$ that does not include such terms. That is, we choose the simplest solutions of the equations for $X_1$ and $Y$ to get the triplet $\mathcal{P}$.

Now, for the solutions $X_1$ and $Y$ to the equations \eqref{def:X1Y}, we define the functions $G^{X_1}(x)$ and $H^{Y}(y)$ by
\begin{align*}
G^{X_1}(x) &:=(X_1'')^2+4cX_1+\left(1+9/(4x^2)\right)(X_1')^2+\left((2/x)X_2-4cx\right)X_1',\\
H^{Y}(y) &:=(Y'-2cy)^2+(Y-cy^2+2c)^2-4c^2,
\end{align*}
where $X_2:=-X_1''-X_1+cx^2$. Then, both functions $G^{X_1}(x)$ and $H^{Y}(y)$ are constant as mentioned in Proposition \ref{prop:class} below. 
Furthermore, for a pair $(X_1,Y)$ of solutions, let $\bar X=\bar X(x)$, $\bar Y=\bar Y(y)$ and $\mathcal{A}=\mathcal{A}(x,y)$ be the functions defined by
\begin{gather*}
\bar X:=xX_1'-X_1,\quad
\bar Y:=yY'-Y,\quad
\mathcal{A}:=\bar X+\bar Y.
\end{gather*}
 
The following proposition \ref{prop:class} was verified in \cite[Example 3.2]{su4} except for the equations for $G^{X_1}(x)$ and $H^{Y}(y)$ in \eqref{item:GX0;prop:class} and \eqref{item:GX1;prop:class}.
Hence, we only prove these equations.
In only the proposition \ref{prop:class}, for the sake of simplicity for description, we say that a triplet $(\bar P,\bar P_z,\bar\kappa_3)$ is a class for $\Phi$, if curvature surfaces $f^0_{V}(x,y)$ and generic conformally flat hypersurfaces $f_{V}(x,y,z)$ in $\R^4$ are determined by $\Phi$ and the triplet $(\bar P,\bar P_z,\bar\kappa_3)$.
\begin{prop}\label{prop:class}
Let $(X_1,Y)$ be a pair of solutions to the equations \eqref{def:X1Y}. Then, we have the following facts (1) and (2):
\begin{enumerate}
\item\label{item:GH-const;prop:class}
The functions $G^{X_1}(x)$ and $H^{Y}(y)$, respectively, are constant.

\item\label{item:class;prop:class}
For any pair $(X_1,Y)$ such that $G^{X_1}(x)+H^{Y}(y)=0$, a class $(\bar P,\bar P_z,\bar\kappa_3)$ for $\Phi$ is determined as follows:
\begin{gather*}
\bar P^{-1}=X_1'-\frac{2x}{x^2+y^2}\mathcal{A},\quad
\bar\kappa_3=-Y'+\frac{2y}{x^2+y^2}\mathcal{A},\\
(\bar P^{-1})_z=-\frac{\bar P_z}{\bar P^2}=\frac{1}{x^2+y^2}\left(xX_1'-\mathcal{A}+\frac{2y^2}{x^2+y^2}\mathcal{A}\right).
\end{gather*}
\end{enumerate}
Conversely, all classes $(\bar P, \bar P_z, \bar\kappa_3)$ for $\Phi$ are determined by the above forms from the pairs $(X_1,Y)$ such that $G^{X_1}(x)+H^{Y}(y)=0$. The set of pairs $(X_1,Y)$ satisfying $G^{X_1}(x)+H^{Y}(y)=0$ is $5$-dimensional.

Furthermore, we have the following facts (3) and (4):
\begin{enumerate}
\addtocounter{enumi}{2}
\item\label{item:GX0;prop:class}
$X_0(x)$ in \eqref{def:X0} is a solution to the first equation in \eqref{def:X1Y} with $c=0$, and $G^{X_0}(x)=-5$ holds.

\item\label{item:GX1;prop:class}
Let us take $c=\sqrt{5}$ in two equations of \eqref{def:X1Y}.
Then $X_1(x)=X_0(x)+\sqrt{5}\left(x^2+5/2\right)$ and $Y(y)=\sqrt{5}\left(y^2-2\right)$, respectively, are solutions to the equations, and the pair satisfies $G^{X_1}(x)+H^{Y}(y)=0$.
In particular, we have $G^{X_1}(x)=20$.
\end{enumerate}
\end{prop}
\begin{proof}
Firstly, note that $X_1(x)=X_0(x)+c(x^2+5/2)$ and $Y(x)=c(y^2-2)$ are the solutions of the first and the second equations of \eqref{def:X1Y}, respectively.
Now, we verify that $G^{X_1}(x)=-5+5c^2$ holds. From
\begin{gather*}
(X_1''(0))^2=4(1+c)^2,\quad
4cX_1(0)=10c(1+c),\quad
\left((1+9/(4x^2))(X_1')^2\right)(0)=9(1+c)^2,\\
\left((2X_2/x-4cx)X_1'\right)(0)
=\left((2X_2-4cx^2)(X_1'/x)\right)(0)=-18(1+c)^2
\end{gather*}
by \eqref{def:X0}, we have $G^{X_1}(0)=-5+5c^2$.
Then, $G^{X_1}(x)=-5+5c^2$ holds for any $x$, since the function $G^{X_1}(x)$ is constant.
In the same way, we have $H^{Y}(y)=-4c^2$ for $Y(y)=c(y^2-2)$.

In consequence, we also have verified that the pair $X_1(x)=X_0(x)+\sqrt{5}(x^2+5/2)$ and $Y=\sqrt{5}(y^2-2)$ is a solution to \eqref{def:X1Y} and satisfies $G^{X_1}(x)+H^{Y}(y)=0$.
\end{proof}

Let $X_1(x)=X_0(x)+\sqrt{5}\left(x^2+5/2\right)$ and $Y(y)=\sqrt{5}\left(y^2-2\right)$ be a pair of solutions to \eqref{def:X1Y}, given in Proposition \ref{prop:class}-(\ref{item:GX1;prop:class}). Note that $X_0'(x)/x$ is also an analytic function on $\R$ satisfying $(X_0'(x)/x)(0)=2$. 
From the pair, the triplet $\mathcal{P}=(\bar P(x,y), \bar P_z(x,y),\bar\kappa_3(x,y))$ of \eqref{expression:Pbar} is determined by Proposition \ref{prop:class}-(\ref{item:class;prop:class}), and then for the triplet $\mathcal{P}$, the singular metric $g_0$ on $\R^2$ and the functions $\bar\kappa_i \ (i=1,2)$ on $\R^2$ are also determined as the forms in \eqref{def:g0} and \eqref{def:kappabar12}, respectively:
\begin{gather*}
g_0
=h^{-2}(x,y)\left(x^{-2}(x^2-y^2)^2(dx)^2+4y^2(dy)^2\right),\\
\begin{split}
\bar\kappa_1=2y(x^2-y^2)^{-1}\big(X_0+\sqrt{5}/2\big), \ \ 
\bar\kappa_2=y^{-1}\big[((x^2+y^2)/2)
(X_0'/x)-\big(X_0+\sqrt{5}/2\big)\big].
\end{split}
\end{gather*}

\subsection{Property of the function $X_0(x)$}

Since the function $X_0(x)$ of \eqref{def:X0} determines the properties of the metric $g_0$ on $\R^2$ and the functions $\bar\kappa_i$ on $\R^2$, we study the property of $X_0(x)$.
\begin{prop}\label{prop:X0}
The function $X_0(x)$ on $\R$ of \eqref{def:X0} satisfies the following conditions:
\begin{enumerate}
\item
$X_0(-x)=X_0(x)\geq 5/2$.

\item
$X_0'(x)/x>0$ for $x\in \R$ and $(X_0'(x)/x)(0)=2$.

\item
$( 2X_0-xX_0')(x)> 0$ for $x\in \R$.
\end{enumerate}
\end{prop}
\begin{proof}
$X_0(-x)=X_0(x)$ follows from the definition \eqref{def:X0} of $X_0$.
For $X_0(x)\geq 5/2$ and (2), suppose that there is a point $x_0>0$ such that $X_0'(x_0)=0$.
Then, we have $G^{X_0}(x_0)=(X_0'')^2(x_0)=-5$, which is a contradiction.
Moreover, we have $(X_0'/x)(0)=2$ by \eqref{def:X0}.
Hence, we have $X_0'(x)>0$ for $x>0$ and $(X_0'/x)(x)>0$ for $x\in \R$.
Furthermore, since $X_0'(x)>0$ for $x>0$, $X_0(x)$ is an increasing function on $[0,\infty)$ and $X_0(0)=5/2$.
Hence, we have $X_0(x)\geq 5/2$ for $x\in \R$.
For (3), we have
\begin{gather*}
G^{X_0}(x)=(X_0''-X_0'/x)^2+(1+5/(4x^2))(X_0')^2-2X_0(X_0'/x)=-5.
\end{gather*}
Hence, we have $(X_0')^2-2X_0(X_0'/x)=(X_0'/x)(xX_0'-2X_0)<-5$.
By $X_0'/x> 0$ for $x\in \R$, we have $2X_0-xX_0'> 0$ for $x\in \R$.
\end{proof}
\begin{cor}\label{cor:ineq4h}
For the function $h(x,y)$ of \eqref{def:h}, we have the following facts. 
\begin{enumerate}
\item
$h(x,y)>\sqrt{5}$ on $\R^2$.

\item
There is a number $t_1 \ (0<t_1<1)$ such that $5+\sqrt{5}+y^2<h(x,y)<6+\sqrt{5}+2y^2$ and $|h(x,y)-h(0,y)|\leq x^2(1+2y^2)/7$ hold for $0\leq x<t_1$ and any $y\in \R$.
\end{enumerate}
\end{cor}
\begin{proof}
The fact (1) follows from the definition \eqref{def:h} of $h(x,y)$ and Proposition \ref{prop:X0}.
The fact (2) follows from the definition \eqref{def:X0} of $X_0(x)$. In fact, as $x$ tends to $0$, for any $y\in \R$ we have
\begin{gather*}
2X_0-xX_0'=5+(2/21)x^4+O(x^6),\quad
X_0'/x=2-(4/21)x^2+O(x^4),\\
|h(x,y)-h(0,y)|\leq |2X_0-xX_0'-5|+y^2|X_0'/x-2|<(x^2/7)(1+2y^2),
\end{gather*}
where $l(x)=O(x^k)$ for a function $l(x)$ implies that $c_1<\lim_{x\rightarrow 0}(l(x)/x^k)<c_2$ holds for some constants $c_i$.
\end{proof}

The function $h(x,y)$ also satisfies the following equations:
\begin{gather*}
h(x,y)=h(-x,y)=h(x,-y),\quad
h(0,y)=5+\sqrt{5}+2y^2,\\
h(x,0)=(2X_0-xX_0')+\sqrt{5},\quad
h(0,0)=5+\sqrt{5}.
\end{gather*}

Next, we study the property of $X_0(x)$ as $x$ tends to $\infty$. In the following proposition \ref{prop:ineq4X0}-(5), we show that there is a number $t_2>0$ such that $X_0'(x)$ is a bounded oscillating function on $[t_2,\infty]$, like the function $\sin x$. Here, we say that $X_0'(x)$ oscillates at $x=\infty$, if there is a bounded interval $J$ (not one point) satisfying the following condition: for any point $p\in J$, there is a sequence $x_n$ ($x_n\rightarrow\infty$) such that $X_0'(x_n)$ converges to $p$. In Proposition \ref{prop:ineq4X0}, the fact (2) is included in (5) in terms of content, but is described as a proof procedure.

\begin{prop}\label{prop:ineq4X0}
The function $X_0(x)$ on $\R$ satisfies the following conditions (1)--(5):
\begin{enumerate}
\item
The function $\tau(x):=(X_0+X_0'')(x)/x$ is decreasing in $(0,\infty)$ and $\tau(x)$ converges to a non-negative constant $\tau(\infty):=\lim_{x\rightarrow \infty}\tau(x)$ as $x$ tends to $\infty$. Then, we have $\tau(\infty)>\sqrt{5}$, which implies that $X_0'(x)$ oscillates at $x=\infty$.\footnote{Actually, we can show $\tau(\infty)=\sqrt{5}\coth (\sqrt{5}\pi/4)\approx\sqrt{5}\times 1.0614\; (>\sqrt{5})$, by using the asymptotic expansion formula for large $x$ of generalized hypergeometric functions of type ${}_1F_2$ (cf.\ \cite{lw}, \cite{lu}).}

\item
We have
\begin{gather*}
\tau(\infty)-\sqrt{\tau^2(\infty)-5}\leq \lim_{x\rightarrow\infty}X_0'(x)\leq \tau(\infty)+\sqrt{\tau^2(\infty)-5},
\end{gather*}
where $ \lim_{x\rightarrow\infty}X_0'(x)$ implies the oscillation of $X_0'(x)$ at $x=\infty$.

\item
For $x\geq 0$, we have $(X_0(x)-\tau(x)x)^2=(X_0''(x))^2\leq \tau^2(x)-5$.

\item
We have
\begin{gather*}
\tau(\infty)-\sqrt{\tau^2(\infty)-5}\leq \lim_{x\rightarrow\infty}
((2X_0-xX_0')/x)\leq \tau(\infty)+\sqrt{\tau^2(\infty)-5},
\end{gather*}
where $ \lim_{x\rightarrow\infty}\left((2X_0-xX_0')/x\right)$ implies the oscillation at $x=\infty$.

\item
For $m(x):=\sqrt{1+9/(4x^2)}$, we have $\lim_{x\rightarrow\infty}(\tau/m)(x)=\tau(\infty)>\sqrt{5}$ by (1). Hence, there is a number $t_2>0$ such that $(\tau/m)^2(x)>5$ on $[t_2,\infty]$. Then, there is a monotone (increasing or decreasing) function $\eta(x)$ on $[t_2,\infty)$ such that $\eta(x)$ diverges to $\infty$ or $-\infty$ as $x$ tends to $\infty$ and   
\begin{gather*}
X_0''(x)=\sqrt{(\tau/m)^2(x)-5}\ \cos\eta(x), \ \ X_0'(x)=(1/m)(x)\left[(\tau/m)(x)+\sqrt{(\tau/m)^2(x)-5}\ \sin\eta(x) \right]
\end{gather*}
hold on $[t_2,\infty)$. That is, $X_0'(x)$ and $X_0''(x)$ are bounded oscillating functions on $[t_2,\infty]$.
\end{enumerate}
\end{prop}
\begin{proof}
The function $X_0(x)$ is a solution to the first equation in \eqref{def:X1Y} with $c=0$. Hence, $X_0(x)$ satisfies the following equations:
\begin{gather}
\label{diffeq:X0-3rd}
xX_0'''+\left(x+9/(4x)\right)X_0'-(X_0''+X_0)=0,\\
\label{diffeq:X0-2nd}
G^{X_0}(x)=(X_0'')^2+(1+9/(4x^2))(X_0')^2-2\tau(x)X_0'=-5.
\end{gather}
Now, for (1) and (2), we have $\tau'(x)=-9/(4x^3)X_0'$ by \eqref{diffeq:X0-3rd}. Hence, we have $\tau'(x)<0$ for $x> 0$ by $X'(x)>0$ in Proposition \ref{prop:X0}. Next, we have $\tau(x)X_0'>5/2$ by \eqref{diffeq:X0-2nd}. Hence, there is the limit $\tau(\infty):=\lim_{x\rightarrow \infty}\tau(x)$ by $\tau(x)>0$ for $x\geq 0$: $\tau(\infty)$ is non-negative.
Next, we have $(X_0')^2-2\tau X_0'+5<0$ by \eqref{diffeq:X0-2nd}. Hence, we have
\begin{gather}
\tau(x)\geq\sqrt{5},\quad
\tau(x)-\sqrt{\tau^2(x)-5}
<X_0'(x)<\label{ineq:X0'}
\tau(x)+\sqrt{\tau^2(x)-5},
\end{gather}
since $X_0'(x)$ is a real-valued function and $\tau(x)>0$. Furthermore, $\tau(x)-\sqrt{\tau^2(x)-5}$ (resp.\ $\tau(x)+\sqrt{\tau^2(x)-5}$) is an increasing function (resp.\ a decreasing function). Thus, we have obtained (1) $\tau(\infty)\geq\sqrt{5}$ and (2). We shall verify $\tau(\infty)\neq \sqrt{5}$ after the proofs of (3) and (4). Here, note that $\tau(\infty)=\sqrt{5}$ is equivalent to $\lim_{x\rightarrow \infty}X_0'(x)= \sqrt{5}$ and $\lim_{x\rightarrow\infty}X_0''(x)=0$, respectively, by the argument above. Hence, $\tau(\infty)> \sqrt{5}$ implies that both functions $X_0'(x)$ and $X_0''(x)$ oscillate at $x=\infty$.

For (3), by $X_0''=\tau(x)x-X_0$ and \eqref{diffeq:X0-2nd}, we have
\begin{align*}
(X_0-\tau(x)x)^2&=(X_0'')^2=-(1+9/(4x^2))(X_0')^2+2\tau(x)X_0'-5\\
&<-(X_0')^2+2\tau(x)X_0'-5=-(X_0'-\tau(x))^2+\tau^2(x)-5\leq\tau^2(x)-5.
\end{align*}

For (4), we have $\lim_{x\rightarrow \infty}\left((2X_0-xX_0')/x\right)=\lim_{x\rightarrow \infty}(2\tau(x)-X_0'(x))$ by (3). Hence, we obtain (4) by the existence of $\tau(\infty)$ and (2).

Now, we give an elementary proof of $\tau(\infty)> \sqrt{5}$.
Firstly, we have
\begin{align*}
\big[\left(1+\textstyle\frac{9}{4x^2}\right)^{-1/2}\tau(x)\big]'
=\textstyle\frac{9}{4x^3}\left(1+\textstyle\frac{9}{4x^2}\right)^{-3/2}\big[\tau(x)-\left(1+\frac{9}{4x^2}\right)X_0'\big]
=\frac{9}{4x^3}\left(1+\frac{9}{4x^2}\right)^{-3/2}X_0'''
\end{align*}
by \eqref{diffeq:X0-3rd} and $\tau'(x)=-9/(4x^3)X_0'$.
The equation shows that $\left(1+9/(4x^2)\right)^{-1/2}\tau(x)$ is almost equal to $\tau(\infty)$ for $x\geq 10$, because $X_0'''(x)$ is an oscillating function taking small values around $0$ and $\lim_{x\rightarrow\infty}\left((1+9/(4x^2))^{-1/2}\tau(x)\right)=\tau(\infty)$ holds. We shall precisely verify it below.

We integrate both sides of the equation on the interval $[x,C]$, where $C>x>0$: for the right side, we use the formula of integration by parts. 
Then, we have $(X_0''(x))^2\leq 4\tau^2(x)$ by \eqref{diffeq:X0-2nd}, \eqref{ineq:X0'} and the fact that $\tau(x)$ is decreasing. Hence, as $C \to \infty$, we obtain     
\begin{gather*}
\left|\tau(\infty)-(1+9/(4x^2))^{-1/2}\tau(x)\right|
\leq (9/x^3)\tau(x),
\end{gather*}
for any $x\in (0,\infty)$. That is, for any $x\in (0,\infty)$, the following inequalities are satisfied
\begin{gather*}
\left((1+9/(4x^2))^{-1/2}-9/x^3\right) \tau(x)
\leq \tau(\infty) \leq \left((1+9/(4x^2))^{-1/2}+9/x^3\right) \tau(x).
\end{gather*} 

Now, if there is an $x_0\in (0,\infty)$ such that $\left((1+9/(4x_0^2))^{-1/2}-(9/x_0^3)\right)\tau(x_0)>\sqrt{5}$, then we have $\tau(\infty)>\sqrt{5}.$ Actually, for $x_0=10$, the desired inequality holds. We can make sure the fact as follows. The function $\tau(x)$ is expressed as the following alternating power series:
\begin{gather*}
\tau(x)=(9/2)x^{-1}+(9/4)\textstyle\sum_{k=1}^{\infty}b_kx^{2k-1},\quad
b_k=a_k/(4k^2+5/4),
\end{gather*}
where $a_k$ is the coefficients of $X_0(x)$ given in \eqref{def:X0}. At $x= 10$, the sequence $|b_k|x^{2k-1}$ is strongly decreasing for $k\geq 4$, and $|(9/4)b_{k}10^{2k-1}|\leq (3/2)\cdot10^{-12}$ holds for $k=20$. Hence, we have $|\tau(x)-\left((9/2)x^{-1}+(9/4)\sum_{k=1}^{20}b_kx^{2k-1}\right)|\leq (3/2)\cdot10^{-12}$ at $x=10$. Furthermore, at $x=10$ we have the inequality
\begin{gather*}
\left((1+9/(4x^2))^{-1/2}-9/x^3\right)
\left((9/2)x^{-1}+(9/4)\textstyle\sum_{k=1}^{20}b_kx^{2k-1}\right)
\geq 2.35 \;(>\sqrt{5}),
\end{gather*}
which shows $\tau(\infty)> \sqrt{5}$.

For (5), let $m(x)=(1+9/(4x^2))^{1/2}$. Then, there is a number $t_2>0$ such that $(\tau/m)^2(x)>5$ holds on $[t_2,\infty]$ by $\lim_{x\rightarrow\infty}\tau^2(x)>5$. Using $m(x)$, we can rewrite the equation \eqref{diffeq:X0-2nd} as follows:
\begin{gather*}
(X_0'')^2+m^2(x)\left[X_0'-\tau(x)/m^2(x)\right]^2=(\tau/m)^2(x)-5.
\end{gather*}
Hence, there is a function $\eta(x)$ on $[t_2,\infty)$ satisfying (5) for $X_0'(x)$ and $X_0''(x)$. Furthermore, when $X_0''(x)$ vanishes at $x=x_1\in[t_2,\infty)$, we have 
\begin{gather*}
X_0'''(x_1)=\pm\sqrt{\tau^2(x_1)-5m^2(x_1)}\neq 0
\end{gather*}
by \eqref{diffeq:X0-3rd} and \eqref{diffeq:X0-2nd}, which shows that $X_0''(x)$ and $X_0'(x)$ move monotonically even around $x=x_1$, without folding. Hence, $\eta(x)$ is a monotone function. Furthermore, since $X_0'(x)$ and $X_0''(x)$ oscillate at $x=\infty$ by $\tau(\infty)>\sqrt{5}$, $\eta(x)$ diverges to $\infty$ or $-\infty$ as $x\rightarrow\infty$.

In consequence, the proof of the proposition has been completed. 
\end{proof}

Proposition \ref{prop:ineq4X0}-(4) implies that the first term of $h(x,y)=(2X_0-xX_0')+y^2(X_0'/x)+\sqrt{5}$ satisfies $2X_0-xX_0'=O(x)$ as $x$ tends to $\infty$, and that $(2X_0-xX_0')/x$ oscillates even at $x=\infty$.

In the next section, we study the property of the singular metric $g_0$ on $\R^2$ of \eqref{def:g0}. The singular set in $\R^2$ of $g_0$ is given by $S_1 = \left\{(x,y)\in \R^2\;|\; xy(x^2-y^2)=0\right\}$:
$g_0$ diverges on the line $x=0$ except for the origin $(0,0)$ and degenerates on the lines $y=0$ and $x^2-y^2=0$;
$\lim_{x\rightarrow \pm 0} g_0(x,0)$ totally degenerates.

\section{Extended frame field and analytic curvature surface defined on $D$}\label{sec:extension}

Let $g_0$ be the singular metric on $\R^2$ defined at \eqref{def:g0}:
\begin{gather*}
g_0=\frac{1}{h^2(x,y)}\left(\frac{(x^2-y^2)^2}{x^2}(dx)^2+4y^2(dy)^2\right), 
\end{gather*}
where $h(x,y)=2X_0-xX_0'+y^2(X_0'/x)+\sqrt{5}$ of \eqref{def:h}. 
The metric has been defined from $\Phi=(\varphi,\varphi_z)(x,y)$ of \eqref{def:phi} and $\mathcal{P}=(\bar P,\bar P_z,\bar\kappa_3)(x,y)$ of \eqref{expression:Pbar}. In this section, we consider the metric $g_0$ only on $D=\{(x,y)|\ x>0\}$, because $g_0$ on $\R^2$ diverges on the line $x=0$. Since $h(x,y)>\sqrt{5}$ holds on $\R^2$ by Corollary \ref{cor:ineq4h},
the degenerate set of the Riemannian metric space $(D,g_0)$ is given by $D\cap S_1=\{(x,y)\;|\; y(x^2-y^2)=0\}$. 
We shall verify in this section that a certain orthonormal frame field of $\R^4$ is determined on the whole space $D$ from $\Phi$ and $\mathcal{P}$, and that the frame field leads to an analytic curvature surface in $\R^4$ as a realization of the space $(D,g_0)$.

Now, as mentioned in (R4) and (R5) of the introduction, for a simply connected open set $V$ with $V\subset D\setminus S$ for $S$ in \eqref{def:S}, a curvature surface $f^0_V(x,y)$ on $V$ and a generic conformally flat hypersurface $f_V(x,y,z)$ on $V\times I$ are determined from $\Phi$ and $\mathcal{P}$ such that $f^0_V(x,y)=f_V(x,y,0)$ holds on $V$.
Then, the coordinates $(x,y,z)$ of $f_V(x,y,z)$ is a principal curvature line coordinate system.
Let $F_V(x,y,z)=\left[N,X_{\alpha},X_{\beta},X_{\gamma}\right](x,y,z)$
be the orthonormal frame field on $f_V(x,y,z)$, where $N$ is a unit normal vector field of $f_V(x,y,z)$ and $(X_{\alpha},X_{\beta},X_{\gamma})(x,y,z)$ are the unit principal curvature vectors corresponding to $x$, $y$, and $z$-curves: let $\varphi(x,y,z)$ be the function determined from $\Phi$ by $\varphi(x,y,0)=\varphi(x,y)$ and $\varphi_z(x,y,0)=\varphi_z(x,y)$, and $P(x,y,z)$ be the function determined from $\mathcal{P}$ by $P(x,y,0)=\bar P(x,y)$ and $P_z(x,y,0)=\bar P_z(x,y)$\footnote{For references on the evolution equations for $\varphi(x,y,z)$ and $P(x,y,z)$, see (R2) and footnote \ref{footnote:evoleq}, respectively, in \S\ref{subsec:existence}.}, then the frame field $(X_{\alpha},X_{\beta},X_{\gamma})(x,y,z)$ is given from \eqref{def:I} by
\begin{gather*}
X_{\alpha}:=(P\cos\varphi)^{-1}\partial f_V/\partial x,\quad
X_{\beta}:=(P\sin\varphi)^{-1}\partial f_V/\partial y,\quad
X_{\gamma}:=P^{-1}\partial f_V/\partial z.
\end{gather*}
Hence, the differential $df_V(x,y,z)$ of $f_V(x,y,z)$ is expressed as
\begin{gather*}
df_V(x,y,z)=P(x,y,z)\left((\cos\varphi\,dx) X_{\alpha}+(\sin\varphi\,dy)X_{\beta}+dzX_{\gamma}\right)(x,y,z).
\end{gather*}
For the frame $F_V(x,y,z)$, we put $\phi(x,y):=N(x,y,0)$, $X^0_{\alpha}(x,y):=X_{\alpha}(x,y,0)$, $X^0_{\beta}(x,y):=X_{\beta}(x,y,0)$ and $\xi(x,y):=X_{\gamma}(x,y,0)$.
Then, we have
\begin{gather}\label{def:X0alphabeta}
X^0_{\alpha}:=(\bar P\cos\varphi)^{-1}\partial f^0_V/\partial x,\quad
X^0_{\beta}:=(\bar P\sin\varphi)^{-1}\partial f^0_V/\partial y
\end{gather}
and that $F^0_V(x,y):=\left[\phi,X^0_{\alpha},X^0_{\beta},\xi\right](x,y)$
is an orthonormal frame field on $f^0_V(x,y)$.  
Furthermore, the vector fields $X^0_{\alpha}(x,y)$ and $X^0_{\beta}(x,y)$ are the unit principal curvature vectors on $f^0_V(x,y)$ (see \eqref{DXalpha:Xalphabeta} and \eqref{DXbeta:Xbetaalpha} below): if we regard $\phi(x,y)$ as a surface in $\mathbb{S}^3$, then $X^0_{\alpha}$ and $X^0_{\beta}$ are the unit principal curvature vectors and $\xi$ is a unit normal vector field of $\phi(x,y)$.
Thus, the differential $df^0_V(x,y)$ of $f^0_V(x,y)$ is determined as
\begin{align}
df^0_V(x,y)
&=\nonumber
\bar P(x,y)
\left(\cos\varphi\,dx X^0_{\alpha}+\sin\varphi\,dy X^0_{\beta}\right)(x,y)\\
&=\label{df0}
h^{-1}(x,y)\left(((x^2-y^2)/x) dx X^0_{\alpha}+2y dy X^0_{\beta}\right)(x,y),
\end{align}
and in particular $f^0_V(x,y)$ has the metric $g_0|_V$.
For the functions $\bar\kappa_i(x,y) \ (i=1,2)$ in \eqref{def:kappabar12}, the principal curvatures $\kappa_i(x,y,z) \ (i=1,2,3)$ of $f_V(x,y,z)$ satisfy $\kappa_i(x,y,0)=\bar\kappa_i(x,y) \ (i=1,2,3)$ on $V$.
The structure equation of $f_V(x,y,z)$ (i.e., the equation for $dF_V(x,y,z)$) is determined from $(\varphi(x,y,z),P(x,y,z),\kappa_i(x,y,z))$, and then the structure equation of $f^0_V(x,y)$ (i.e., the equation for $dF^0_V(x,y)$) is determined from $\Phi$ and $\mathcal{P}$, as the restriction to $F_V(x,y,0)$ of those for $F_V(x,y,z)$ (see the proof of Lemma \ref{lemma:B-C} below).

Now, we divide $D$ into several closed sub-domains according to the singularity $D\cap S_1$ of $(D,g_0)$:
\begin{gather}
\label{def:D1D2Di1Di2}
\begin{split}
D_1:=\{(x,y)\in D\;|\; y\geq 0\},&\quad
D_2:=\{(x,y)\in D\;|\; y\leq 0\},\\
D_{i1}:=\{(x,y)\in D_i\;|\; x^2\leq y^2\},&\quad
D_{i2}:=\{(x,y)\in D_i\;|\; x^2\geq y^2\}.
\end{split}
\end{gather}
For each $D_{ij}$, we arbitrarily fix a simply connected open set $V_{ij}$ such that $ V_{ij}\subset D_{ij}\setminus S$. Let $\nabla'$ be the canonical connection of $\R^4$: $\nabla'_{\partial/\partial x}=\partial/\partial x$.  On each domain $V_{ij}$, the structure equation of $f^0_{V_{ij}}(x,y)$ is determined as the following form:
\begin{gather}
\label{DXalpha:Xalphabeta}
\begin{split}
\nabla'_{X_{\alpha}^0}X_{\alpha}^0
&=
\bar\kappa_1\phi-B_1\xi -C_1X_{\beta}^0,\\
\nabla'_{X_{\alpha}^0}\phi=-\bar\kappa_1X_{\alpha}^0,\quad
&\nabla'_{X_{\alpha}^0}\xi=B_1X_{\alpha}^0,\quad
\nabla'_{X_{\alpha}^0}X_{\beta}^0=C_1X_{\alpha}^0,
\end{split}
\end{gather}
and
\begin{gather}
\label{DXbeta:Xbetaalpha}
\begin{split}
\nabla'_{X_{\beta}^0}X_{\beta}^0
&=
\bar\kappa_2\phi-B_2\xi-C_2X_{\alpha}^0,\\
\nabla'_{X_{\beta}^0}\phi=-\bar\kappa_2X_{\beta}^0,\quad
&\nabla'_{X_{\beta}^0}\xi=B_2X_{\beta}^0,\quad
\nabla'_{X_{\beta}^0}X_{\alpha}^0=C_2X_{\beta}^0,
\end{split}
\end{gather}
where $B_k$ and $C_k$ are functions on $V_{ij}$ and $\bar\kappa_i \ (i=1,2)$ are the functions given in \eqref{def:kappa12}:
\begin{gather*}
\bar\kappa_1=(2y/(x^2-y^2))\big(X_0+\textstyle\frac{\sqrt{5}}{2}\big),\quad
\bar\kappa_2=y^{-1}\left(((x^2+y^2)/2)(X_0'/x)-\big(X_0+\textstyle\frac{\sqrt{5}}{2}\big)\right).
\end{gather*}
\begin{lemma}\label{lemma:B-C}
On each $V_{ij}$ given above, the functions $B_k$, $C_k$ in \eqref{DXalpha:Xalphabeta} and \eqref{DXbeta:Xbetaalpha} are determined as follows:
\begin{gather*}
B_1(x,y) = -(x^2-y^2)^{-1}\big(X_0+\textstyle\frac{\sqrt{5}}{2}\big)-\sqrt{5},\quad
C_1(x,y) = -2(x^2-y^2)^{-1}\big(X_0+\textstyle\frac{\sqrt{5}}{2}\big),\\
B_2(x) = -({1}/{2})\left({X_0'}/{x}\right)-\sqrt{5},\quad
C_2(x) = X_0''-{X_0'}/{x}.
\end{gather*}
Then, for these functions $B_k$ and $C_k$, the equations in \eqref{DXalpha:Xalphabeta} and \eqref{DXbeta:Xbetaalpha} extend to $D\setminus S_1$.
\end{lemma}
\begin{proof}
The structure equation of the hypersurface $f_{V_{ij}}(x,y,z)$ is given at (\cite[equations (2.2.3) and (2.2.4)]{su4}).
Then, by $F^0_{V_{ij}}(x,y)=F_{V_{ij}}(x,y,0)$, the derivative of $F^0_{V_{ij}}=[\phi,X_{\alpha}^0,X_{\beta}^0,\xi]$ is obtained from the equation by taking as $\varphi(x,y,0)=\varphi(x,y)$, $\varphi_z(x,y,0)=\varphi_z(x,y)$ and $P(x,y,z)=\bar P(x,y)$, $P_z(x,y,0)=\bar P_z(x,y)$.
For $B_1$: We have
\begin{gather*}
B_1=\frac{\bar P^{-2}}{\cos\varphi}(\bar P\cos\varphi)_z
=-\frac{1}{\cos\varphi}\left((\bar P^{-1})_z\cos\varphi+\bar P^{-1}\varphi_z\sin\varphi\right).
\end{gather*}
Then, we obtain $B_1$ by
\begin{gather*}
\varphi_z = y/(x^2+y^2),\quad
(\bar P^{-1})_z=\sqrt{5}+\big((x^2-y^2)\big(X_0+\textstyle\frac{\sqrt{5}}{2}\big)+2xy^2X_0'\big)/(x^2+y^2)^2.
\end{gather*}

For $C_1$: We have
\begin{gather*}
C_1=\frac{\bar P^{-2}}{\sin\varphi\cos\varphi}(\bar P\cos\varphi)_y
=\frac{-1}{\sin\varphi\cos\varphi}\left((\bar P^{-1})_y\cos\varphi+\bar P^{-1}\varphi_y\sin\varphi\right).
\end{gather*}
Then, we obtain $C_1$ by
\begin{gather*}
\varphi_y=2x/(x^2+y^2),\quad
(\bar P^{-1})_y=2xy(x^2+y^2)^{-2}\big(-2X_0+2xX_0'-\sqrt{5}\big).
\end{gather*}
The functions $B_2$ and $C_2$ are obtained in the same way. 

The last statement follows from the fact that the functions $B_k$ and $C_k$ above are independent of the choice of domains $V_{ij}\subset D_{ij}\setminus S$.
\end{proof}

Next, we obtain the following lemma directly from \eqref{def:X0alphabeta} and Lemma \ref{lemma:B-C}:
\begin{lemma}\label{lemma:dPhi}
On each $V_{ij}$, we have the following equations:
\begin{gather*}
\nabla'_{\partial/\partial x}\phi=-a_1X_{\alpha}^0, \ \ \ \nabla'_{\partial/\partial x}\xi=b_1X_{\alpha}^0, \ \ \ \nabla'_{\partial/\partial x}X_{\beta}^0=c_1X_{\alpha}^0, \\
\nabla'_{\partial/\partial y}\phi=-a_2X_{\beta}^0, \ \ \ \nabla'_{\partial/\partial y}\xi=b_2X_{\beta}^0, \ \ \ 
\nabla'_{\partial/\partial y}X_{\alpha}^0=c_2X_{\beta}^0,
\end{gather*}
where
\begin{gather*}
a_1=a_1(x,y):=\frac{2y}{xh(x,y)}\left(X_0+\textstyle\frac{\sqrt{5}}{2}\right), \ \ 
b_1=b_1(x,y):=\frac{-1}{xh(x,y)}\left(X_0+\textstyle\frac{\sqrt{5}}{2}+\sqrt{5}(x^2-y^2)\right) \\ 
c_1=c_1(x,y):=\frac{-2}{xh(x,y)}\left(X_0+\textstyle\frac{\sqrt{5}}{2}\right), \\
a_2=a_2(x,y):=\frac{-1}{h(x,y)}\left(2X_0+\sqrt{5}-(x^2+y^2)\frac{X_0'}{x}\right), \ \ 
b_2=b_2(x,y):=\frac{-2y}{h(x,y)}\left(\frac{X_0'}{2x}+\sqrt{5}\right), \\ 
c_2=c_2(x,y):=\frac{2y}{h(x,y)}\left(X_0''-\frac{X_0'}{x}\right).
\end{gather*}
Then, all equations above are independent of the choice of domains $V_{ij}\subset D_{ij}\setminus S$ and they are analytic equations defined on $D$.
\end{lemma}

We also have $\nabla'_{\partial/\partial x}X_{\alpha}^0=a_1\phi-b_1\xi-c_1X^0_{\beta}$ and $\nabla'_{\partial/\partial y}X_{\beta}^0=a_2\phi-b_2\xi-c_2X^0_{\alpha}$ by \eqref{DXalpha:Xalphabeta} and \eqref{DXbeta:Xbetaalpha}.

For the analytic functions $(a_i,b_i,c_i)$ on $D$ in Lemma \ref{lemma:dPhi}, we define the matrix-valued functions $\Omega_1$ and $\Omega_2$ by
\begin{gather}\label{def:Omega12}
\Omega_1(x,y)=
\begin{bmatrix}
0 & a_1 & 0 & 0\\
-a_1 & 0 & c_1 & b_1\\
0 & -c_1 & 0 & 0\\
0 & -b_1 & 0 & 0
\end{bmatrix}
(x,y),\quad
\Omega_2(x,y)=
\begin{bmatrix}
0 & 0 & a_2 & 0\\
0 & 0 & -c_2 & 0\\
-a_2 & c_2 & 0 & b_2\\
0 & 0 & -b_2 & 0
\end{bmatrix}
(x,y).
\end{gather}
For a frame field $F^0(x,y):=[\phi,X^0_{\alpha},X^0_{\beta},\xi](x,y)$ and the matrix-valued differential $1$-form $\Omega=\Omega_1 dx +\Omega_2 dy$ on $D$, the equations of Lemma \ref{lemma:dPhi} are summarized as
\begin{gather}\label{eq:dF0}
dF^0=F^0\Omega,
\end{gather}
which is an analytic equation on the whole domain $D$, and in particular, it is independent of the choice of $V_{ij}\subset D_{ij}\setminus S$. Now, in Theorem \ref{thm:extension} below, we shall verify that the equation \eqref{eq:dF0} has a solution $F^0(x,y)$ on $D$. Then, the solution $F^0(x,y)$ is uniquely determined up to a transformation $QF^0(x,y)$ by a constant orthogonal matrix $Q$.
For the sake of simplicity for the argument in Theorem \ref{thm:extension}, we determine an initial condition of \eqref{eq:dF0} by
\begin{gather}\label{def:init-F0}
F^0(x_0,y_0)=\mathrm{Id}
\end{gather}
for a while, where $(x_0,y_0)$ is a point of $V_{12}$ and $\mathrm{Id}$ is the unit matrix.
It is possible to take such an initial condition.
In fact, for the frame field $F_{V_{12}}(x,y,z)$ determining the hypersurface $f_{V_{12}}(x,y,z)$, suppose $F_{V_{12}}(x_0,y_0,0)\neq \mathrm{Id}$. We take a constant orthogonal matrix $C$ such that $CF_{V_{12}}(x_0,y_0,0)=\mathrm{Id}$ holds. Then, the frame $CF_{V_{12}}(x,y,z)$ determines the hypersurface $Cf_{V_{12}}(x,y,z)$ by $C df_{V_{12}}=d(Cf_{V_{12}})$.
\begin{thm}\label{thm:extension}
An analytic orthonormal frame field $F^0(x,y)$ is uniquely determined on $D$ such that it is a solution to the structure equation \eqref{eq:dF0} under the initial condition \eqref{def:init-F0}.
Then, the surface $f^0_{V_{12}}(x,y)$ on $V_{12}$ determined above extends to an analytic surface $f^0(x,y)$ with the metric $g_0$ defined on the whole domain $D$.
Furthermore, for any simply connected domain $V'_{ij}$ satisfying $ V'_{ij}\subset D_{ij}\setminus S$, there is a generic conformally flat hypersurface $ f_{V'_{ij}}(x,y,z)$ on $V'_{ij}\times I$ such that $f^0(x,y)= f_{V'_{ij}}(x,y,0)$ holds on $V'_{ij}$.
\end{thm}
\begin{proof}
The frame field $F_{V_{12}}(x,y,z)$ on $V_{12}\times I$ is determined from the hypersurface $f_{V_{12}}(x,y,z)$.
Since $F^0_{V_{12}}(x,y)=F_{V_{12}}(x,y,0)$ and $f^0_{V_{12}}(x,y)=f_{V_{12}}(x,y,0)$, the $1$-form $\Omega(x,y)$ satisfies the Maurer-Cartan equation $d\Omega+\Omega\wedge \Omega=0$ on the open domain $V_{12}$.
Then, since $\Omega$ is analytic on $D$, the Maurer-Cartan equation is satisfied on the whole domain $D$.
Hence, a frame field $F^0(x,y)$ on $D$ is uniquely determined under the condition \eqref{def:init-F0}, which is the extension of $F^0_{V_{12}}(x,y)$ on $V_{12}$.

Furthermore, for the vector fields $X^0_{\alpha}$ and $X^0_{\beta}$ of $F^0(x,y)$ on $D$, an analytic surface $f^0(x,y)$ on $D$ is determined by
\begin{gather}\label{def:theta12}
df^0=\theta_1X^0_{\alpha}+\theta_2X^0_{\beta},\quad
\theta_1:=(x^2-y^2)/(xh(x,y)) dx,\quad
\theta_2:=2y/h(x,y) dy,
\end{gather}
since $f^0_{V_{12}}(x,y)$ satisfies \eqref{df0} on $V_{12}$.
In fact, $df^0(x,y)$ in \eqref{def:theta12} satisfies $d (df^0) = 0$ on $D$ by Lemma \ref{lemma:dPhi}.
Hence, $f^0(x,y)$ is an extension of $f^0_{V_{12}}(x,y)$ to $D$.
Then, $\phi$ and $\xi$ are the normal vector fields of $f^0(x,y)$, which are distinguished by the condition for the hypersurface $f_{V_{12}}(x,y,z)$.
Furthermore, the surface $f^0(x,y)$ on $D$ has the metric $g_0$ by \eqref{def:theta12}.

Finally, the existence of the generic conformally flat hypersurface $f_{V'_{ij}}(x,y,z)$ in the last statement follows from the fact that $F^0(x,y)$ is determined by $\Phi$ and $\mathcal{P}$.
\end{proof}

\begin{defn}[Curvature surface defined on $D$]
By Theorem \ref{thm:extension}, we may recognize that the analytic surface $f^0(x,y)$ on $D$ is an extended curvature surface.
Our aim is to study the property and the structure on $f^0(x,y)$.
We call $f^0(x,y)$ and $F^0(x,y)$ in Theorem \ref{thm:extension}, respectively, a curvature surface defined on $D$ and a frame field determining the curvature surface $f^0(x,y)$ on $D$.
Then, we can arbitrarily take the initial condition of $F^0(x,y)$ by $F^0(x_0,y_0)=Q$ not only \eqref{def:init-F0}, where $(x_0,y_0)$ and $Q$ are a point of $D$ and an orthogonal matrix, respectively.
For a fixed frame field $F^0(x,y)$, the curvature surface $f^0(x,y)$ is determined uniquely up to a parallel translation.
\end{defn}

Next, we study the coordinate (the extended curvature) lines of a curvature surface $f^0(x,y)$ on $D$.
Let $\|{\bs{v}}\|$ be the Euclidean norm for a vector ${\bs{v}}\in \R^4$.
\begin{thm}\label{thm:xcurve}
Let $f^0(x,y)$ be a curvature surface on $D$ and $F^0(x,y)=[\phi,X^0_{\alpha},X^0_{\beta},\xi](x,y)$ be the frame field determining $f^0(x,y)$. Then, for an $x$-curve $f^0(x,y)$ with fixed $y$, we have the following facts (1), (2) and (3):
\begin{enumerate}
\item
Along any $x$-curve $f^0(x,y)$, the vector $(y X^0_{\beta}-\phi)(x,y)$ is constant. That is, an analytic unit vector ${\bs{u}}(y)$ of $y$ is determined by ${\bs{u}}(y)=(1+y^2)^{-1/2}(y X^0_{\beta}-\phi)(x,y)$.

\item
Let ${\bs{f}}(x,y)$ be an analytic unit vector defined by
\begin{gather*}
{\bs{f}}(x,y):=((1+y^2)(5+4y^2))^{-1/2}\left(X^0_{\beta}-2(1+y^2)\xi+y\phi\right)(x,y).
\end{gather*}
When we regard the surface $f^0(x,y)$ as a one-parameter family of $x$-curves, it is expressed as
\begin{gather*}
f^0(x,y)=(2\sqrt{5})^{-1}\sqrt{(5+4y^2)(1+y^2)^{-1}}{\bs{f}}(x,y)+A(y),
\end{gather*}
where $A(y)$ is a $\R^4$-valued analytic function of $y$ determined uniquely up to a parallel translation.  

\item
Any $x$-curve $f^0(x,y)$ lies on a $2$-sphere $\mathbb{S}^2_y$ of radius $(2\sqrt{5})^{-1}\sqrt{(5+4y^2)(1+y^2)^{-1}}$ in an affine hyperplane $\R^3_y$ perpendicular to ${\bs{u}}(y)$, and $A(y)$ is the center of $\mathbb{S}^2_y$.
\end{enumerate}
\end{thm}
\begin{proof}
We firstly verify (1)--(3) on the domain $U:=\{(x,y)\in D_{12}\;|\; x>y\geq 0\}$, and then we use the equations in Lemma \ref{lemma:B-C}.
Now, we have $\nabla'_{X^0_{\alpha}}(yX^0_{\beta}-\phi)=0$.
Hence, we have (1).
Next, we have $\nabla'_{X^0_{\alpha}}(X^0_{\beta}-2\xi)=2\sqrt{5}X^0_{\alpha}$.
The component $\hat{\bs{f}}(x,y)$ of $(X^0_{\beta}-2\xi)(x,y)$ perpendicular to ${\bs{u}}(y)$ is given by
\begin{gather*}
\hat{\bs{f}}(x,y):=X^0_{\beta}-2\xi-y(1+y^2)^{-1/2}{\bs{u}}
=(1+y^2)^{-1}\left(X^0_{\beta}-2(1+y^2)\xi+y\phi\right).
\end{gather*}
Hence, we have $\nabla'_{X^0_{\alpha}}\hat{\bs{f}}=2\sqrt{5}X^0_{\alpha}$.
Then, since ${\bs{f}}(x,y)$ is the normalization of $\hat{\bs{f}}(x,y)$, we have (2) by $X^0_{\alpha}f^0=X^0_{\alpha}$.
Finally, since $X^0_{\alpha}(x,y)\perp {\bs{u}}(y)$ (or ${\bs{f}}(x,y)\perp {\bs{u}}(y)$) and the fact that $\|\hat{\bs{f}}(x,y)\|$ is a function only of $y$, we obtain (3).
By the above argument, we have verified the theorem for $x$-curves on $U$.

Next, all $x$-curves $f^0(x,y)$ on $D$ are also expressed as the form in (2), since all our objects: the frame field $F^0(x,y)$, the surface $f^0(x,y)$ and the vector ${\bs{f}}(x,y)$, are analytic on $D$.
Actually, by Lemma \ref{lemma:dPhi}, we have the following equation,
\begin{gather*}
\nabla'_{\partial/\partial x}\big((2\sqrt{5})^{-1}\sqrt{(5+4y^2)(1+y^2)^{-1}}\,{\bs{f}}_2+A\big)= [(x^2-y^2)/(xh(x,y))]\, X^0_{\alpha}
\end{gather*}
on $D$, which coincides with $\partial f^0/\partial x$ on $D$.
We can verify directly by Lemma \ref{lemma:dPhi} that all $x$-curves on $D$ also satisfy (1) and (3).
In consequence, the proof has been completed.
\end{proof}
\begin{thm}\label{thm:ycurve}
Let $f^0(x,y)$ be a curvature surface on $D$ and $F^0(x,y)=[\phi,X^0_{\alpha},X^0_{\beta},\xi](x,y)$ be the frame field determining $f^0(x,y)$. Then, for a $y$-curve $f^0(x,y)$ with fixed $x$, we have the following facts (1), (2) and (3):
\begin{enumerate}
\item
Along any $y$-curve $f^0(x,y)$, the vector $(-B_2X^0_{\alpha}+C_2\xi)(x,y)$ is constant, where $B_2(x)=-(1/2)(X_0'/x)-\sqrt{5}$ and $C_2(x)=X_0''-(X_0'/x)$.
That is, an analytic unit vector $\tilde{\bs{u}}(x)$ of $x$ is determined by \ $\tilde{\bs{u}}(x)={(B_2^2+C_2^2)^{-1/2}(-B_2X^0_{\alpha}+C_2\xi)}(x,y)$, and then $(B_2^2+C_2^2)(x)=(X_0'/x) (2X_0-xX_0'-X_0'/x+\sqrt{5})>0$ holds for $x\in \R$.

\item
Let $\tilde{\bs{f}}(x,y)$ be an analytic unit vector defined by
\begin{gather*}
\tilde{\bs{f}}(x,y):=(B_2^2+C_2^2)^{-1/2}(x)(B_2\xi+C_2X^0_{\alpha})(x,y).
\end{gather*}
When we regard the surface $f^0(x,y)$ as a one-parameter family of $y$-curves, it is expressed as
\begin{gather*}
f^0(x,y)=(B_2^2+C_2^2)^{-1/2}(x)\tilde{\bs{f}}(x,y)+\tilde A(x),
\end{gather*}
where $\tilde A(x)$ is a $\R^4$-valued analytic function of $x$ determined uniquely up to a parallel translation.

\item
Any $y$-curve $f^0(x,y)$ lies on a standard $2$-sphere $\mathbb{S}^2_x$ of radius $(B_2^2+C_2^2)^{-1/2}(x)$ in an affine hyperplane $\R^3_x$ perpendicular to $\tilde{\bs{u}}(x)$, and $\tilde A(x)$ is the center of $\mathbb{S}^2_x$.  
\end{enumerate}
\end{thm}
\begin{proof}
We firstly prove the theorem on the domain $\tilde U:=\{(x,y)\in D_1|\ y>0\}$, and then we use the equations in Lemma \ref{lemma:B-C}.
Now, we have $\nabla'_{X^0_{\beta}}(-B_2X^0_{\alpha}+C_2\xi)=0$ and $(B_2^2+C_2^2)(x)>0$ on $\R$ by $B_2(x)<0$.
Furthermore $(B_2^2+C_2^2)(x)$ is expressed as the form in (1) by \eqref{diffeq:X0-2nd}.
Hence, we have obtained (1).
Next, we have $\nabla'_{X^0_{\beta}}(B_2\xi+C_2X^0_{\alpha})=(B_2^2+C_2^2)X^0_{\beta}$.
Then, by $(B_2^2+C_2^2)(x)\neq 0$ and $X^0_{\beta}f^0=X^0_{\beta}$, we have (2).
Finally, since $X^0_{\beta}\perp\tilde{\bs{u}}(x)$ (or $\tilde{\bs{f}}(x,y)\perp \tilde{\bs{u}}(x)$), we have (3) by $\|(B_2^2+C_2^2)^{-1/2}(x) \tilde{\bs{f}}(x,y)\|=(B_2^2+C_2^2)^{-1/2}(x)$.
Thus, we have verified the theorem for $y$-curves on $\tilde U$.
These results also hold for all $y$-curves on $D$ by Lemma \ref{lemma:dPhi} similarly to the proof of Theorem \ref{thm:xcurve}.
In consequence, the proof has been completed.
\end{proof}
\begin{rem}\label{rem:uutilde}
(1)\;
In Theorem \ref{thm:xcurve} and Theorem \ref{thm:ycurve}, the vectors ${\bs{u}}(y)$ and $\tilde{\bs{u}}(x)$ are perpendicular for all $(x,y)\in D$.
In fact, by the definitions of ${\bs{u}}(y)$ and $\tilde{\bs{u}}(x)$, we have
\begin{gather*}
\langle {\bs{u}}(y),\tilde{\bs{u}}(x)\rangle=C(x,y)\,\langle yX^0_{\beta}-\phi,-B_2X^0_{\alpha}+C_2\xi\rangle(x,y)=0,
\end{gather*}
where $\langle {\bs{u}},\tilde{\bs{u}}\rangle$ is the inner product of $\R^4$ and $C(x,y):=\left((1+y^2)(B_2^2+C_2^2)(x)\right)^{-1/2}$. 

(2)\;
The vector ${\bs{u}}(y)$ of $y$ (or $\tilde{\bs{u}}(x)$ of $x$) moves on the unit circle $\mathbb{S}^1$ in a plane without stopping.
We shall give a simple proof of these facts in the next section (see Theorem \ref{thm:binfcinf}).
Certainly, we can also make sure these facts by direct calculation, but it is very hard.
Here, we only give the norms of the first derivatives of ${\bs{u}}(y)$ and $\tilde{\bs{u}}(x)$:
\begin{gather*}
\|\left({d{\bs{u}}}/{d y}\right)(y)\|=\frac{y^2}{1+y^2},\quad
\|\left({d\tilde{\bs{u}}}/{d x}\right)(x)\|=\frac{\left(X_0'+2\sqrt{5}x\right)\left(2X_0+\sqrt{5}\right)}{4xX_0' \left(2X_0-xX'_0-X_0'/x+\sqrt{5}\right)}=:T(x)
\end{gather*}
These norms show that the length of the curve ${\bs{u}}(y)$ (resp.\ $\tilde{\bs{u}}(x)$) diverges to $\infty$ as $y$ tends to $\pm\infty$ (resp.\ as $x$ tends to $\infty$): in the right hand side of the second equation, we have $(2X_0-xX_0')(x)=O(x)$ and $X_0'(x)$ is a positive bounded function, by Proposition \ref{prop:ineq4X0}.
Furthermore, the length of $\tilde{\bs{u}}(x)$ on $(\varepsilon,1]$ also diverges to $\infty$ as $\varepsilon(>0)\rightarrow 0$, since we have $\|\nabla'_{\partial/\partial x}\tilde{\bs{u}}(x)\|\approx \sqrt{5}/(2x)$ by $X_0(0)=5/2$ and $X_0'(x)\approx 2x$.

(3)\;
The curvature surface $f^0(x,y)$ on $D$ is bounded in $\R^4$. In fact, let us take $\tilde A(1)$ as a point in $\R^4$. In the theorems \ref{thm:xcurve}-(2) and \ref{thm:ycurve}-(2), the first terms of the right hand sides in $f^0(x,y)$ lie in uniformly bounded 2-spheres centered on the origin. Hence, we have only to show that $A(y)$ is a bounded function. As we determine $\tilde A(1)$ as a point in $\R^4$, the curve $f^0(1,y)$ is bounded by Theorem \ref{thm:ycurve}. Hence, the curve $f^0(1,y)$ and $A(y)$ in Theorem \ref{thm:xcurve} are also bounded.  
\end{rem}

By Theorems \ref{thm:xcurve} and \ref{thm:ycurve}, any $x$-curve $f^0(x,y)$ with fixed $y$ (resp.\ any $y$-curve $f^0(x,y)$ with fixed $x$) belongs to an affine hyperplane perpendicular to ${\bs{u}}(y)$ (resp.\ $\tilde{\bs{u}}(x)$).
We can determine the following orthonormal frame fields along the curves: let ${\bs{f}}(x,y)$ and $\tilde{\bs{f}}(x,y)$ be the vectors in Theorems \ref{thm:xcurve} and \ref{thm:ycurve}, respectively; along each $x$-curve, the frame field is given by
\begin{gather}\label{def:ff-u-X0alpha-f-u2}
\left[{\bs{u}}(y), X^0_{\alpha}(x,y),{\bs{f}}(x,y),{\bs{u}}_2(x,y)\right]
\end{gather}
where ${\bs{u}}_2(x,y)$ is defined by ${\bs{u}}_2(x,y):=(5+4y^2)^{-1/2}(2X^0_{\beta}+\xi+2y\phi)$; along each $y$-curve, the frame field is given by
\begin{gather}\label{def:ff-utilde-X0beta-ftilde-phi}
\big[\tilde{\bs{u}}(x),X^0_{\beta}(x,y),\tilde{\bs{f}}(x,y),\phi(x,y)\big].
\end{gather}

Now, in the following theorem, we verify that each $x$-curve with $y\neq 0$ (resp.\ each $y$-curve with $x$) of a curvature surface $f^0(x,y)$ on $D$ has a cusp at the point $x=|y|$ (resp.\ at the point $y=0$).
Then, we say that a curve $p(t)$ in $\R^3$ with $p(0)={\bf 0}$ has a cusp of type $(2,3,4)$ at $t=0$, if $p(t)$ is expressed as $p(t)=(at^2,bt^3,ct^4)$ around $t=0$ with constants $a$, $b$ and $c$ ($abc\neq 0$).
\begin{thm}\label{thm:cusp}
For any coordinate curve of a curvature surface $f^0(x,y)$ on $D$, we have the following facts (1) and (2):
\begin{enumerate}
\item
Any $x$-curve with $y\neq 0$ has a cusp of type $(2,3,4)$ at $x=|y|$.

\item
Any $y$-curve has a cusp of type $(2,3,4)$ at $y=0$.
\end{enumerate}
\end{thm}
\begin{proof}
Let $F^0(x,y)=[\phi,X^0_{\alpha},X^0_{\beta},\xi](x,y)$ be the frame field determining $f^0(x,y)$ on $D$.
For (1), let $f^0(x,y)$ be an $x$-curve with fixed $y\neq 0$.
We study the curve only in a small neighborhood of $(|y|,y)$.
Now, the first derivative of the curve is given by $f^0_x(x,y)=\left((x^2-y^2)/(xh(x,y))\right) X^0_{\alpha}(x,y)$ and we have
\begin{gather*}
\left((x^2-y^2)/(xh(x,y))\right)_x
=(x^2h^2(x,y))^{-1}\left((x^2+y^2)h(x,y)+(x^2-y^2)^2(X_0''-X_0'/x)\right),\\
f^0_{xx}(x,y)=\left((x^2-y^2)/(xh(x,y))\right)_x X^0_{\alpha}
+\left((x^2-y^2)/(xh(x,y))\right)\left(\nabla'_{\partial/\partial x}X^0_{\alpha}\right)(x,y),\\
f^0_{xxx}(|y|,y)=\left((x^2-y^2)/(xh(x,y))\right)_{xx}{(|y|,y)}X^0_{\alpha}(|y|,y)+(4/h(|y|,y))\left(\nabla'_{\partial/\partial x}X^0_{\alpha}\right)(|y|,y).
\end{gather*}
We define the functions $w_i(x) \ (i=1,2,3)$ of $x$ by
\begin{align*}
w_1(x):=\langle f^0(x,y),X^0_{\alpha}(|y|,y)\rangle,\quad
w_2(x):=\langle f^0(x,y),{\bs{u}}_2(|y|,y)\rangle,\quad
w_3(x):=\langle f^0(x,y),{\bs{f}}(|y|,y)\rangle.
\end{align*}
Then, we directly have
\begin{gather*}
(w_1)_x(|y|)=(w_2)_x(|y|)=(w_3)_x(|y|)=(w_2)_{xx}(|y|)=(w_3)_{xx}(|y|)=0,\quad
(w_1)_{xx}(|y|)\neq 0.
\end{gather*}
Furthermore, we have $(w_2)_{xxx}(|y|)\neq 0$ by $\nabla'_{\partial/\partial x}X^0_{\alpha}=a_1\phi-b_1\xi-c_1X^0_{\beta}$ and Lemma \ref{lemma:dPhi}.
For $w_3$, $(w_3)_{xxx}(|y|)=0$ holds by $\langle\nabla'_{\partial/\partial x}X^0_{\alpha},{\bs{f}}\rangle(|y|,y)=0$, and further we have $(w_3)_{xxxx}(|y|)\neq 0$ by $\langle (\nabla'_{\partial/\partial x})^2X^0_{\alpha}, \ {\bs{f}} \rangle (|y|,y)\neq 0$.
Hence, for $t:=x-|y|$, we have
\begin{gather*}
w_1(|y|+t)\approx w_1(|y|)+(t^2/2)(w_1)_{xx}(|y|),\quad
w_2(|y|+t)\approx w_2(|y|)+(t^3/6)(w_2)_{xxx}(|y|),\\
w_3(y+t)\approx w_3(|y|)+(t^4/24)(w_3)_{xxxx}(|y|),
\end{gather*}
as $t$ tends to 0, which show that the $x$-curve $f^0(x,y)=(w_1,w_2,w_3)(x)$ has the cusp of type $(2,3,4)$ at the point $(|y|,y)$.
 
For (2), let $f^0(x,y)$ be a $y$-curve with fixed $x$. We study the curve only in a small neighborhood of $(x,0)$.
The first derivative of the curve is given by $f^0_y(x,y)=(2y/h(x,y)) X^0_{\beta}(x,y)$, and we have
\begin{gather*}
(2y/h(x,y))_y=(2/h^2(x,y))\left(h(x,y)-2y^2(X_0'/x)\right),\\
f^0_{yy}(x,y)=(2y/h(x,y))_y X^0_{\beta}(x,y) +(2y/h(x,y)) (\nabla'_{\partial/\partial y}X^0_{\beta})(x,y),\\
f^0_{yyy}(x,0)=(2y/h(x,y))_{yy}(x,0) X^0_{\beta}(x,0) +(4/h(x,0)) (\nabla'_{\partial/\partial y}X^0_{\beta})(x,0).
\end{gather*}
We define the functions ${\tilde w}_i(y) \ (i=1,2,3)$ by
\begin{align*}
{\tilde w}_1(y):=\langle f^0(x,y),X^0_{\beta}(x,0)\rangle,\quad
{\tilde w}_2(y):=\langle f^0(x,y),\phi(x,0)\rangle,\quad
{\tilde w}_3(y):=\langle f^0(x,y),\tilde{\bs{f}}(x,0)\rangle.
\end{align*}
Then, we obtain
\begin{gather*}
({\tilde w}_1)_{yy}(0)\neq 0,\quad
({\tilde w}_2)_{yyy}(0)\neq 0,\quad
({\tilde w}_3)_{yyyy}(0)\neq 0,
\end{gather*}
and that the lower derivatives of each $\tilde{w}_i(y)$ vanish at $y=0$, in the same way as in (1).
Hence, we have verified that the $y$-curve also has the cusp of type $(2,3,4)$ at $y=0$.
\end{proof}

By Theorem \ref{thm:cusp}, the curves $f^0(|y|,y)$ of $y$ and $f^0(x,0)$ of $x$ are the cuspidal edges in a curvature surface $f^0(x,y)$ on $D$.
The curve $f^0(x,0)$ of $x$ is also the singular set of the Poincar\'{e} metric $\check{g}_H$ on $D$.
For these curves, we have the following corollary. In the following two corollaries, $F^0(x,y)=[\phi,X^0_{\alpha},X^0_{\beta},\xi](x,y)$ is the frame field determining $f^0(x,y)$ on $D$. 
\begin{cor}\label{cor:f0x}
(1) The two curves $f^0(|y|,y)$ of $y$ are the envelopes of the family of $y$-curves $f^0(x,y)$ with $x$.
(2) The vector $\phi^0:=\phi(x,0)$ does not depend on $x$. The curve $f^0(x,0)$ of $x$ lies on a $2$-sphere $\mathbb{S}^2_{y=0}$ of radius $1/2$ in an affine hyperplane $\R^3_{y=0}$ perpendicular to $\phi^0$.
\end{cor}
\begin{proof}
(1) For the sake of simplicity, we assume $y>0$.
The derivative of the $y$-curve $f^0(y,y)$ is given by $\left(\left(\partial/\partial x+\partial/\partial y\right)f^0\right)(y,y) =(\partial f^0/\partial y)(y,y)$ by \eqref{def:theta12}.
Hence, the tangent vectors of both $y$-curves $f^0(x,y)$ and $f^0(y,y)$ coincide at $(y,y)$, which implies that the $y$-curve $f^0(y,y)$ is the envelope of the family of $y$-curves $f^0(x,y)$.

(2) We have $(\nabla'_{\partial/\partial x}\phi)(x,0)=0$ by Lemma \ref{lemma:dPhi}.
The other statement is already proved in Theorem \ref{thm:xcurve}.
\end{proof}

\small
Now, let us visualize the curvature surface $f^0 (x,y)$ around the curve $f^0(y,y)$ by using the approximation $f^{\delta_n} (x,y)$ of $f^0 (x,y)$, where $\delta_n:=1/n$, that will be defined for each positive integer $n$ in Section \ref{sec:approx}. In Figures \ref{fig:diagonal-x} and \ref{fig:diagonal-y} below, we illustrate $f^{\delta_n} (x,y)$ via the projection $\pi: \R^4\rightarrow \R^3$, where $f^{\delta_n} (y_0,y_0)={\bf 0}$ and $\pi \left((r_1,r_2,r_3,r_4)\right) = (r_1,r_2,r_3)$ for the coordinates $(r_1,r_2,r_3,r_4)$ of $\R^4$ with respect to the frame $F^0 (y_0,y_0) = \mathrm{Id}$.
\normalsize

\noindent
\begin{figure}[H]
\hfill
\begin{minipage}{0.45\linewidth}
\centering
\includegraphics[width=\linewidth,keepaspectratio]{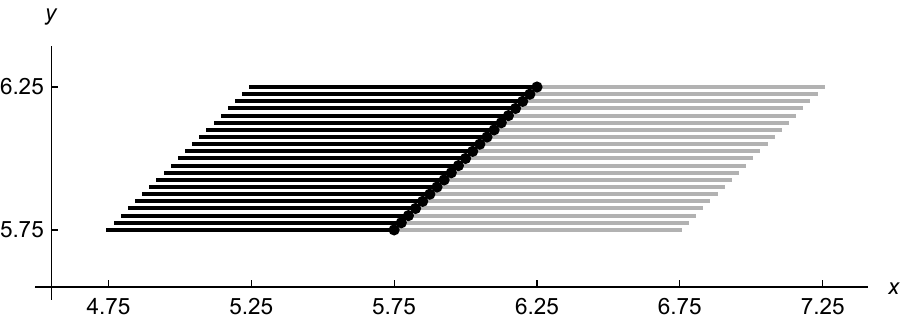}
\end{minipage}
\hfill
\begin{minipage}{0.06\linewidth}
$\xrightarrow{\pi\circ f^{1/20}}$
\end{minipage}
\hfill
\begin{minipage}{0.45\linewidth}
\centering
\includegraphics[width=\linewidth,keepaspectratio]{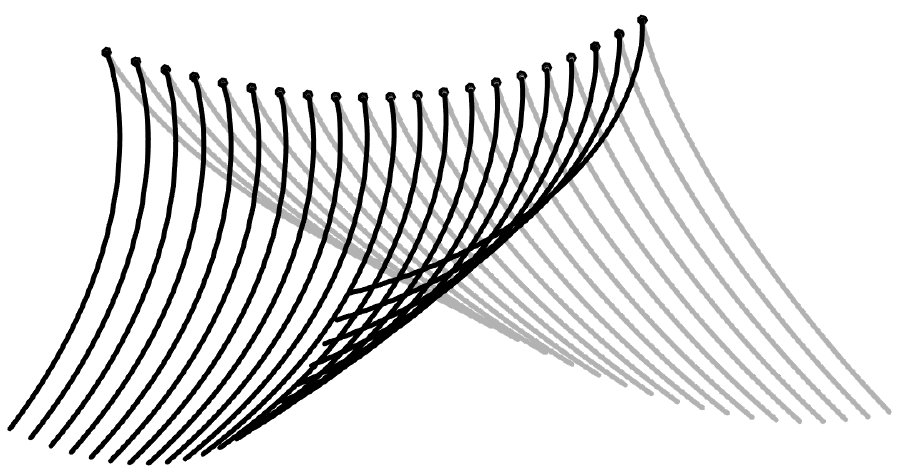}
\end{minipage}
\hfill
\caption{
\small
This shows $x$-curves passing through the points $\left\{f^{1/20} (y_k,y_k)\right\}$ on the cuspidal edge: the lines in the domain are given by $y_k = 5.75 + 0.025 k$\, $(k=0,1,\dots,20)$ and $y_k-1 \leq x \leq y_k+1$.
Each $x$-curve is drawn with black for $y_k-1 \leq x < y_k$ and with gray for $y_k < x \leq y_k+1$.
\normalsize}
\label{fig:diagonal-x}
\end{figure}

\noindent
\begin{figure}[H]
\hfill
\begin{minipage}{0.45\linewidth}
\centering
\includegraphics[width=\linewidth,keepaspectratio]{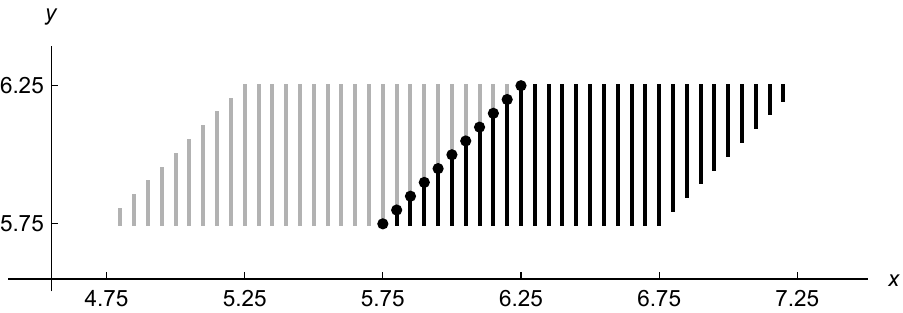}
\end{minipage}
\hfill
\begin{minipage}{0.06\linewidth}
$\xrightarrow{\pi\circ f^{1/20}}$
\end{minipage}
\hfill
\begin{minipage}{0.45\linewidth}
\centering
\includegraphics[width=\linewidth,keepaspectratio]{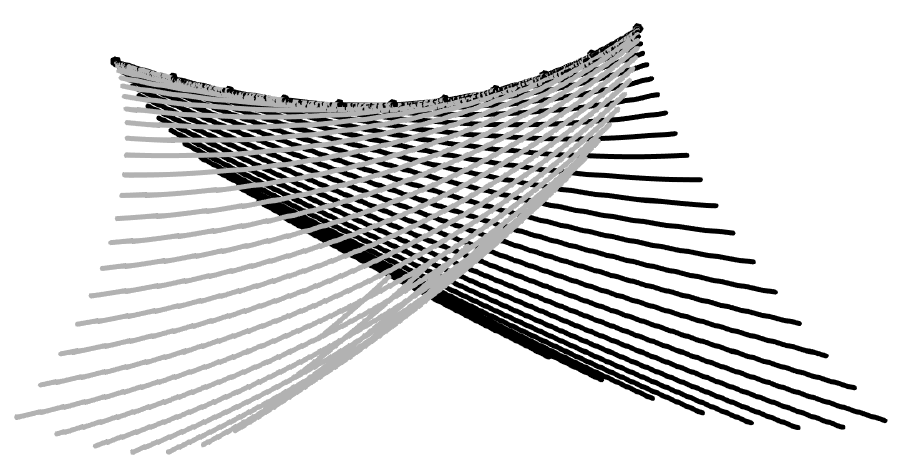}
\end{minipage}
\hfill
\caption{
\small
This shows the envelope made by $y$-curves passing through the points $\left\{f^{1/20} (x_k,x_k)\right\}$ from the same domain as in Figure \ref{fig:diagonal-x}: the lines in the domain are given by $x_k = 4.75 + 0.05 k$\, $(k=1,2,\dots,49)$, $x_k-1 \leq y \leq x_k+1$ and $5.75 \leq y \leq 6.25$.
\normalsize}
\label{fig:diagonal-y}
\end{figure}

\normalsize
Next, for a curvature surface $f^0(x,y)$ on $D$, we study the relation with $f^0|_{D_1}(x,y)$ and $f^0|_{D_2}(x,-y)$.
To the end, we translate the surface $f^0(x,y)$ such that the obtained surface $f^0(x,y)$ satisfies $f^0(p_0)=\bs{0}$ at some point $p_0:=(x_0,0)$.
Then, for the sake of simplicity, we recognize that the frame of $\R^4$ containing the new $f^0(x,y)$ is given by $F^0(p_0)=\mathrm{Id}$.
That is, $[X^0_{\alpha},X^0_{\beta},\xi](p_0)$ is the standard frame of $\R^3_{y=0}=\{(0,x_1,x_2,x_3)^T|x_i\in \R\}$, where $\bs{a}^T$ is the transposed vector of $\bs{a}$.  
\begin{cor}\label{cor:Af0}
Let $f^0(x,y)$ and $F^0(x,y)$ be the curvature surface and the frame field, expressed by the above coordinate system of $\R^4$.
Then, for the orthogonal matrix $B=[b_{ij}]$ such that $b_{11}=-1$, $b_{ii}=1 \ (2\leq i\leq 4)$ and $b_{ij}=0 \ (i\neq j)$, we have $(B\circ f^0)(x,y)=f^0(x,-y)$, $(B\circ \phi)(x,y)=-\phi(x,-y)$ and $(B\circ \xi)(x,y)=\xi(x,-y)$ for $y\geq 0$.
Furthermore, $(B\circ\tilde{\bs{u}})(x)=\tilde{\bs{u}}(x)$ and $(B\circ\tilde A)(x)=\tilde A(x)$ hold.
\end{cor}
\begin{proof}
Let $\bar y=-y$ for $y\geq 0$.
We put $\bar\phi(x,y):=-\phi(x,\bar y)$, $\bar X^0_{\alpha}(x,y):=X^0_{\alpha}(x,\bar y)$, $\bar X^0_{\beta}(x,y):=X^0_{\beta}(x,\bar y)$ and $\bar\xi(x,y):=\xi(x,\bar y)$.
We study the derivative in $y\geq 0$ of $\bar F^0:=[\bar\phi, \bar X^0_{\alpha}, \bar X^0_{\beta},\bar\xi]$ by Lemma \ref{lemma:dPhi}.
Then, we have $d\bar F^0(x,y)=\bar F^0(x,y)\Omega(x,y)$ for $y\geq 0$, where $\Omega(x,y)$ is the differential $1$-form in \eqref{eq:dF0} such that $d F^0(x,y)= F^0(x,y)\Omega(x,y)$.
Thus, $\bar F^0(x,y)$ is another solution to \eqref{eq:dF0} in $y\geq 0$, and hence there is an orthogonal matrix $B$ such that $BF^0(x,y)=\bar F^0(x,y)$.
Then, $B$ is determined at $p_0$ by $F^0(p_0)=[\phi,X^0_{\alpha},X^0_{\beta},\xi](p_0)=\mathrm{Id}$ and $\bar F^0(p_0)=[-\phi,X^0_{\alpha},X^0_{\beta},\xi](p_0)$ as the matrix in the statement.
Furthermore, $f^0(x,\pm y)$ is the integral surface of $(X^0_{\alpha},X^0_{\beta})(x,\pm y)$ in \eqref{def:theta12}, respectively, and $(B\circ f^0)(x,0)=f^0(x,0)$ holds by $f^0(x,0)\in \R^3_{y=0}$.
Hence, we have $(B\circ f^0)(x,y)=f^0(x,-y)$.
In consequence, we have verified the first statement in the corollary. 

For $(B\circ\tilde{\bs{u}})(x)=\tilde{\bs{u}}(x)$: Any $y$-curve $f^0(x,y)$ lies on a $2$-sphere $\mathbb{S}^2_x$ in $\R^3_x$ perpendicular to $\tilde{\bs{u}}(x)$ and it is not a circle (for example, see Figure \ref{fig:ycurvelong} in the next section).
Hence, we have $B(\R^3_x)=\R^3_x$ by $(B\circ f^0)(x,y)=f^0(x,-y)$, which shows $(B\circ\tilde{\bs{u}})(x)=\tilde{\bs{u}}(x)$ by $(B\circ\tilde{\bs{u}})(x_0)=\tilde{\bs{u}}(x_0)$.

For $(B\circ\tilde A)(x)=\tilde A(x)$: We have $f^0(x,y)=(B_2^2+C_2^2)^{-1/2}(x)\tilde{\bs{f}}(x,y)+\tilde A(x)$ and $f^0(x,-y)=(B_2^2+C_2^2)^{-1/2}(x)\tilde{\bs{f}}(x,-y)+\tilde A(x)$ for $y\geq 0$.
Then, since $(B\circ f^0)(x,y)=f^0(x,-y)$ and $(B\circ\tilde{\bs{f}})(x,y)=\tilde{\bs{f}}(x,-y)$ by the definition of $\tilde{\bs{f}}(x,y)$, we obtain $(B\circ\tilde A)(x)=\tilde A(x)$.
In consequence, we have verified the corollary.

Here, we remark the relation with $F^0(x,y)$ and $F^0(x,-y)$ for $y\geq 0$, explicitly.
For $\phi(x,\pm y)$, the first coordinate elements are equal and the other elements have the different sign.
For each $X^0_{\alpha}(x,\pm y)$, $X^0_{\beta}(x,\pm y)$ and $\xi(x,\pm y)$, the first elements have the different sign and the other elements are equal.
\end{proof}

\section{Structure of the extended curvature surface}\label{sec:structure}

Let $f^0(x,y)$ be a curvature surface on $D=\{(x,y)|\ x>0\}$ and $F^0(x,y)=[\phi,X^0_{\alpha},X^0_{\beta},\xi](x,y)$ be the frame field determining $f^0(x,y)$, defined in the previous section.
In this section, we study the limit of $x$-curves $f^0(x,y)$ with fixed $y$ as $x$ tends to $0$ and $\infty$, and the limit of $y$-curves $f^0(x,y)$ with fixed $x$ as $y$ tends to $\pm\infty$.
For $x$-curves $f^0(x,y)$ as $x$ tends to $0$, we change $x$ for a new parameter $u$ by $x=e^{-u}$: the change is reasonable for the metric $g_0$ on $D$ of \eqref{def:g0} and the equations of Lemma \ref{lemma:dPhi}.
Then, any $u$-curve $f^0(e^{-u},y)$ uniformly converges to a circle $\mathbb{S}^1(y)$ parametrized by $u$ as $u$ tends to $\infty$, and any $y$-curve uniformly converges to a point $p(x)$ as $y$ tends to $\pm\infty$.
Furthermore, the convergence of $u$-curves is uniform with respect to $y\in (-\infty,\infty)$ and the convergence of $y$-curves is also uniform in the wider sense with respect to $x\in (0,\infty)$ (see Definition \ref{lemdef:unifconv} in this section for the uniform convergence).
Thus, $\mathbb{S}^1(y)$ and $p(x)$, respectively, are continuous for $y$ and $x$.
Through further study for these convergences, we can understand the structure on the surface $f^0(x,y)$ in $\R^4$ in detail as mentioned in the introduction, and can connect two curvature surfaces defined on $D$ and $D(-):=\{(x,y)\;|\; x<0\}$ continuously at the origin $(0,0)\in \R^2$ in a sense. In this section, we write $\nabla'_{\partial/\partial x}$ as $\partial/\partial x$, since $\nabla'$ is the canonical connection of $\R^4$. 

Now, let $f^0(x,y)$ be an $x$-curve with fixed $y$ in a curvature surface on $D$.
The $x$-curve $f^0(x,y)$ is included in an affine hyperplane $\R^3_y$ perpendicular to ${\bs{u}}(y)$ by Theorem \ref{thm:xcurve}, and the orthonormal frame field of $\R^3_y$ along the $x$-curve is given by $[X^0_{\alpha},{\bs{f}},{\bs{u}}_2](x,y)$ in \eqref{def:ff-u-X0alpha-f-u2}.
We change the parameter $x$ for $u$ by $x=e^{-u}$, and then, for a vector $Z(x)$ of $x$, we denote by $Z(e^{-u})$ the vector $\bar Z(u):= Z(e^{-u})$ of $u$.
The vectors ${\bs{f}}(e^{-u},y)$ and ${\bs{u}}_2(e^{-u},y)$ satisfy the equations
\begin{gather*}
(\partial{\bs{f}}/\partial u)(e^{-u},y)=-v_3(e^{-u},y)X^0_{\alpha}(e^{-u},y), \ \ 
(\partial{\bs{u}}_2/\partial u)(e^{-u},y)=v_2(e^{-u},y)X^0_{\alpha}(e^{-u},y),
\end{gather*}
by $\partial/\partial u=-x\,\partial/\partial x$ and Lemma \ref{lemma:dPhi}, where
\begin{align*}
v_2(x,y) &:= (5+4y^2)^{-1/2} h^{-1}(x,y)\big((5+4y^2)(X_0+{\sqrt{5}}/{2})+\sqrt{5}(x^2-y^2)\big),\\
v_3(x,y) &:= 2\sqrt{5}(x^2-y^2) h^{-1}(x,y) \sqrt{(1+y^2)(5+4y^2)^{-1}}.
\end{align*}
Hence we have $\partial X^0_{\alpha}/\partial u=v_3{\bs{f}}-v_2{\bs{u}}_2$.
For the functions $v_i(x,y) \ (i=2,3)$, $v_i(0,y)$ are also well defined: we have $v_2(0,0)=\sqrt{5}/2$, $v_2(0,y)>0$ for $y\in \R$ and $v_3(0,0)=0$, $v_3(0,y)<0$ for $y\neq 0$.
As $x$ tends to $0$, the functions $v_2(x,y)$ and $v_3(x,y)$, respectively, converge to  $v_2(0,y)$ and $v_3(0,y)$ uniformly with respect to $y\in \R$.
In fact, we have the following fact in the same way as in Corollary \ref{cor:ineq4h}-(2): there is a number $t_1 \ (0<t_1 <1)$ such that
\begin{gather}
|v_2(x,y)-v_2(0,y)|
<\frac{3}{2} \frac{5+\sqrt{5}+4y^2}{\sqrt{5+4y^2}} \frac{x^2}{h(x,y)}
<\frac{3}{2} \frac{5+\sqrt{5}+4y^2}{\sqrt{5+4y^2}} \frac{x^2}{5+\sqrt{5}+y^2},
\label{ineq:v2}\\
|v_3(x,y)-v_3(0,y)|
<\frac{2\sqrt{5}}{7} \sqrt\frac{1+y^2}{5+4y^2} \frac{x^2(7+y^2)}{h(x,y)}
<\frac{2\sqrt{5}}{7} \sqrt\frac{1+y^2}{5+4y^2} x^2
\label{ineq:v3}
\end{gather}
hold for $0<x<t_1$ and any $y\in \R$.
By the above equations, $M(u,y):=[X^0_{\alpha},{\bs{f}},{\bs{u}}_2](e^{-u},y)$ satisfies the following equation for $(u,y)\in \R^2$:
\begin{gather}\label{diffeq:M}
\frac{\partial M}{\partial u}(u,y) = M(u,y)V(u,y),\quad
V(u,y):=
\begin{bmatrix}
0 & -v_3(x(u),y) & v_2(x(u),y)\\
v_3(x(u),y) & 0 & 0\\
-v_2(x(u),y) & 0 & 0
\end{bmatrix}.
\end{gather}

In $V(u,y)$, we replace the functions $v_2(x,y)$ and $v_3(x,y)$ with $v_2^0(y):=v_2(0,y)$ and $v_3^0(y):=v_3(0,y)$, respectively, and denote by $V^0(y)$ the new $V(u,y)$.
Then, we define another matrix $N(u,y)= \left[{\bs{b}}(u,y),{\bs{a}}(u,y),{\bs{c}}(u,y)\right]$, by the differential equation 
\begin{gather}\label{diffeq:N}
\frac{\partial N}{\partial u}(u,y)
=N(u,y)V^0(y),\quad
V^0(y):=
\begin{bmatrix}
0 & -v^0_3 & v^0_2\\
v^0_3 & 0 & 0\\
-v^0_2 & 0 & 0
\end{bmatrix}
(y).
\end{gather}
\begin{lemma}\label{lemma:v0}
Let $N(u,y)=\left[{\bs{b}},{\bs{a}},{\bs{c}}\right](u,y)$ be a solution to \eqref{diffeq:N}.
Then, ${\bs{v}}^0(y):=(v_2^0{\bs{a}}+v_3^0{\bs{c}})(u,y)$ does not depend on $u$.
The solution $N(u,y)$ is a rotation with respect to the axis ${\bs{v}}^0(y)$ and the speed of its rotation is $\| {\bs{v}}^0(y)\|$.
\end{lemma}
\begin{proof}
In this proof, we fix the parameter $y$ arbitrarily.
We have $\partial(v_2^0{\bs{a}}+v_3^0{\bs{c}})/\partial u=(-v_2^0v_3^0+v_3^0v_2^0){\bs{b}}=0$ by \eqref{diffeq:N}.
Hence, ${\bs{v}}^0(y)$ does not depend on $u$. 
Now, we set
\begin{gather*}
B^0(y):=\frac{1}{\|{\bs{v}}^0(y)\|}
\begin{bmatrix}
0 & \|{\bs{v}}^0\| & 0\\
v^0_2 & 0 & v^0_3\\
v^0_3 & 0 & -v^0_2
\end{bmatrix}
(y),\quad
\Phi_u(y):=
\begin{bmatrix}
1 & 0 & 0\\
0 & \cos(u\|{\bs{v}}^0(y)\|) & -\sin(u\| {\bs{v}}^0(y)\|)\\
0 & \sin(u\|{\bs{v}}^0(y)\|) & \cos(u\| {\bs{v}}^0(y)\|)
\end{bmatrix},
\end{gather*}
and then $B^0(y)$ and $\Phi_u(y)$ belong to the special orthogonal group $\mathrm{SO}(3)$.
Then, we have
\begin{gather}\label{ode:Phiu}
\left(B^0\Phi_u^{-1} \left(\partial\Phi_u/\partial u\right) (B^0)^{-1}\right)(y)=V^0(y)
\end{gather}
by direct calculation.
Next, with any orthogonal matrix $ N^{\infty}(y)$ depending on $y$, the solution $ N(u,y)$ to \eqref{diffeq:N} will be given by
\begin{gather}\label{sol:N}
N(u,y)=N^{\infty}(y)\left(B^0(y)\Phi_u(y)(B^0)^{-1}(y)\right).
\end{gather}
The equation implies that the lemma holds good.
We can verify that $N(x,y)$ in \eqref{sol:N} is a solution to \eqref{diffeq:N} as follows: taking the derivative of $N(u,y)$ by $u$, we obtain
\begin{gather*}
\partial N/\partial u=N^{\infty}B^0\Phi_u(B^0)^{-1}\left(B^0\Phi_u^{-1}(\partial\Phi_u/\partial u)(B^0)^{-1}\right)=NV^0
\end{gather*}
by \eqref{ode:Phiu}.
In consequence, the proof of Lemma \ref{lemma:v0} has been completed.
\end{proof}

Now, we return to the equation \eqref{diffeq:M} for $M$ on $\R^2$. In the following lemma, we use the notations in the proof of Lemma \ref{lemma:v0}, and for a square matrix $A=[a_{ij}]$, we define the norm $\| A\|$ by $\| A\|:=\sqrt{\sum (a_{ij})^2}$.
\begin{lemdef}\label{lemdef:unifconv}
\begin{enumerate}
\item
Let $\bar M(u,y):=M(u,y)\left(B^0\Phi_{-u}(B^0)^{-1}\right)(y)$ for $(u,y)\in \R^2$.
Then, there is an orthogonal matrix $ N^{\infty}(y)$ continuous for $y\in \R$ such that $\bar M(u,y)$ with fixed $y$ converges to $N^{\infty}(y)$ as $u$ tends to $\infty$.
Furthermore, the convergence for $u$-curves $\bar M(u,y)$ is uniform with respect to $y\in \R$.

\item
Let $ N^{\infty}(y)$ be the matrix of $y$ in (1).
Then, the frame field $M(u,y)$ with fixed $y$ uniformly converges to the rotation $N(u,y)=N^{\infty}(y)(B^0\Phi_u(B^0)^{-1})(y)$ as $u$ tends to $\infty$, and further the convergence for $u$-curves $M(u,y)$ is also uniform with respect to $y\in \R$.
\end{enumerate}
Here, as $u$ tends to $\infty$, we say that $M(u,y)$ with fixed $y$ uniformly converges to $N(u,y)$, if there is a real number $U$ for any $\varepsilon>0$ such that $\| M(u,y)-N(u,y)\|<\varepsilon$ holds for $u>U$.
Then, we say that the convergence for $u$-curves $M(u,y)$ is uniform with respect to $y\in \R$, if we can take the above $U$ independently of $y\in \R$ for any $\varepsilon>0$.
\end{lemdef}
\begin{proof}
(1)\ Taking the derivative of $\bar M(u,y)$ by $u$, we have
\begin{align*}
\partial\bar M/\partial u
&=(\partial M/\partial u) B^0\Phi_{-u}(B^0)^{-1}+
MB^0(\partial \Phi_{-u}/\partial u)(B^0)^{-1}\\
&=MV\left(B^0\Phi_{-u}(B^0)^{-1}\right) - M\left(B^0(\partial \Phi_{-u}/\partial (-u))\Phi_u(B^0)^{-1}\right) \left(B^0\Phi_{-u}(B^0)^{-1}\right).
\end{align*}
Then, we have $B^0(\partial \Phi_{-u}/\partial (-u))\Phi_u(B^0)^{-1}=V^0$ by \eqref{ode:Phiu}. Hence, we obtain
\begin{gather}\label{pde:Mtilde}
\partial\bar{M}/\partial u=M(V-V^0)\left(B^0\Phi_{-u}(B^0)^{-1}\right).
\end{gather}

Now, we have
\begin{gather}\label{norm-dMtilde}
\left\|\partial\bar{M}/\partial u(\log(1/x),y)\right\|
= \|V(x,y)-V(0,y)\|
\end{gather}
by \eqref{pde:Mtilde}, and hence we have
\begin{align*}
\|\bar M(u_1,y)-\bar M(u_2,y)\|
\leq \left|\displaystyle\int_{u_1}^{u_2} \| V(x(u),y)-V(0,y)\| du\right|
=\left|\displaystyle\int_{e^{-u_1}}^{e^{-u_2}} \| V(x,y)-V(0,y)\|\frac{dx}{x}\right|
\end{align*}
by \eqref{pde:Mtilde} and \eqref{norm-dMtilde}.
Then, by \eqref{ineq:v2}--\eqref{ineq:v3} there is an orthogonal matrix $ N^{\infty}(y)$ for each $y\in \R$ such that
\begin{gather}\label{ineq:Mtilde}
\|\bar M(\log(1/x),y)-N^{\infty}(y)\|\leq Cx^2
\end{gather}
holds as $x\searrow 0$.
In particular, $N^{\infty}(y)$ is continuous for $y\in \R$, since the continuous matrix $\bar M(u,y)$ of $y$ converges to $N^{\infty}(y)$ uniformly with respect to $y\in \R$ as $x\searrow 0$.

(2) \ We have $\| M(u,y)-N^{\infty}(y)\left(B^0\Phi_u(B^0)^{-1}\right)(y)\|=\| M(u,y)\left(B^0\Phi_{-u}(B^0)^{-1}\right)(y)-N^{\infty}(y)\|\leq Ce^{-2u}$ by \eqref{ineq:Mtilde}.
Then, since $C$ is independent of $y$, we have the assertion.
\end{proof}

In Lemma \ref{lemdef:unifconv}-(2), we put $N^{\infty}(y)=[{\bs{b}}^{\infty},{\bs{a}}^{\infty},{\bs{c}}^{\infty}](y)$ and $N(u,y)=[{\bs{b}},{\bs{a}},{\bs{c}}](u,y)$.
Since $N(u,y)$ is a matrix in \eqref{sol:N} determined by $N^{\infty}(y)$, the vector ${\bs{v}}^0(y):=(v^0_2{\bs{a}}+v^0_3{\bs{c}})(u,y)$
does not depend on $u$ by Lemma \ref{lemma:v0}.
Furthermore, we have $\| {\bs{v}}^0(y)\|=\sqrt{5}/2$ by direct calculation.
We define ${\bs{v}}_1(y):=-(2/\sqrt{5}){\bs{v}}^0(y)$ and another unit vector ${\bs{v}}_2(u,y)$ by
\begin{gather*}
{\bs{v}}_2(u,y):=(2/\sqrt{5})\left(v^0_3(y){\bs{a}}(u,y)-v^0_2(y){\bs{c}}(u,y)\right),
\end{gather*}
which is perpendicular to ${\bs{b}}(u,y)$ and ${\bs{v}}_1(y)$.
Then, we have
\begin{gather}\label{eq:bv2-abcinfty}
\left[{\bs{b}},{\bs{v}}_2\right](u,y)=\left[{\bs{b}}^{\infty},\frac{2}{\sqrt{5}}\left(v^0_3{\bs{a}}^{\infty}-v^0_2{\bs{c}}^{\infty}\right)\right](y)
\begin{bmatrix}
\cos(\sqrt{5}u/2) & -\sin(\sqrt{5}u/2)\\
\sin(\sqrt{5}u/2) & \cos(\sqrt{5}u/2)
\end{bmatrix}
\end{gather}
by \eqref{sol:N}, and
\begin{gather*}
{\bs{a}}(u,y)=\textstyle\frac{2}{\sqrt{5}}\left(-v^0_2(y){\bs{v}}_1(y)+v^0_3(y){\bs{v}}_2(u,y)\right),\quad
{\bs{c}}(u,y)=-\textstyle\frac{2}{\sqrt{5}}\left(v^0_3(y){\bs{v}}_1(y)+v^0_2(y){\bs{v}}_2(u,y)\right)
\end{gather*}
by the definitions of ${\bs{v}}_i \ (i=1,2)$.
Here, any $u$-curve ${\bs{a}}(u,y)$ with fixed $y$ is a circle of center $-(2/\sqrt{5})v^0_2(y){\bs{v}}_1(y)$ and radius $(2/\sqrt{5})|v^0_3(y)|$ by \eqref{eq:bv2-abcinfty}.
In particular, in the case $y=0$, the circle degenerates into one point $-{\bs{v}}_1(0)$ by $v^0_3(0)=0$.

Furthermore, we have $(\partial{\bs{b}}/\partial y)(u,y)=0$.
In fact, as $u$ tends to $\infty$, $X_{\alpha}^0(e^{-u},y)$ and $(\partial X_{\alpha}^0/\partial y)(e^{-u},y)$ with fixed $y$, respectively, converge uniformly to ${\bs{b}}(u,y)$ and zero-vector, and these convergences for $u$-curves are also uniform with respect to $y\in \R$.
Here, the first convergence follows from Lemma \ref{lemdef:unifconv} and the second one follows from $X_0''-X_0'/x=O(x^2)$ in the equation of Lemma \ref{lemma:dPhi}.
Hence, the vectors ${\bs{b}}^{\infty}(y)$ and $(2/\sqrt{5})(v^0_3{\bs{a}}^{\infty}-v^0_2{\bs{c}}^{\infty})(y)$ are constant by \eqref{eq:bv2-abcinfty}: we write these constant vectors as
\begin{gather}\label{def:bcinfty}
{\bs{b}}^{\infty}:={\bs{b}}^{\infty}(y),\quad
-{\bs{c}}^{\infty}:=-{\bs{c}}^{\infty}(0)=(2/\sqrt{5})((v^0_3{\bs{a}}^{\infty}-v^0_2{\bs{c}}^{\infty}))(y)
\end{gather}
by $v^0_2(0)>0$ and $v^0_3(0)=0$.
In consequence, $[{\bs{b}},{\bs{v}}_2](u,y)$ in \eqref{eq:bv2-abcinfty} does not depend on $y$.
Then, since ${\bs{v}}_1(y)$ is perpendicular to ${\bs{u}}(y)$, ${\bs{b}}^{\infty}$ and ${\bs{c}}^{\infty}$, the pair $({\bs{u}}(y),{\bs{v}}_1(y))$ is an orthonormal frame field of the plane perpendicular to the vectors ${\bs{b}}^{\infty}$ and ${\bs{c}}^{\infty}$.
In particular, ${\bs{u}}(y)$ moves on a circle $\mathbb{S}^1$, of which fact we have mentioned in Remark \ref{rem:uutilde}-(2).

By the argument above, we can write $[{\bs{b}},{\bs{v}}_2](u,y)$ as $[{\bs{b}},{\bs{v}}_2](u)$, and then by \eqref{eq:bv2-abcinfty} and \eqref{def:bcinfty} we have
\begin{gather}\label{eq:5.12}
\begin{split}
{\bs{b}}(u)=
\cos(\sqrt{5}u/2){\bs{b}}^{\infty}-\sin(\sqrt{5}u/2){\bs{c}}^{\infty},\quad
{\bs{v}}_2(u)=
-\sin(\sqrt{5}u/2){\bs{b}}^{\infty}-\cos(\sqrt{5}u/2){\bs{c}}^{\infty},\\
{\bs{a}}(u,y)=
(2/\sqrt{5})\left(-v^0_2(y){\bs{v}}_1(y)+v^0_3(y){\bs{v}}_2(u)\right),\quad
{\bs{a}}(u,0)=-{\bs{v}}_1(0).
\end{split}
\end{gather}
In the expression of ${\bs{a}}(u,y)$, $y$ is a parameter for the family $\mathbb{S}^1(y)$ of circles and $u$ is a rotation parameter for each circle $\mathbb{S}^1(y)$: ${\bs{v}}_1(y)$ is perpendicular to the circle $\mathbb{S}^1(y)$.

Next, ${\bs{f}}(e^{-u},y)$ converges uniformly to ${\bs{a}}(u,y)$ as $u$ tends to $\infty$, by Lemma \ref{lemdef:unifconv}.
For each $y\in \R$, let $\Gamma(u,y)$ be the following circle in $\R^4$ with rotation parameter $u$:
\begin{gather}\label{def:Gamma}
\Gamma(u,y):=
\begin{cases}
(2\sqrt{5})^{-1}\sqrt{(5+4y^2)(1+y^2)^{-1}}{\bs{a}}(u,y)+A(y) & y\neq 0,\\
-(1/2){\bs{v}}_1(0)+A(0) & y=0,
\end{cases}
\end{gather}
where $A(y)$ is the $\R^4$-valued function in Theorem \ref{thm:xcurve}-(2).
Each circle $\Gamma(u,y)$ with $y$ has the radius of $2y^2/(\sqrt{5}\,h(0,y))$.

In the following theorem, we use the notations in \eqref{def:bcinfty}--\eqref{def:Gamma}.
\begin{thm}\label{thm:binfcinf}
Let $f^0(x,y)$ be a curvature surface defined on $D$ and $F^0(x,y)=[\phi,X^0_{\alpha},X^0_{\beta},\xi](x,y)$ be the orthonormal frame field determining $f^0(x,y)$.
Let ${\bs{u}}(y)$ and ${\bs{f}}(x,y)$ be the analytic vectors in Theorem \ref{thm:xcurve}.
Then, there is an orthonormal pair $({\bs{b}}^{\infty},{\bs{c}}^{\infty})$ of constant vectors such that it satisfies the following conditions (1), (2) and (3):
\begin{enumerate}
\item
The vector ${\bs{u}}(y)$ moves on the unit circle in a plane perpendicular to ${\bs{b}}^{\infty}$ and ${\bs{c}}^{\infty}$.
Let ${\bs{v}}_1(y)$ be a vector determined by $d{\bs{u}}/dy=(y^2/(1+y^2)){\bs{v}}_1(y)$.
Then, $[{\bs{u}}(y), {\bs{v}}_1(y), {\bs{b}}^{\infty},{\bs{c}}^{\infty}]$ is an orthonormal frame field of $\R^4$ depending on $y$.

\item
As $u$ tends to $\infty$, all $u$-curves $X^0_{\alpha}(e^{-u},y)$ with $y$ uniformly converge to a circle ${\bs{b}}(u)$, and each $u$-curve ${\bs{f}}(e^{-u},y)$ with fixed $y$ uniformly converges to the circle ${\bs{a}}(u,y)$.
These convergences for $u$-curves are also uniform with respect to $y\in \R$.
In particular, the circles ${\bs{a}}(u,y)$ deform continuously for $y$ and the circle ${\bs{a}}(u,0)$ degenerates to one point.

\item
Any $u$-curve $f^0(e^{-u},y)$ with fixed $y$ uniformly converges to the circle $\Gamma(u,y)$ as $u$ tends to $\infty$.
The convergence for $u$-curves is also uniform with respect to $y\in \R$.
In particular, the circles $\Gamma(u,y)$ deform continuously for $y$ and the circle $\Gamma(u,0)$ degenerates to one point.
\end{enumerate}
\end{thm}
\begin{proof}
Almost all facts have been verified in the argument above.
Now, since the vector ${\bs{u}}(y)$ of $y$ is perpendicular to ${\bs{b}}^{\infty}$ and ${\bs{c}}^{\infty}$, we have $(d{\bs{u}}/dy)(y)=\pm(y^2/(1+y^2)){\bs{v}}_1(y)$ for (1) by Remark \ref{rem:uutilde}-(2): we shall determine the sign $\pm$ at the end of this proof.
The convergences in (2) follow from Lemma \ref{lemdef:unifconv}, \eqref{eq:bv2-abcinfty}, \eqref{def:bcinfty} and \eqref{eq:5.12}: the frame field $M(u,y)=[X^0_{\alpha},{\bs{f}},{\bs{u}}_2](e^{-u},y)$ with fixed $y$ uniformly converges to the rotation $N(u,y)=[{\bs{b}},{\bs{a}},{\bs{c}}](u,y)$ as $u\rightarrow \infty$, and the converge for $u$-curves is also uniform with respect to $y\in \R$.
Then, ${\bs{a}}(u,y)$ is analytic for $y$ and ${\bs{a}}(u,0)$ is one point.
We obtain (3) by Theorem \ref{thm:xcurve} and (2).

Now, we verify $(d{\bs{u}}/dy)(y)=(y^2/(1+y^2)){\bs{v}}_1(y)$.
Then, we use the fact (2): $\lim_{x\rightarrow 0}{\bs{f}}(x,0)=-{\bs{v}}_1(0)$.
Firstly, we have 
\begin{gather}\label{eq:5.14}
(d{\bs{u}}/dy)(y)=(1+y^2)^{-3/2}\left[(1+(1+y^2)a_2)(X^0_{\beta}+y\phi)-y(1+y^2)(b_2\xi+c_2X^0_{\alpha})\right](x,y)
\end{gather}
by Lemma \ref{lemma:dPhi}: note that $(d{\bs{u}}/dy)(y)$ is independent of $x$.  
Next, we define ${\bs{p}}(x,0)$ for each $x$ by the limit ${\bs{p}}(x,0):=\lim_{y\rightarrow 0}\ y^{-2}(\text{the right hand side of \eqref{eq:5.14}})$.
Then, we have $\lim_{x\rightarrow 0}{\bs{p}}(x,0)=-\lim_{x\rightarrow 0}{\bs{f}}(x,0)={\bs{v}}_1(0)$ by \eqref{eq:5.14} and the definition \eqref{def:X0} of $X_0(x)$.
In consequence, we obtain the equation desired for all $y\in \R$ by the continuity of $(d{\bs{u}}/dy)(y)$ and ${\bs{v}}_1(y)$.
\end{proof}
\begin{rem}\label{rem:5.1}
For the vector $\tilde{\bs{u}}(x)$ in Theorem \ref{thm:ycurve}, $\langle {\bs{u}}(y),\tilde{\bs{u}}(x)\rangle=0$ holds by Remark \ref{rem:uutilde}-(1).
Hence, $\tilde{\bs{u}}(x)$ is expressed as a linear combination of $X^0_{\alpha}$, ${\bs{f}}$ and ${\bs{u}}_2$: explicitly, we have
\begin{gather*}
\tilde{\bs{u}}(e^{-u})=\frac{1}{\sqrt{B_2^2+C_2^2}(e^{-u})}\left(-B_2X^0_{\alpha}+\frac{C_2}{\sqrt{5+4y^2}}\left({\bs{u}}_2-2\sqrt{1+y^2}\,{\bs{f}}\right)\right)(e^{-u},y),
\end{gather*}
where $B_2$ and $C_2$ are the functions in Lemma \ref{lemma:B-C}.
Then, since $\lim_{u\rightarrow \infty}\tilde{\bs{u}}(e^{-u})=\lim_{u\rightarrow\infty} X^0_{\alpha}(e^{-u},y)$ by $B_2(x)<0$ and $C_2(x)=O(x^2)$, $\tilde{\bs{u}}(e^{-u})$ uniformly converges to the circle ${\bs{b}}(u)=\cos(\sqrt{5}u/2){\bs{b}}^{\infty}-\sin(\sqrt{5}u/2){\bs{c}}^{\infty}$ as $u$ tends to $\infty$.
\end{rem}

Now, for each $y\in \R$ and $0<\varepsilon<1$, let $l^y(x)$ be an $x$-curve given by $l^y(x):=f^0(x,y)$ for $x\in[\varepsilon,1]$ and $L(l^y)$ be the length of $l^y(x)$ by the metric $g_0$.
As $\varepsilon\rightarrow 0$, $L(l^y)$ diverges to $\infty$ if $y\neq 0$ and $L(l^{y=0})$ converges to a finite value.
However, we have the following corollary:
\begin{cor}\label{cor:seq-f0}
Let $\{(x_n,y_n)\}_{n=1}^{\infty}$ be a sequence in $D$ convergent to the origin $(0,0)$ with respect to the Euclidean distance of $\R^2$.
Then, the sequence $f^0(x_n,y_n)$ converges to the point $\Gamma(u,0)=-(1/2){\bs{v}}_1(0)+A(0)$.
That is, when we define $f^0(0,0):=\Gamma(u,0)$, $f^0(x,y)$ on $D$ extends to $D\cup \{(0,0)\}$ continuously in this sense.
\end{cor}
\begin{proof}
Let $(x_n,y_n)\in D$ be a convergent sequence to $(0,0)$ with respect to the Euclidean distance.
Let $u_n=\exp(1/x_n)$ and $\mathbb{S}^1_y:=\{\Gamma(u,y)|u\in \R\}$, where $\mathbb{S}^1_0$ is one point $p_0:=\Gamma(u,0)$.
We define a kind of distance $\bar d(\mathbb{S}^1_y,p_0)$ between $\mathbb{S}^1_y$ and $p_0$ by $\bar d(\mathbb{S}^1_y,p_0):=\Max\{d(p,p_0)\;|\; p\in \mathbb{S}^1_y\}$ with respect to the distance $d$ of $\R^4$.
Then, there is a number $N_1$ for any $\varepsilon>0$ such that $\bar d(\mathbb{S}^1_{y_n},p_0)<\varepsilon$ holds for $n>N_1$, because $\mathbb{S}^1_y$ degenerates to $p_0$ continuously as $y\rightarrow 0$.
Furthermore, $f^0(e^{-u_n},y)$ converges to $\Gamma(u_n,y)$ uniformly with respect to $y\in \R$ as $n\rightarrow \infty$.
Hence, there is a number $N_2$ for any $\varepsilon>0$ such that $\| f^0(e^{-u_n},y)-\Gamma(u_n,y)\|<\varepsilon$ holds for any $y\in \R$ if $n>N_2$.
In consequence, for any $\varepsilon>0$ and $n>\Max\{N_1,N_2\}$ we have
\begin{gather*}
\| f^0(e^{-u_n},y_n)-p_0\|<\| f^0(e^{-u_n},y_n)-\Gamma(u_n,y_n)\|+\bar d(\mathbb{S}^1_{y_n},p_0)<2\varepsilon,
\end{gather*}
which shows that the corollary holds.
\end{proof}
\begin{cor}\label{cor:5.2}
The vector ${\bs{u}}(y)$ determines an orthonormal pair $(\tilde{\bs{b}}{}^{\infty},\tilde{\bs{c}}^{\infty})$ of constant vectors uniquely such that it satisfies the following conditions (1) and (2):
\begin{enumerate}
\item
The system $[{\bs{b}}^{\infty},{\bs{c}}^{\infty},\tilde{\bs{b}}{}^{\infty},\tilde{\bs{c}}^{\infty}]$ is an orthonormal base of $\R^4$.

\item
${\bs{u}}(y)$ and ${\bs{v}}_1(y)$ satisfy the following equations:
\begin{align*}
{\bs{u}}(y)&=-\sin (y-\arctan y)\,\tilde{\bs{b}}{}^{\infty}+\cos (y-\arctan y)\,\tilde{\bs{c}}^{\infty},\\
{\bs{v}}_1(y)&=-\cos(y-\arctan y)\,\tilde{\bs{b}}{}^{\infty}-\sin(y-\arctan y)\,\tilde{\bs{c}}^{\infty}.
\end{align*}
\end{enumerate}
In particular, as $y\rightarrow\infty$, ${\bs{u}}(y)$ and ${\bs{v}}_1(y)$, respectively, converge uniformly to the following circles $\tilde{\bs{b}}(y)$ and $\tilde{\bs{c}}(y)$:\ $\tilde{\bs{b}}(y):=\cos y\,\tilde{\bs{b}}{}^{\infty}+\sin y\,\tilde{\bs{c}}^{\infty},\ \tilde{\bs{c}}(y):=-\sin y\,\tilde{\bs{b}}{}^{\infty}+\cos y\,\tilde{\bs{c}}^{\infty}$.
\end{cor}
\begin{proof}
The unit vector ${\bs{u}}(y)$ belongs to the plane perpendicular to ${\bs{b}}^{\infty}$ and ${\bs{c}}^{\infty}$, and it satisfies the equation $(d{\bs{u}}/dy)(y)=(y^2/(1+y^2)){\bs{v}}_1(y)$.
\end{proof}

Now, for the function $b_1(x,y)$ in Lemma \ref{lemma:dPhi}, we have $b_1(x,y)\rightarrow-\sqrt{5}x/(2X_0-xX'_0)$ as $x\rightarrow\infty$ for any fixed $y$.
Here, $x/(2X_0-xX'_0)$ is a positive and oscillating function for $x\in[t_2,\infty]$ satisfying $(x/(2X_0-xX'_0))(0)=0$ and $x/(2X_0-xX'_0)=O(1)$ as $x\rightarrow\infty$, by Propositions \ref{prop:X0} and \ref{prop:ineq4X0}.
Hence, the integral $w(x):=\sqrt{5}\int_0^x[x/(2X_0-xX_0')]dx$ is an increasing function such that $w(x)=O(x)$ as $x\rightarrow\infty$.
Then, we have the following theorem, where, for $\mathbb{S}^2_y$ and $A(y)$, see Theorem \ref{thm:xcurve}. 
\begin{thm}\label{thm:5.2}
Let $f^0(x,y)$ be a curvature surface defined on $D$ and $F^0(x,y)=[\phi,X^0_{\alpha},X^0_{\beta},\xi](x,y)$ be the orthonormal frame field determining $f^0(x,y)$.
Let ${\bs{n}}(x,y):=(1+y^2)^{-1/2}(X^0_{\beta}+y\phi)(x,y)$.
Then, we have the following facts as $x\rightarrow\infty$:
\begin{enumerate}
\item
Each $x$-curve ${\bs{n}}(x,y)$ for $x\in[1,\infty)$ with fixed $y$ has infinite length and the curve approaches the point $-{\bs{v}}_1(y)$ while winding uniformly around the point.
Then, the curve never reaches $-{\bs{v}}_1(y)$ at any finite $x$.

\item
There is a constant orthonormal base $[\hat{\bs{b}},\hat{\bs{c}}]$ of the plane spanned by ${\bs{b}}^{\infty}$ and ${\bs{c}}^{\infty}$ such that all $x$-curves $\xi(x,y)$ and $X^0_{\alpha}(x,y)$ with $y$, respectively, converge uniformly to the circles expressed as $\bar{\bs{b}}(x)=\cos w(x)\,\hat{\bs{b}}-\sin w(x)\,\hat{\bs{c}}$ and $\bar{\bs{c}}(x)=\sin w(x)\,\hat{\bs{b}}+\cos w(x)\,\hat{\bs{c}}$.

\item
Each $x$-curve $f^0(x,y)$ converges uniformly to the following circle of $\mathbb{S}^2_y$:
\begin{gather*}
-(2\sqrt{5})^{-1}(1+y^2)^{-1/2}
\left({\bs{v}}_1(y)+2(1+y^2)^{1/2}\,\bar{\bs{b}}(x)\right)+A(y).
\end{gather*}
\end{enumerate}
The convergences for the $x$-curves in (2) and (3) are uniform in the wider sense with respect to $y\in \R$.
\end{thm}
\begin{proof}
We consider the $x$-curves ${\bs{n}}(x,y)$ as $x\rightarrow\infty$.
Before the proof, we note that the vector ${\bs{n}}(x,y)$ is perpendicular to ${\bs{u}}(y)$ for any $x$. 
Now, the Propositions \ref{prop:X0} and \ref{prop:ineq4X0} are important for the proof, which give the property of the function $X_0(x)$.

For (1), we firstly have the following equation for any fixed $y\neq 0$\footnote{For $y=0$, see the proof of Theorem \ref{thm:binfcinf}.},
\begin{gather}\label{eq:5.15}
\|{\bs{v}}_1(y)+{\bs{n}}(x,y)\|=O(1/x)
\end{gather}
by \eqref{eq:5.14} and $(d{\bs{u}}/dy)(y)=(y^2/(1+y^2)){\bs{v}}_1(y)$.
Hence, the $x$-curve ${\bs{n}}(x,y)$ converges uniformly to the point $-{\bs{v}}_1(y)$ as $x\rightarrow\infty$.
Then, the equation ${\bs{n}}(x,y)=-{\bs{v}}_1(y)$ does not holds at any finite $x$, since we have $b_2(x,y)\neq 0$ for $(x,y)\in D$ in \eqref{eq:5.14}.
Note that the convergence of the $x$-curve is also uniform with respect to $y$ of any bounded interval $[-y_1,y_1]$, since \eqref{eq:5.15} holds uniformly for $y\in[-y_1,y_1]$.
Next, we have
\begin{gather}\label{eq:5.16}
\partial{\bs{n}}/\partial x=(1+y^2)^{1/2}c_1X^0_{\alpha},\quad
\partial X^0_{\alpha}/\partial x=-(1+y^2)^{1/2}c_1{\bs{n}}-b_1\xi,\quad
\partial \xi/\partial x=b_1X^0_{\alpha}
\end{gather}
by Lemma \ref{lemma:dPhi}.
The length $L(x)$ on $[1,x]$ of the $x$-curve ${\bs{n}}(x,y)$ diverges to $\infty$ as $x\rightarrow\infty$, from $(\partial{\bs{n}}/\partial x)(x,y)=O(1/x)$ by Proposition \ref{prop:ineq4X0}.

Now, we regard ${\bs{n}}(x,y)$ as an analytic curve on the unit $2$-sphere $\mathbb{S}^2$, and then $X^0_{\alpha}$ and $\xi$ are the unit tangent and the unit normal vectors of ${\bs{n}}(x,y)$, respectively.
The curvature $\kappa(x,y)=(1+y^2)^{-1/2}(b_1/c_1)(x,y)$ of ${\bs{n}}(x,y)$ diverges to $\infty$ as $x\rightarrow\infty$, by $\kappa(x,y)=O(x)$.
In consequence, when we take the geodesic $m_x$ on $\mathbb{S}^2$ for each $x$ connecting ${\bs{n}}(x,y)$ with $-{\bs{v}}_1(y)$ and put $t_1:=\inf\{x>t_0\;|\; {\bs{n}}(x,y)\in m_{t_0}\}$ for a fixed $x=t_0$, the length $L(t_1)-L(t_0)$ of the curve ${\bs{n}}(x,y) \ (t_0\leq x\leq t_1)$ converges to $0$ and further the curvature of the curve ${\bs{n}}(x,y) \ (t_0\leq x\leq t_1)$ is almost constant $\kappa(t_0,y)$, which diverges to $\infty$ as $t_0\rightarrow\infty$.
This fact implies that, as $x$ increases, the curve ${\bs{n}}(x,y)$ gradually approaches a smaller and smaller nearly circle of central axis $-{\bs{v}}_1(y)$, which shows (1) (see Figure \ref{fig:xcurve-n} below). 

For (2) and (3), we firstly verify them for an arbitrarily fixed $y$.
For (2): At first, note that every tangent plane at ${\bs{n}}(x,y)$ converges to $T_{-{\bs{v}}_1(y)}\mathbb{S}^2$ uniformly as $x\rightarrow\infty$, by (1).
We regard the normal vector $\xi(x,y)$ and the tangent vector $X^0_{\alpha}(x,y)$ of the curve ${\bs{n}}(x,y)$ as the vectors at $-{\bs{v}}_1(x)$ by the parallel translation of $\R^4$.
Then, these vectors $\xi(x,y)$ and $X^0_{\alpha}(x,y)$ converge uniformly to the curves parametrized by $x$ on the unit circle of $T_{-{\bs{v}}_1(y)}\mathbb{S}^2$ by (1), of which curves we denote by $\bar{\bs{b}}(x,y)$ and $\bar{\bs{c}}(x,y)$, respectively.
Then, we have $\partial\bar{\bs{b}}/\partial x=-w'(x)\bar{\bs{c}}$ and $\partial\bar{\bs{c}}/\partial x=w'(x)\bar{\bs{b}}$ from \eqref{eq:5.16} and the definition of $w(x)$ by $c_1=O(1/x)$, which implies the fact (2) for each $y$.
That is, there is an orthonormal pair $(\hat{\bs{b}}(y),\hat{\bs{c}}(y))$ of vectors depending on $y$ and perpendicular to ${\bs{u}}(y)$ and ${\bs{v}}_1(y)$ such that 
\begin{gather*}
\bar{\bs{b}}(x,y)=\cos w(x)\,\hat{\bs{b}}(y)-\sin w(x)\,\hat{\bs{c}}(y),\quad
\bar{\bs{c}}(x,y)=\sin w(x)\,\hat{\bs{b}}(y)+\cos w(x)\,\hat{\bs{c}}(y)
\end{gather*}
hold (see Figure \ref{fig:xcurve-xi} below). 
The fact (3) follows from (1) and (2) directly by ${\bs{f}}(x,y)=(5+4y^2)^{-1/2}[{\bs{n}}-2(1+y^2)^{1/2}\xi](x,y)$ in Theorem \ref{thm:xcurve}.

Finally, we obtain that $\hat{\bs{b}}(y)$ and $\hat{\bs{c}}(y)$ are constant vectors.
In fact, as $x\rightarrow\infty$, $\partial\xi/\partial y$ and $\partial X^0_{\alpha}/\partial y$ converges to zero-vector uniformly with respect to $y$ of any bounded interval $[-y_1,y_1]$, by Lemma \ref{lemma:dPhi} and Proposition \ref{prop:ineq4X0}.

In consequence, we have completed the proof.
\end{proof}

\noindent
\begin{figure}[H]
\begin{minipage}{0.5\linewidth}
\centering
\includegraphics[width=.9\linewidth]{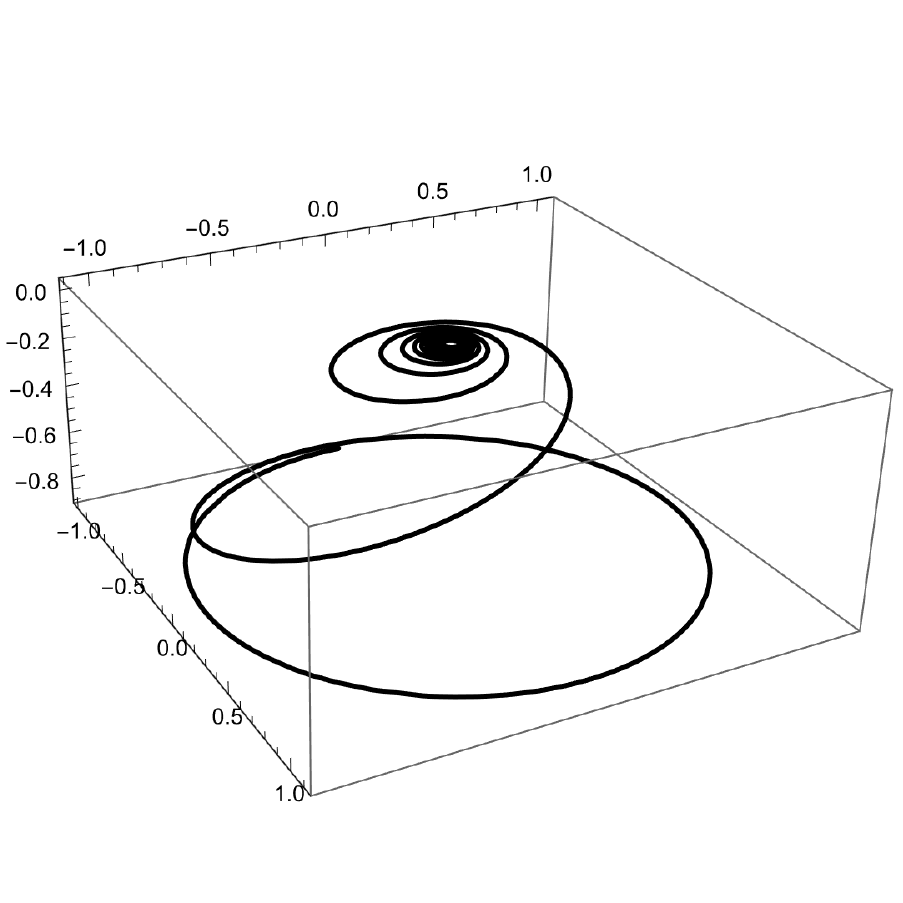}
\vspace{0pt}
\end{minipage}
\begin{minipage}{0.4\linewidth}
\centering
\includegraphics[width=.9\linewidth]{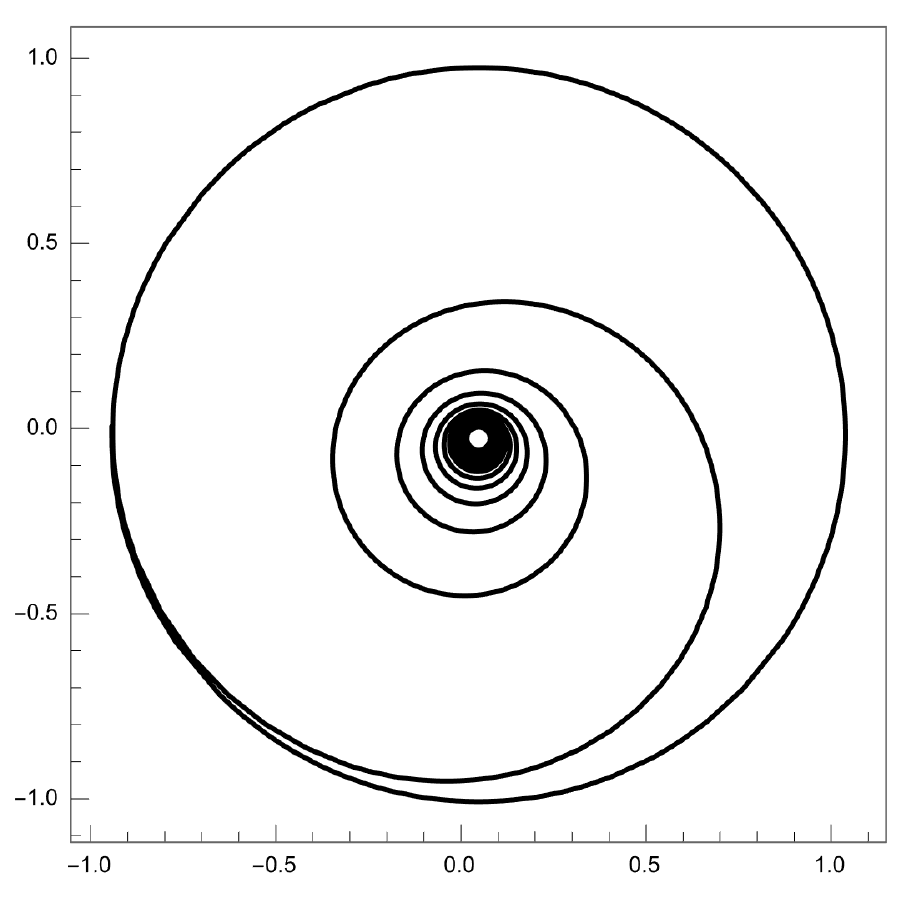}
\end{minipage}
\caption{
\small
These show the $x$-curve $\bs{n}^{1/20}(x,1)$ on $[1/500,100]$: the right hand-side is a bird's-eye view of the left.
\normalsize}
\label{fig:xcurve-n}
\end{figure}

\noindent
\begin{figure}[H]
\begin{minipage}{0.5\linewidth}
\centering
\includegraphics[width=.9\linewidth]{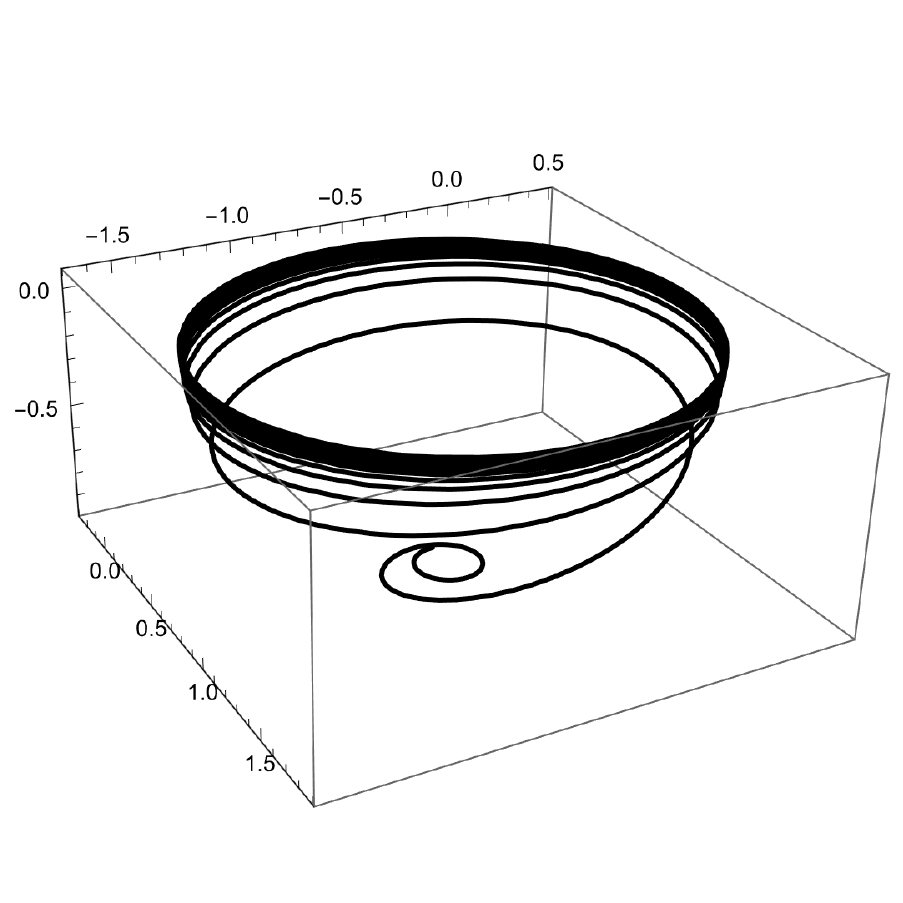}
\vspace{0pt}
\end{minipage}
\begin{minipage}{0.4\linewidth}
\centering
\includegraphics[width=.9\linewidth]{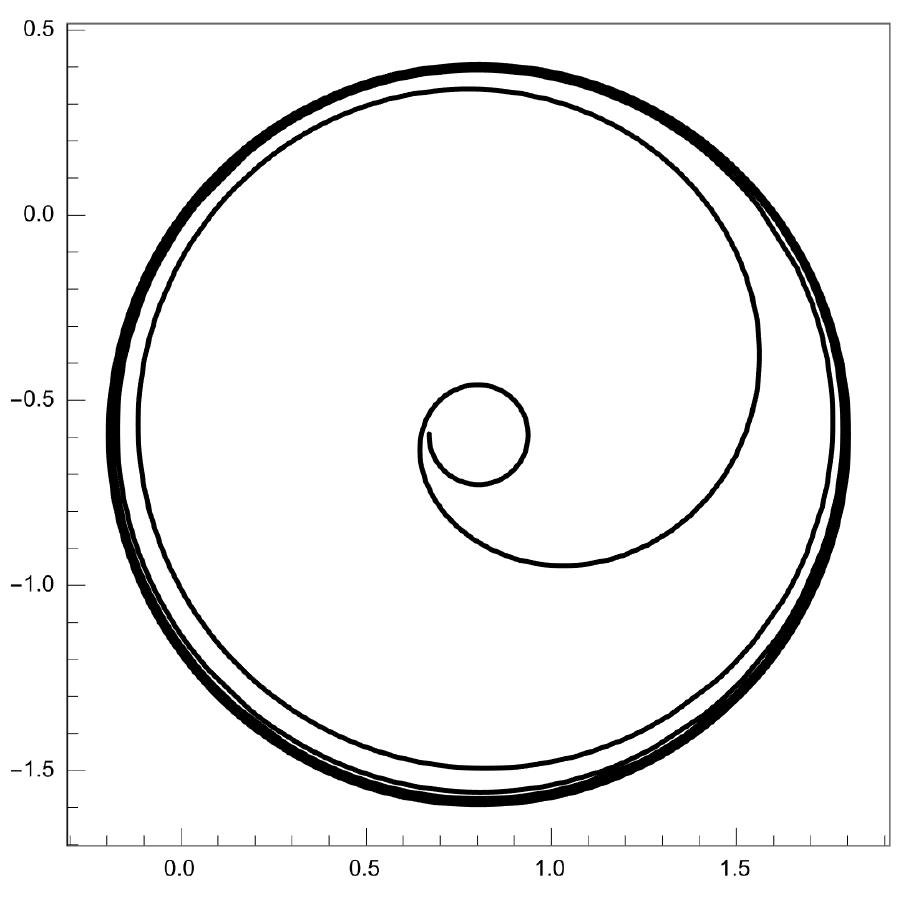}
\end{minipage}
\caption{
\small
These show the $x$-curve $\xi^{1/20}(x,1)$ on $[1/500,100]$: the right hand-side is a bird's-eye view of the left.
\normalsize}
\label{fig:xcurve-xi}
\end{figure}

By Theorem \ref{thm:binfcinf}-(3) and Theorem \ref{thm:5.2}-(3), each $x$-curve $f^0(x,y)$ with $y\neq 0$ converges uniformly to the parallel small circles in $\mathbb{S}^2_y$ as $x \to 0$ and $x \to \infty$. Furthermore, the two points $\pm(2\sqrt{5})^{-1}\sqrt{(5+4y^2)(1+y^2)^{-1}}{\bs{v}}_1(y)+A(y)$ are the antipodal points of $\mathbb{S}^2_y$ to each other and the tangent spaces at these points are spanned by the vectors ${\bs{b}}^{\infty}$ and ${\bs{c}}^{\infty}$.

Next, we study $y$-curves $f^0(x,y)$ with fixed $x$ as $y$ tends to $\infty$.
Along any $y$-curve $f^0(x,y)$, we have an orthonormal frame field $[\tilde{\bs{u}}(x),X^0_{\beta}(x,y), \tilde{\bs{f}}(x,y),\phi(x,y)]$ as in \eqref{def:ff-utilde-X0beta-ftilde-phi}: any $y$-curve lies on $\mathbb{S}^2_x$ of an affine hyperplane $\R^3_x$ perpendicular to $\tilde{\bs{u}}(x)$, where $(\tilde{\bs{u}}(x), \tilde{\bs{f}}(x,y))$ is defined in Theorem \ref{thm:ycurve}.
In particular, the vector $\tilde{\bs{f}}=(B_2^2+C_2^2)^{-1/2}(C_2X^0_{\alpha}+B_2\xi)$ is analytic and $(B_2^2+C_2^2)(x)=(X_0'/x)(2X_0-xX_0'-X_0'/x+\sqrt{5})>0$ holds.
By Lemma \ref{lemma:dPhi}, we have
\begin{gather*}
\partial\tilde{\bs{f}}/\partial y=-{\tilde v}_3X^0_{\beta},\quad
\partial\phi/\partial y={\tilde v}_2X^0_{\beta}~(=-a_2X^0_{\beta}~),\quad
\partial X^0_{\beta}/\partial y=\tilde v_3\tilde {\bs{f}}-\tilde v_2\phi,\quad
\text{where}
\end{gather*}
${\tilde v}_2(x,y):=h^{-1}(x,y)[2X_0+\sqrt{5}-(x^2+y^2)(X_0'/x)]$,\ %
${\tilde v}_3(x,y):=-2(y/h(x,y))(B_2^2+C_2^2)^{1/2}(x)$.
\begin{lemma}\label{lemma:5.3}
\begin{enumerate}
\item
For an arbitrarily fixed $x_1>0$, let $(0,x_1]$ be the bounded interval.
Then, there is a constant $C>0$ independent of $x\in(0,x_1]$ such that $|\tilde v_2(x,y)+1|<Cy^{-2}$ and $|\tilde v_3+q(x)/y|<Cy^{-3}$ hold for $x\in(0,x_1]$ as $y$ tends to $\infty$, where
\begin{gather*}
q(x):=2(x/X_0')(B_2^2+C_2^2)^{1/2}(x).
\end{gather*}

\item
For arbitrarily fixed two numbers $x_1>x_0>0$, let $[x_0,x_1]$ be the bounded interval.
Then, with $T(x)$ given in Remark \ref{rem:uutilde}, there is a constant $C>0$ independent of $x\in[x_0,x_1]$ such that
\begin{gather*}
\|(d\tilde{\bs{u}}/d x)(x)- T(x)\tilde{\bs{f}}(x,y)\| <C/y,\quad
\|(\partial\tilde{\bs{f}}/\partial x)(x,y)+ T(x)\tilde{\bs{u}}(x)\| <C/y
\end{gather*}  
hold for $x\in[x_0,x_1]$ as $y$ tends to $\infty$. 
\end{enumerate}
\end{lemma}
\begin{proof}
We have $h(x,y)=2X_0-xX_0'+y^2(X_0'/x)+\sqrt{5}>\sqrt{5}$ and $B_2^2+C_2^2=(X_0'/x)(2X_0-xX_0'-X_0'/x+\sqrt{5})>0$.
Here, $X_0'/x$ is a bounded positive functions on $\R$ satisfying $(X_0'/x)(0)=2$, and $2X_0-xX_0'$ is also a positive function satisfying $(2X_0-xX_0')(0)=5$ and $(2X_0-xX_0)(x)=O(x)$ as $x\rightarrow\infty$, by Propositions \ref{prop:X0} and \ref{prop:ineq4X0}.
Now, we have
\begin{gather*}
1/h(x,y)=(1/y^2)(x/X_0')-(x/X_0')(2X_0-xX_0'+\sqrt{5})/(y^2h(x,y)).
\end{gather*}
The fact (1) follows from the equation. For (2), we have
\begin{gather*}
T(x)=[(B_2^2+C_2^2)^{-1}(-B_2'C_2+B_2C_2') +\sqrt{5}/X_0'](x).
\end{gather*}
Then, we obtain (2) from Lemma \ref{lemma:dPhi} by direct calculation.
\end{proof}

In the study of $y$-curves $f^0(x,y)$ as $y\rightarrow\infty$, we can not adopt the way of the proof for Theorem \ref{thm:binfcinf} by the fact $\tilde v_3(x)\rightarrow q(x)/y$ in Lemma \ref{lemma:5.3}-(1).

Now, for the function $T(x)$ in Lemma \ref{lemma:5.3}, we define its integral $\tilde w(x)$ by $\tilde w(x):=(\sqrt{5}/2)\log x+\int_0^x[T(x)-(\sqrt{5}/2)(1/x)]dx$,~since $T(x)=(\sqrt{5}/2)(1/x)+O(x)$ as $x\rightarrow0$ and $T(x)=O(1)$ as $x\rightarrow\infty$ by the definition of $X_0(x)$ and Propositions \ref{prop:X0} and \ref{prop:ineq4X0}.

For the function $\tilde w(x)$ and the orthonormal frame $[{\bs{b}}^{\infty},{\bs{c}}^{\infty},\tilde{\bs{b}}{}^{\infty},\tilde{\bs{c}}^{\infty}]$ in Theorem \ref{thm:binfcinf} and Corollary \ref{cor:5.2}, we have the following theorem, where, for $\tilde A(x)$, see Theorem \ref{thm:ycurve}. 
\begin{thm}\label{thm:5.3}
Let $f^0(x,y)$ be a curvature surface defined on $D$ and $F^0(x,y)=[\phi,X^0_{\alpha},X^0_{\beta},\xi](x,y)$ be the orthonormal frame field determining $f^0(x,y)$.
Then, we have the following facts (1)--(3):
\begin{enumerate}
\item
There is a unit vector $\tilde{\bs{v}}_1(x)$ such that $(d\tilde{\bs{u}}/dx)(x)=T(x)\tilde{\bs{v}}_1(x)$ and $(d\tilde{\bs{v}}_1/dx)(x)=-T(x)\tilde{\bs{u}}(x)$ hold. Furthermore, $\tilde{\bs{u}}(x)$ and $\tilde{\bs{v}}_1(x)$, respectively, are expressed as follows:
\begin{gather*}
\tilde{\bs{u}}(x)=\cos \tilde w(x)\,{\bs{b}}^{\infty}+ \sin \tilde w(x)\,{\bs{c}}^{\infty},\quad
\tilde{\bs{v}}_1(x)=-\sin \tilde w(x)\,{\bs{b}}^{\infty}+ \cos \tilde w(x)\,{\bs{c}}^{\infty}.
\end{gather*}

\item
As $y$ tends to $\infty$, each $y$-curve $f^0(x,y)$ with fixed $x$ converges uniformly to the point $(B_2^2+C_2^2)^{-1/2}(x)\tilde{\bs{v}}_1(x)+{\tilde A}(x)$.

\item
As $y$ tends to $\infty$, every $y$-curve $X^0_{\beta}(x,y)$ (resp.\ every $y$-curve $\phi(x,y)$) with $x$ uniformly converges to the circle $\tilde{\bs{b}}(y)=\cos y\,\tilde{\bs{b}}{}^{\infty}+\sin y\,\tilde{\bs{c}}^{\infty}$ (resp.\ $-\tilde{\bs{c}}(y)=\sin y\,\tilde{\bs{b}}{}^{\infty}-\cos y\,\tilde{\bs{c}}^{\infty}$).
\end{enumerate}
Furthermore, the convergences for these $y$-curves in (2) and (3) are uniform in the wider sense with respect to $x\in (0,\infty)$.
\end{thm}
\begin{proof}
Let $(0,x_1]$ or $[x_0,x_1]$ be an arbitrarily fixed bounded interval, as in Lemma \ref{lemma:5.3}.

For (1) and (2): There is a vector $\tilde{\bs{v}}_1(x):=\lim_{y\rightarrow \infty}\tilde{\bs{f}}(x,y)$ on $[x_0,x_1]$ by Lemma \ref{lemma:5.3}-(2) such that $(d\tilde{\bs{u}}/dx)(x)=T(x)\tilde{\bs{v}}_1(x)$ and $(d\tilde{\bs{v}}_1/dx)(x)=-T(x)\tilde{\bs{u}}(x)$ hold.
Then, since $[x_0,x_1]$ is any bounded interval in $(0,\infty)$, these equations hold on $(0,\infty)$.
Hence, there is a constant orthonormal base $[{\bs{b}},{\bs{c}}]$ of the plane spanned by ${\bs{b}}^{\infty}$ and ${\bs{c}}^{\infty}$ such that
\begin{gather*}
\tilde{\bs{u}}(x)=\cos \tilde w(x)\,{\bs{b}}+ \sin \tilde w(x)\,{\bs{c}},\quad
\tilde{\bs{v}}_1(x)=-\sin \tilde w(x)\,{\bs{b}}+ \cos \tilde w(x)\,{\bs{c}}
\end{gather*}
hold, since $\tilde{\bs{u}}(x)$ and ${\bs{u}}(y)$ are perpendicular to each other for all $(x,y)\in D$ by Remark \ref{rem:uutilde}-(1).
Then, $\tilde{\bs{u}}(x)$ converges uniformly to the circle $\cos ((\sqrt{5}/2)\log x){\bs{b}}+ \sin ((\sqrt{5}/2)\log x){\bs{c}}$ as $x\rightarrow 0$.
On the other hand, the circle is expressed as $\cos (\sqrt{5}u/2){\bs{b}}^{\infty}- \sin (\sqrt{5}u/2){\bs{c}}^{\infty}$ by Remark \ref{rem:5.1}.
Hence, we obtain ${\bs{b}}={\bs{b}}^{\infty}$ and ${\bs{c}}={\bs{c}}^{\infty}$ by $u=-\log x$, which shows (1).
Furthermore, since any $y$-curve $f^0(x,y)$ is given by $(B_2^2+C_2^2)^{-1/2}(x)\tilde{\bs{f}}(x,y)+{\tilde A}(x)$, we obtain (2).
  
For (3): For $x\in[x_0,x_1]$, we have ${\bs{u}}(y)=(1+y^2)^{-1/2}(yX^0_{\beta}-\phi)(x,y)=X^0_{\beta}(x,y)+O(1/y)$ and $(d{\bs{u}}/dy)(y)=-\phi(x,y)+O(1/y)$ as $y\rightarrow\infty$.
Furthermore, by Corollary \ref{cor:5.2}, ${\bs{u}}(y)$ and ${\bs{v}}_1(y)$ converge uniformly to the following circles as $y\rightarrow\infty$:
\begin{gather*}
{\bs{u}}(y)\rightarrow \tilde{\bs{b}}(y)=\cos y\,\tilde{\bs{b}}{}^{\infty}+\sin y\,\tilde{\bs{c}}^{\infty},\quad
{\bs{v}}_1(y)\rightarrow \tilde{\bs{c}}(y)=-\sin y\,\tilde{\bs{b}}{}^{\infty}+ \cos y\,\tilde{\bs{c}}^{\infty}.
\end{gather*}
Hence, we obtain (3).
In consequence, the proof has been completed.
\end{proof}
\begin{cor}\label{cor:5.3}
We have $\lim_{y\rightarrow\infty}f^0(x,y)=\lim_{y\rightarrow -\infty}f^0(x,y)$.
\end{cor}
\begin{proof}
For the sake of simplicity, we translate the curvature surface $f^0(x,y)$ such that the obtained surface $f^0(x,y)$ satisfies $f^0(p_0)=\bs{0}$ for some point $p_0=(x_0,0)$, and take the same coordinate system of $\R^4$ as that in Corollary \ref{cor:Af0}: $F^0(p_0)=\mathrm{Id}$, and hence $\phi^0(=\phi(x,0))=(1,0,0,0)^T$ and $[X^0_{\alpha},X^0_{\beta},\xi](p_0)$ is the standard frame of $\R^3_{y=0}=\{(0,x_1,x_2,x_3)^T|x_i\in\R\}$. 

Now, let $B$ be the reflection in Corollary \ref{cor:Af0}. We firstly use the equations $B [\phi,X^0_{\alpha},X^0_{\beta},\xi](x,-y)=[-\phi,X^0_{\alpha},X^0_{\beta},\xi](x,y)$ and $B(f^0(x,-y)):=f^0(x,y)$. Then, $\hat f^0(x,y):=f^0(x,y)$ ($y\geq 0)$ is a curvature surface determined by the frame field $[-\phi,X^0_{\alpha},X^0_{\beta},\xi](x,y)$. Hence, we have $\lim_{y\rightarrow\infty}\hat f^0(x,y)=\lim_{y\rightarrow\infty}f^0(x,y)=(B_2^2+C_2^2)^{-1/2}(x)\tilde{\bs{v}}_1(x)+\tilde A(x)$ by Theorem \ref{thm:5.3}.
Next, by the equations $B(\hat f^0(x,y))=f^0(x,-y)$, $B(\tilde A(x))=\tilde A(x)$ and $B(\tilde{\bs{v}}_1(x))=\tilde{\bs{v}}_1(x)$, we obtain $\lim_{y\rightarrow\infty} f^0(x,-y)=(B_2^2+C_2^2)^{-1/2}(x)\tilde{\bs{v}}_1(x)+\tilde A(x)$. 
Here, $B(\tilde{\bs{v}}_1(x))=\tilde{\bs{v}}_1(x)$ follows from $B((d\tilde{\bs{u}}/dx)(x))=(d\tilde{\bs{u}}/dx)(x)$.
\end{proof}

By Theorem \ref{thm:5.2} and Corollary \ref{cor:5.2}, every tangent space of $\mathbb{S}^2_x$ at the point $\lim_{y\rightarrow\infty}f^0(x,y)$ is spanned by the vectors $\tilde{\bs{b}}{}^{\infty}$ and $\tilde{\bs{c}}^{\infty}$ independently of $x\in (0,\infty)$.

Now, for the centers $A(y)$ and ${\tilde A}(x)$ of $\mathbb{S}^2_y$ and $\mathbb{S}^2_x$ in Theorems \ref{thm:xcurve} and \ref{thm:ycurve}, we have the following facts:
\begin{cor}\label{cor:AAtilde}
(1) The curve $A(y)$ lies on an affine plane spanned by $\tilde{\bs{b}}{}^{\infty}$ and $\tilde{\bs{c}}^{\infty}$.

(2) The curve ${\tilde A}(x)$ lies on an affine plane spanned by ${\bs{b}}^{\infty}$ and ${\bs{c}}^{\infty}$.
\end{cor}
\begin{proof}
We only verify the fact (1), since (2) is obtained in the same way. A curvature surface $f^0(x,y)$ on $D$ is expressed as
\begin{gather*}
f^0(x,y)=(2\sqrt{5}(1+y^2))^{-1}\left(X^0_{\beta}-2(1+y^2)\xi+y\phi\right)+A(y),
\end{gather*}
(as a family of $x$-curves $f^0(x,y)$ with $y$).
The vector $\tilde{\bs{u}}(x)=(B_2^2+C_2^2)^{-1/2}(-B_2X^0_{\alpha}+C_2\xi)(x,y)$ moves on the circle of a plane spanned by ${\bs{b}}^{\infty}$ and ${\bs{c}}^{\infty}$.
Hence, we have only to show that $\langle A'(y),(-B_2X^0_{\alpha}+C_2\xi)(x,y)\rangle=0$ holds for any $(x,y)$.
Now, by $f^0_y(x,y)=(2y/h(x,y))X^0_{\beta}$, we have $A'(y)= (2\sqrt{5}(1+y^2))^{-1}(b_2\xi+c_2X^0_{\alpha})(x,y)$ except for the terms of $X^0_{\beta}(x,y)$ and $\phi(x,y)$.
Then, we have $-c_2B_2+b_2C_2=0$ by Lemmata \ref{lemma:B-C} and \ref{lemma:dPhi}, which implies $\langle A'(y),(-B_2X^0_{\alpha}+C_2\xi)(x,y)\rangle=0$.
\end{proof}

Let $\{\bs{b}^{\infty}, \bs{c}^{\infty}\}_{\R}$ (resp.\ $\{\tilde{\bs{b}}{}^{\infty}, \tilde{\bs{c}}{}^{\infty}\}_{\R}$) be the plane spanned by $\bs{b}^{\infty}$ and $\bs{c}^{\infty}$ (resp. $\tilde{\bs{b}}{}^{\infty}$ and $\tilde{\bs{c}}{}^{\infty}$), including the origin.
Then, there is a curvature surface on $D$ such that the curves $A(y)$ and ${\tilde A}(x)$ for the surface lie in $\{\tilde{\bs{b}}{}^{\infty}, \tilde{\bs{c}}^{\infty}\}_{\R}$ and $\{\bs{b}^{\infty}, \bs{c}^{\infty}\}_{\R}$, respectively.
In fact, let $f^0(x,y)$ be a curvature surface.
By Corollary \ref{cor:AAtilde}, we can denote $A(y)$ and ${\tilde A}(x)$ for $f^0(x,y)$ by $A(y)=B(y)+\bs{p}$ and ${\tilde A}(x)={\tilde B}(x)+\tilde{\bs{p}}$, where $B(y)$ (resp.\ ${\tilde B}(x)$) is a curve in $\{\tilde{\bs{b}}{}^{\infty}, \tilde{\bs{c}}{}^{\infty}\}_{\R}$ (resp.\ in $\{\bs{b}^{\infty}, \bs{c}^{\infty}\}_{\R}$) and $\bs{p}$ (resp.\ $\tilde{\bs{p}}$) is a constant vector of $\{\bs{b}^{\infty}, \bs{c}^{\infty}\}_{\R}$ (resp.\ of $\{\tilde{\bs{b}}{}^{\infty}, \tilde{\bs{c}}{}^{\infty}\}_{\R}$).
Then, by Theorem \ref{thm:xcurve}-(2) and Theorem \ref{thm:ycurve}-(2), we have
\begin{align*}
f^0(x,y) - \bs{p} - \tilde{\bs{p}}
&= (2 \sqrt{5})^{-1}
\sqrt{(5+4y^2)(1+y^2)^{-1}} \bs{f}(x,y) + B(y) - \tilde{\bs{p}}\\
&= (B_2^2+C_2^2)^{-1/2}(x) \tilde{\bs{f}}(x,y) + {\tilde B}(x) - \bs{p}.
\end{align*}
In consequence, for the curvature surface $f^0(x,y)-\bs{p}-\tilde{\bs{p}}$, we have $A(y)=B(y)-\tilde{\bs{p}}$ and ${\tilde A}(x)={\tilde B}(x)-\bs{p}$ by definition, which shows the assertion.

\noindent
\begin{figure}[H]
\renewcommand{\thesubfigure}{\alph{subfigure}}
\hfill
\begin{minipage}{0.45\linewidth}
\centering
\vspace{0pt}
\includegraphics[width=.85\linewidth]{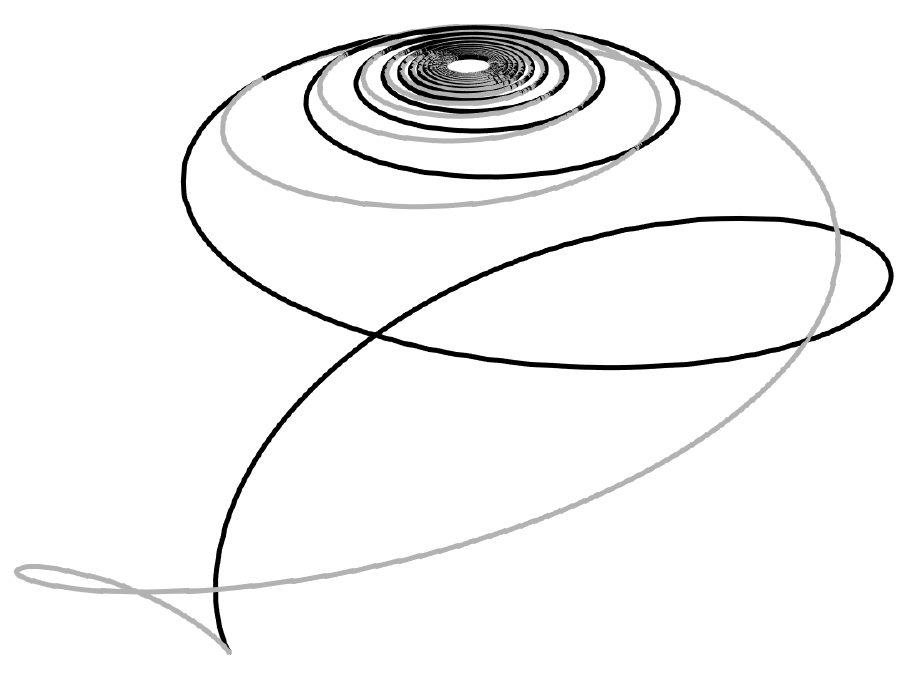}
\subcaption{$y$-curve at $x_0 = 2$.}
\label{fig:ycurvelong-2}
\end{minipage}
\hfill
\begin{minipage}{0.45\linewidth}
\centering
\includegraphics[width=.85\linewidth]{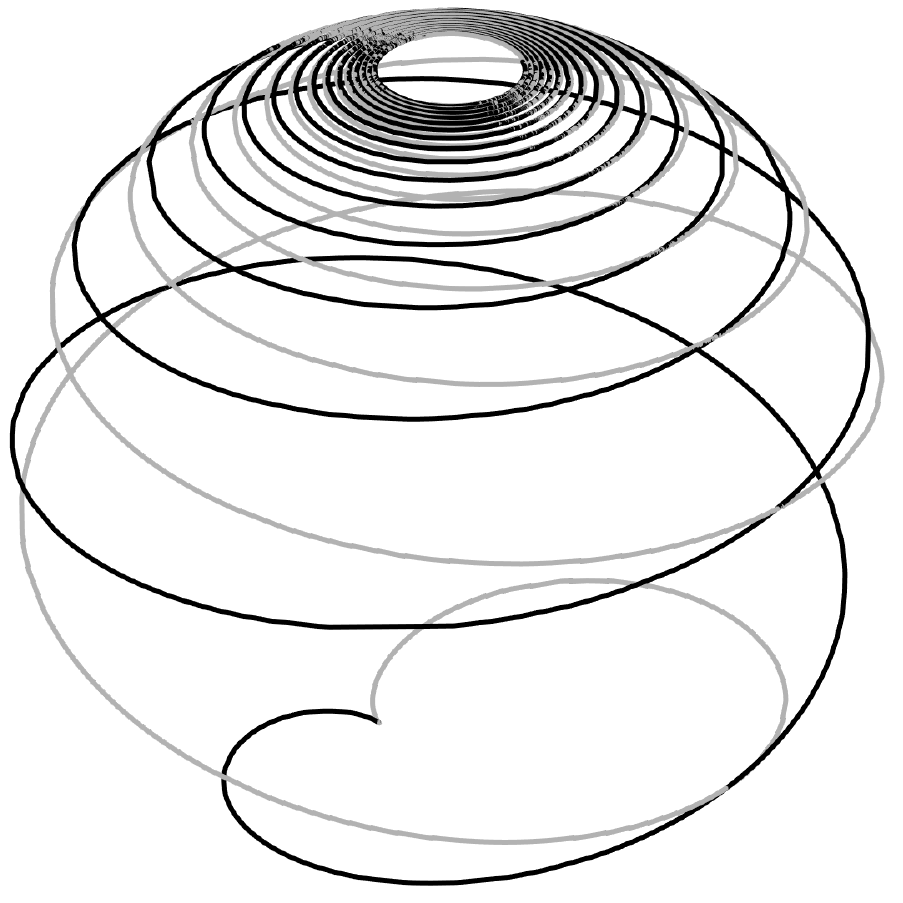}
\subcaption{$y$-curve at $x_0 = 5$.}
\label{fig:ycurvelong-5}
\end{minipage}
\hfill
\caption{
\small
These are the $y$-curves $f^{1/20}(x_0,y)$ on $[-100,100]$.
Each curve is in a sphere $\mathbb{S}^2_{x_0}$, and has a cusp at $y=0$.
We change its color from gray to black at $y=0$.
As for the asymptotic behaviors, each curve converges to one point as $y \to \pm \infty$.
\normalsize
}
\label{fig:ycurvelong}
\end{figure}

\noindent
\begin{figure}[H]
\hfill
\begin{minipage}{0.45\linewidth}
\centering
\includegraphics[width=.85\linewidth]{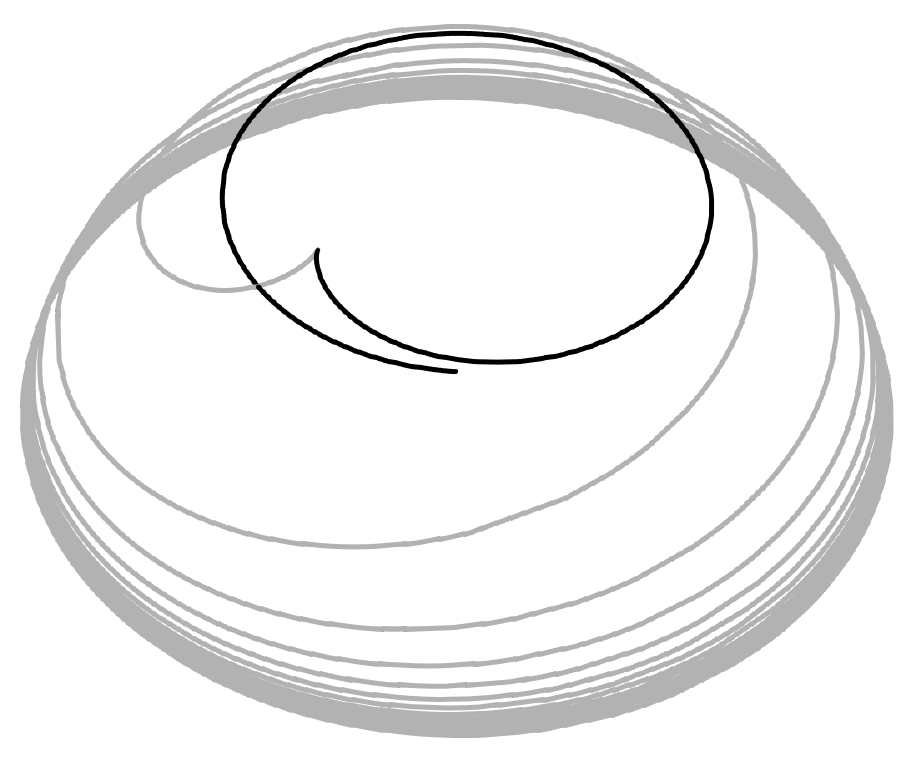}
\end{minipage}
\hfill
\begin{minipage}{0.45\linewidth}
\centering
\includegraphics[width=.85\linewidth]{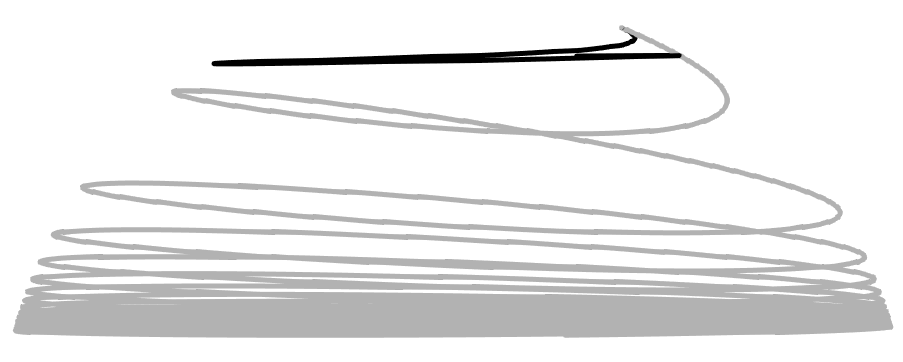}
\end{minipage}
\hfill
\caption{
\small
This shows the $x$-curve $f^{1/20}(x,2)$ on $[1/500,100]$: the figure on the right hand-side is a side view of the figure on the left.
This curve is in a sphere $\mathbb{S}^2_{y=2}$, and has a cusp at $x=2$.
We change its color from black to gray at $x=2$.
As for the asymptotic behaviors, the curve converges uniformly to parallel small circles in $\mathbb{S}^2_{y=2}$ as $x \to 0$ and $x \to \infty$.
\normalsize}
\label{fig:xcurvelong}
\end{figure}

\noindent
\begin{figure}[H]
\hfill
\begin{minipage}{0.45\linewidth}
\centering
\includegraphics[width=.85\linewidth]{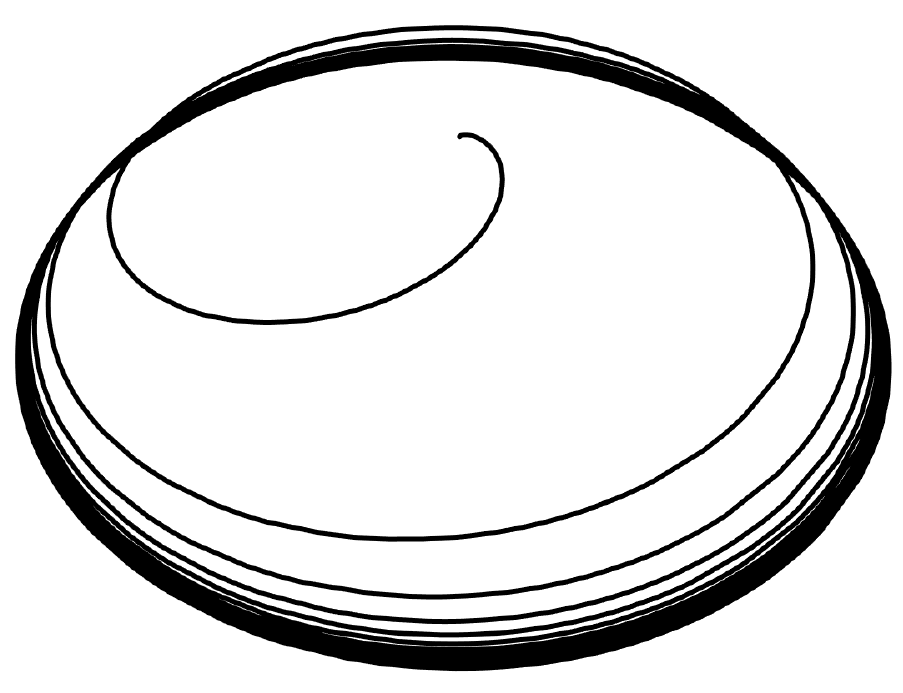}
\end{minipage}
\hfill
\begin{minipage}{0.45\linewidth}
\centering
\includegraphics[width=.85\linewidth]{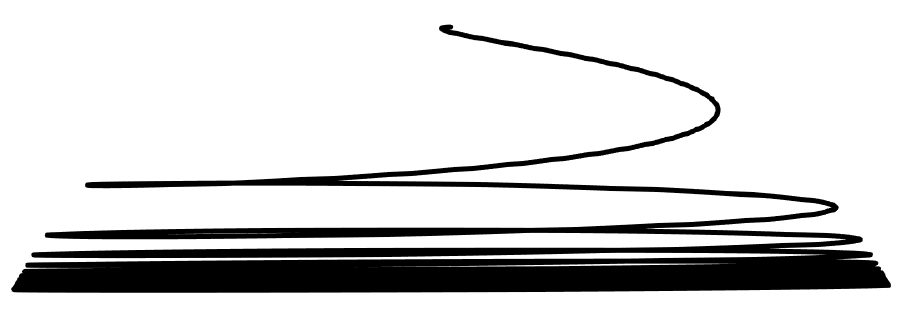}
\end{minipage}
\hfill
\caption{
\small
This shows the $x$-curve $f^{1/20}(x,0)$ on $[1/100,100]$: the figure on the right hand-side is a side view of the figure on the left.
The curve is in a sphere $\mathbb{S}^2_{y=0}$. The curve converges to a point as $x \to 0$, and converges uniformly to a small circle as $x \to \infty$.
This curve has no cusp.
\normalsize}
\end{figure}

In consequence of the all results above, we have obtained the simple structure on the curvature surface $f^0(x,y)$ on $D$, as mentioned in the introduction: in a sense, $f^0(0,y)$ is a curve on the plane $\{\tilde{\bs{b}}{}^{\infty},\tilde{\bs{c}}^{\infty}\}_{\R}$ and $f^0(x,\infty)=f^0(x,-\infty)$ is a curve on the plane $\{{\bs{b}}^{\infty},{\bs{c}}^{\infty}\}_{\R}$; these planes are orthogonal to each other.

At the end of this section, we study a curvature surface defined on $D(-)=\{(x,y)\;|\;x<0\}$ in relation to a curvature surface $f^0(x,y)$ on $D=\{(x,y)\;|\;x>0\}$.
The singular metric $g_0$ in \eqref{def:g0} is defined on $D\cup\{(0,0)\}\cup D(-)$, and the map $j:D\ni (x,y)\leftrightarrow (-x,y)\in D(-)$ is an isometry between the spaces $(D,g_0|_D)$ and $(D(-),g_0|_{D(-)})$.
Although $\bar P$ in \eqref{expression:Pbar} is a negative function on $D(-)$, the equations of Lemma \ref{lemma:B-C} and Lemma \ref{lemma:dPhi} also hold on $D(-)$ for $\Phi$ in \eqref{def:phi} and $\mathcal{P}$ in \eqref{expression:Pbar}, and hence these equations also determine a curvature surface $\hat f^0(x,y)$ defined on $D(-)$.

Now, the following lemma is verified in the same way as the proof of Corollary \ref{cor:Af0}.
\begin{lemma}\label{lemma:F0hat}
Let $F^0(x,y)=[\phi,X^0_{\alpha},X^0_{\beta},\xi](x,y)$ be a frame field on $D$ satisfying the equations of Lemma \ref{lemma:dPhi}.
Then, for $x>0$, the frame field $\hat F^0(-x,y):=F^0(x,y)$ also satisfies the equations in Lemma \ref{lemma:dPhi} on $D(-)$.
\end{lemma}

By Lemma \ref{lemma:F0hat}, a curvature surface $\hat f^0(x,y)$ on $D(-)$ is obtained by $\hat f^0(-x,y):=f^0(x,y)$ from a curvature surface $f^0(x,y)$ on $D$ determined by $F^0(x,y)$.
We regard the surface $\hat f^0(x,y)$ on $D(-)$ as the back side of the surface $f^0(x,y)$ on $D$.
Then, we can define $\hat f^0(0,0)=f^0(0,0):=-(1/2){\bs{v}}_1(0)+A(0)$ by the continuity of $f^0(x,y)$ in the sense at Corollary \ref{cor:seq-f0}. In consequence, we obtain the following corollary.
\begin{cor}
Let $f^0(x,y)$ be a curvature surface on $D$.
When we attach one point $-(1/2){\bs{v}}_1(0)+A(0)$ to the surface formed by both sides of $f^0(x,y)$, the extended surface is curvature on $D\cup\{(0,0)\}\cup D(-)$ with the metric $g_0$.
\end{cor}

\section{Approximation of frame field determining extended curvature surface}\label{sec:approx}

It would be difficult to express explicitly the curvature surfaces in $\R^4$ arising from the metric $\check g_H$ on $\R^2$, by Lemma \ref{lemma:dPhi} and \eqref{def:theta12}. However, in Sections \ref{sec:extension} and \ref{sec:structure}, we have given the approximate figures of curvature lines and the other curves in the curvature surface $f^0(x,y)$ on $D=\{(x,y)\in \R^2\;|\;x>0\}$. 
These figures helped us to capture the surface $f^0(x,y)$ in details: in particular, the entire pictures of the curvature lines were obtained.
In this section, for each positive integer $n$ we construct an approximation $F^{\delta_n}(x,y)$ of the orthonormal frame field $F^0(x,y)$ from the structure equation $dF^0=F^0\Omega$ in \eqref{eq:dF0}, which induces these approximate figures, in particular, for the coordinate lines by \eqref{eq:f0asxcurves} and \eqref{eq:f0asycurves}. 

Now, let $F^0(x,y)=[\phi,X^0_{\alpha},X^0_{\beta},\xi](x,y)$ be the orthonormal frame field determining a curvature surface $f^0(x,y)$ on $D$.
Then, we regard $\phi(x,y)$ as a (singular) surface in the standard unit $3$-sphere $\mathbb{S}^3$: $X^0_{\alpha}(x,y)$ and $X^0_{\beta}(x,y)$ are the principal curvature vectors and $\xi(x,y)$ is a normal vector field of $\phi$.
The frame $F^0(x,y)$ satisfies the following equations with the analytic functions $a_i,b_i,c_i$ on $D$ by Lemma \ref{lemma:dPhi}:
\begin{gather}\label{eq:dphi}
\begin{split}
d\phi=-(a_1dx)X^0_{\alpha}-(a_2dy)X^0_{\beta},&\quad
d\xi=(b_1dx)X^0_{\alpha}+(b_2dy)X^0_{\beta},\\
\partial X^0_{\beta}/\partial x=c_1X^0_{\alpha},&\quad
\partial X^0_{\alpha}/\partial y=c_2X^0_{\beta}.
\end{split}
\end{gather}

Then, we have the following lemma:
\begin{lemma}\label{lemma:pde-abc}
The functions $a_i,b_i$ and $c_i$ satisfy the following equations on $D$:
\begin{gather*}
(a_1)_y=a_2c_1,\quad
(a_2)_x=a_1c_2,\quad
(b_1)_y=b_2c_1,\quad
(b_2)_x=b_1c_2,\\
(c_2)_x+(c_1)_y+a_1a_2+b_1b_2=0.
\end{gather*}
\end{lemma}
\begin{proof}
These equations are equivalent to the Maurer-Cartan equation $d\Omega+\Omega\wedge\Omega=0$.
Here, we give another proof.
The first four equations follow from \eqref{eq:dphi} by $d(d\phi)=d(d\xi)=0$.
For the last equation, suppose firstly that $\phi$ is a surface in $\mathbb{S}^3$.
Then, the Gauss curvature $K_{\phi}$ of $\phi$ is given by
\begin{gather*}
K_{\phi}=-(a_1a_2)^{-1}\big[\left((a_2)_x/a_1\right)_x+\left((a_1)_y/a_2\right)_y\big]
=1+\lambda_1\lambda_2,
\end{gather*}
where $\lambda_1:=b_1/a_1$ and $\lambda_2:=b_2/a_2$ are the principal curvatures of $\phi$.
Thus, the lemma holds good if $\phi$ is a surface.
Next, since the functions $a_i,b_i,c_i$ are analytic on $D$ and the domain for $\phi$ to be a surface is open, the last equation also holds on $D$.
\end{proof}

Now, for a point $(x_0,y_0)\in D$ and a constant $a(>0)$, let $[x_0,x_0+a]$ and $[y_0,y_0+a]$ be the closed intervals and $E:=[x_0,x_0+a]\times [y_0,y_0+a]$ be a compact domain in $D$.
From now on, we write $x_e:=x_0+a$ and $y_e:=y_0+a$.
For the domain $E$, let us fix an orthonormal frame $F^0(x_0,y_0)$ at $(x_0,y_0)$ arbitrarily.
For a point $(x_p,y_p)\in E$ and an integer $n>0$, we put $E_p:=[x_0,x_p]\times [y_0,y_p]$ and $\delta_n:=a/n$.
\begin{defn}[Division of the domain $E_p$ and Path]\label{def:path}
(1)
For a point $(x_p,y_p)\in E$, we divide the intervals $[x_0,x_p]$ and $[x_0,y_p]$, respectively, into sub-intervals of equal length:
\begin{gather*}
x_0<x_1<\dots <x_s=x_p,\quad
y_0<y_1<\dots <y_t=y_p.
\end{gather*}
Here, we take the integers $s$ and $t$ such that $0<s,t\leq n$ and $(x_p-x_0)/s\leq \delta_n$, $(y_p-y_0)/t\leq \delta_n$.
Then, we denote $\Delta_{i,j}^n:=[x_i,x_{i+1}]\times [y_j,y_{j+1}] \ (\subset E_p)$.

(2) For two lattice points $P_{i,j}:=(x_i,y_j)$ and $P_{i+k,j+l}:=(x_{i+k},y_{j+l})$ in a division of $E_p$, where $k\geq 0$ and $l\geq 0$, let $m$ be a polygonal line in the division connecting the two points.
We express $m$ as
\begin{gather*}
m: P_{i,j}\rightarrow \dots \rightarrow P_{a,b}\rightarrow P_{c,d}\rightarrow \dots \rightarrow P_{i+k,j+l}
\end{gather*}
by pointing to lattice points through which $m$ passes, in order.
Then, we call $m$  a path from $P_{i,j}$ to $P_{i+k,j+l}$ if $c\geq a$ and $d\geq b$ are satisfied.
\end{defn}

Now, let us fix a division of $E_p$ for $(x_p,y_p)\in E$.
Under $F^{\delta_n}(x_0,y_0):=F^0(x_0,y_0)$, we construct an approximation $F^{\delta_n}(x,y)$ of $F^0(x,y)$ on any path from $(x_0,y_0)$ to $(x_p,y_p)$.
At the beginning, we state the program for the construction of $F^{\delta_n}(x,y)$.

\textbf{Step A}: We take a sub-domain $\Delta_{i,j}^n$ arbitrarily, and suppose that an orthonormal frame $F^{\delta_n}(x_i,y_j)$ at $P_{i,j}$ is determined.

(1)\;
We construct an orthonormal frame field $F^{\delta_n}(x,y_j)$ on $[x_i,x_{i+1}]\times \{y_j\}$ from $F^{\delta_n}(x_i,y_j)$.

(2)\;
We construct an orthonormal frame field $F^{\delta_n}(x_i,y)$ on $\{x_i\}\times[y_j,y_{j+1}]$ from $F^{\delta_n}(x_i,y_j)$.

In the constructions of $F^{\delta_n}(x_{i+1},y_j)$ and $F^{\delta_n}(x_i,y_{j+1})$ in (1) and (2), suppose that we start from the other orthonormal frame $\bar F^{\delta_n}(x_i,y_j)$ at $P_{i,j}$.
Then, we obtain distinct frames $\bar F^{\delta_n}(x_{i+1},y_j)$ and $\bar F^{\delta_n}(x_i,y_{j+1})$ from $F^{\delta_n}(x_{i+1},y_j)$ and $F^{\delta_n}(x_i,y_{j+1})$, respectively.
In this case, these frames will satisfy the equations
\begin{gather}\label{eqns:norm-Fdeltas}
\|(F^{\delta_n}-\bar F^{\delta_n})(x_i,y_j)\|
=\|(F^{\delta_n}-\bar F^{\delta_n}) (x_{i+1},y_j)\|
=\|(F^{\delta_n}-\bar F^{\delta_n}) (x_i,y_{j+1})\|,
\end{gather}
where $\| A\|=\sqrt{\sum (a_{ij})^2}$ for a square matrix $A=[a_{ij}]$, as in the previous section.

(3)\;
From $F^{\delta_n}(x_{i+1},y_j)$ obtained by (1), we firstly have an orthonormal frame field $F^{\delta_n}(x_{i+1},y)$ on $\{x_{i+1}\}\times[y_j,y_{j+1}]$ by (2): we denote by $F^{\underline{\delta}_n}(x_{i+1},y_{j+1})$ the frame at $P_{{i+1},{j+1}}$ determined in this way.
Next, from $F^{\delta_n}(x_i,y_{j+1})$ obtained by (2), we have an orthonormal frame field $F^{\delta_n}(x,y_{j+1})$ on $[x_i,x_{i+1}]\times \{y_{j+1}\}$ by (1): we denote by $F^{\overline{\delta}_n}(x_{i+1},y_{j+1})$ the frame at $P_{{i+1},{j+1}}$ determined in this way.
That is, the two frames $F^{\underline{\delta}_n}(x_{i+1},y_{j+1})$ and $F^{\overline{\delta}_n}(x_{i+1},y_{j+1})$ are determined at $P_{i+1,j+1}$.
Then, although $F^{\underline{\delta}_n}(x_{i+1},y_{j+1})$ and $F^{\overline{\delta}_n}(x_{i+1},y_{j+1})$ do not coincide, we shall have the following inequality,
\begin{gather}\label{ineq:norm-Fdeltas}
\| (F^{\underline{\delta}_n}-F^{\overline{\delta}_n})(x_{i+1},y_{j+1})\| \ \leq \ K(\delta_n)^3,
\end{gather}
if $n$ is sufficiently large, where $K$ is a constant on $E$ determined independently of $n$ (see Definition \ref{def:K} below).
In the proof of \eqref{ineq:norm-Fdeltas}, we shall use Lemma \ref{lemma:pde-abc} essentially.

\textbf{Step B}: Let $P_{i,j}$ be a lattice point of $E_p$ and $m$ be a path from $P_{0,0}$ to $P_{i,j}$.
By Step A, we have an orthonormal frame field $F^{\delta_n}$ on the path $m$.

\begin{prop}\label{prop:norm-Fdeltas}
Let $P:=(x_p,y_p)\in E$ and $F^0(x_0,y_0)$ be an orthogonal matrix fixed arbitrarily. 
For an integer $n$, we take a division of the domain $E_p$.
Let $m_i \ (i=1,2)$ be two paths from $P_{0,0}$ to $P$ in the division, and $F^{\delta_n}_i(x_p,y_p) \ (i=1,2)$ be the frames at $P$ determined from $m_i$, respectively.
Then, we have
\begin{gather*}
\| (F^{{\delta}_n}_1-F^{{\delta}_n}_2)(x_p,y_p)\| \leq Ka^2\delta_n,
\end{gather*}
if $n$ is sufficiently large.
\end{prop}
\begin{proof}
In this proof, we suppose that \eqref{eqns:norm-Fdeltas} and \eqref{ineq:norm-Fdeltas} hold good and that $n$ is sufficiently large.
Now, under the assumption that a frame $F^{\delta_n}(x_i,y_j)$ is determined at a point $P_{i,j} \ (i\leq s-2, \ j\leq t-1)$, we study the frame at $P_{i+2,j+1}$.
In this case, there are three paths $m_i \ (i=1,2,3)$ from $P_{i,j}$ to $P_{i+2,j+1}$.
For each $m_i$, we point to the lattice points only on the way, and then we have
\begin{gather*}
m_1: P_{i+1,j}\rightarrow P_{i+2,j},\quad
m_2: P_{i+1,j}\rightarrow P_{i+1,j+1},\quad
m_3: P_{i,j+1}\rightarrow P_{i+1,j+1}.
\end{gather*}
Let $F^{\delta_n}_i(x_{i+2},y_{j+1})$ be the frames at $P_{i+2,j+1}$ determined from $m_i$, respectively.
Then, we firstly have $\| (F^{\delta_n}_2-F^{\delta_n}_1)(x_{i+2},y_{j+1})\| \leq  K\delta_n^3$ and $\| (F^{\delta_n}_3-F^{\delta_n}_2)(x_{i+2},y_{j+1})\| \leq  K\delta_n^3$ by \eqref{eqns:norm-Fdeltas} and \eqref{ineq:norm-Fdeltas}.
Furthermore, the closed polygonal line $m_3-m_1$ includes two sub-domains $\Delta_{i,j}^n$ and $\Delta_{i+1,j}^n$, and we have
\begin{gather*}
\|(F^{\delta_n}_3-F^{\delta_n}_1)(x_{i+2},y_{j+1})\|\\
\leq\| (F^{\delta_n}_3-F^{\delta_n}_2)(x_{i+2},y_{j+1})\|+\| (F^{\delta_n}_2-F^{\delta_n}_1)(x_{i+2},y_{j+1})\|
\leq 2 K\delta_n^3.
\end{gather*}

Next, we arbitrarily take two paths $m_i$ from $P_{0,0}$ to $P=(x_p,y_p)$.
Then, the closed polygonal line $m_1-m_2$ includes at most $n^2$ sub-domains $\Delta^n_{k,l}$.
In consequence, we obtain the lemma by \eqref{eqns:norm-Fdeltas}, \eqref{ineq:norm-Fdeltas} and the above fact.
\end{proof}

\textbf{Step C}: For a given integer $n$ and a point $(x_p,y_p)\in E$, we take a division of $E_p$.
For the division, let $\underline{m}$ and $\overline{m}$ be two paths  defined by
\begin{align}
\underline{m}:& P_{0,0}=(x_0,y_0)\rightarrow P_{s,0}=(x_p,y_0)\rightarrow P_{s,t}=(x_p,y_p),
\label{def:mlower}\\
\overline{m}:& P_{0,0}=(x_0,y_0)\rightarrow P_{0,t}=(x_0,y_p)\rightarrow P_{s,t}=(x_p,y_p).
\label{def:mupper}
\end{align}
Let $F^{\underline\delta_n}(x_p,y_p)$ and $F^{\overline\delta_n}(x_p,y_p)$ be the frames at $(x_p,y_p)$ determined from $\underline{m}$ and $\overline{m}$, respectively.
Then, we shall verify that $F^{\underline\delta_n}(x_p,y_p)$ converges to $F^0(x_p,y_p)$ uniformly for any $(x_p,y_p)\in E$ as $n$ tends to $\infty$, where $F^0(x,y)$ is the frame field with a given $F^0(x_0,y_0)$.
In consequence, $F^{\overline\delta_n}(x_p,y_p)$ also converges to the frame $F^0(x_p,y_p)$ uniformly for any $(x_p,y_p)\in E$ as $n$ tends to $\infty$, since $\|(F^{\underline\delta_n}-F^{\overline\delta_n})(x_p,y_p) \|\leq Ka^2\delta_n$ holds by Proposition \ref{prop:norm-Fdeltas}.\\

Now, we verify Steps A and C.
In these proofs, we fix an integer $n$, and study only the case of $E=[x_0,x_e]\times [y_0,y_e]$ for $E_p$, by Definition \ref{def:path}.
Hence, we have $s=t=n$, and denote $\delta:=\delta_n=a/n,$ $\Delta_{i,j}:=\Delta_{i,j}^n$ and $F^{\delta}:=F^{\delta_n}$.

For Step A: Under the assumption that an orthonormal frame $F^{\delta}(x_i,y_j)$ at $P_{i,j}$ is determined, we construct $F^{\delta}$ on boundary of the sub-domain $\Delta_{i,j}$ according to the ways (1), (2) and (3), and verify \eqref{eqns:norm-Fdeltas} and \eqref{ineq:norm-Fdeltas}.
Let $F^{\delta}=:[\phi^{\delta},X^{\delta}_{\alpha},X^{\delta}_{\beta},\xi^{\delta}]$.
We simply write $(x_i,y_j)$ as $(0,0)$ in the following Steps A-(1) and A-(2).

Step A-(1): For a given $F^{\delta}(0,0)$, we express $F^{\delta}(x,y_j)$ for $x\in [x_i,x_{i+1}]$ as
\begin{align}
\phi^{\delta}(x,y_j)-\phi^{\delta}(0,0)
&:=\label{def:phidelta-x}
-\textstyle\frac{x-x_i}{2}a_1(0,0)\left(X^{\delta}_{\alpha}(0,0)+X^{\delta}_{\alpha}(x,y_j)\right),\\
\xi^{\delta}(x,y_j)-\xi^{\delta}(0,0)
&:=\label{def:xidelta-x}
\textstyle\frac{x-x_i}{2}b_1(0,0)\left(X^{\delta}_{\alpha}(0,0)+X^{\delta}_{\alpha}(x,y_j)\right),\\
X^{\delta}_{\beta}(x,y_j)-X^{\delta}_{\beta}(0,0)
&:=\label{def:Xbetadelta-x}
\textstyle\frac{x-x_i}{2}c_1(0,0)\left(X^{\delta}_{\alpha}(0,0)+X^{\delta}_{\alpha}(x,y_j)\right),
\end{align}
and find out a suitable $X^{\delta}_{\alpha}(x,y_j)$ from these equations.
Now, let $s^{\delta}_1(x,y_j)$ and $t^{\delta}_1(x,y_j)$ be the functions on $(x,y_j)\in [x_i,x_{i+1}]\times \{y_j\}$ defined by
\begin{gather}\label{def:sdelta1}
s^{\delta}_1(x,y_j):=\langle X^{\delta}_{\alpha}(x,y_j),X^{\delta}_{\alpha}(0,0)\rangle,\quad
t^{\delta}_1(x,y_j):=\left(1+s^{\delta}_1(x,y_j)\right)/2,
\end{gather}
respectively.
\begin{lemma}\label{lemma:s1t1}
Suppose that the frame field $F^{\delta}(x,y_j)$ in \eqref{def:phidelta-x}--\eqref{def:Xbetadelta-x} is orthonormal at each point $(x,y_j)\in[x_i,x_{i+1}]\times \{y_j\}$.
Then, for $x\in [x_i,x_{i+1}]$ we have the following equations:
\begin{align*}
\langle X^{\delta}_{\alpha}(x,y_j),\phi^{\delta}(0,0)\rangle
=-\langle \phi^{\delta}(x,y_j),X^{\delta}_{\alpha}(0,0)\rangle
&=a_1(0,0)(x-x_i)t^{\delta}_1(x,y_j),\\
\langle X^{\delta}_{\alpha}(x,y_j),\xi^{\delta}(0,0)\rangle
=-\langle \xi^{\delta}(x,y_j),X^{\delta}_{\alpha}(0,0)\rangle
&=-b_1(0,0)(x-x_i)t^{\delta}_1(x,y_j),\\
\langle X^{\delta}_{\alpha}(x,y_j),X^{\delta}_{\beta}(0,0)\rangle
=-\langle X^{\delta}_{\beta}(x,y_j),X^{\delta}_{\alpha}(0,0)\rangle
&=-c_1(0,0)(x-x_i)t^{\delta}_1(x,y_j),
\end{align*}
\begin{align*}
s^{\delta}_1(x,y_j)=\frac{4-(x-x_i)^2(a_1^2+b_1^2+c_1^2)(0,0)}{4+(x-x_i)^2(a_1^2+b_1^2+c_1^2)(0,0)},\quad
t^{\delta}_1(x,y_j)=\frac{4}{4+(x-x_i)^2(a_1^2+b_1^2+c_1^2)(0,0)}.
\end{align*}
\end{lemma}
\begin{proof}
Let $x\in [x_i,x_{i+1}]$.
Suppose that $F^{\delta}(x,y_j)$ in \eqref{def:phidelta-x}--\eqref{def:Xbetadelta-x} is an orthonormal frame field at each point of $[x_i,x_{i+1}]\times\{y_j\}$.
Firstly, we verify the first equation under the assumption $a_1(0,0)\neq 0$.
By \eqref{def:phidelta-x}--\eqref{def:Xbetadelta-x}, we have
\begin{align*}
&-((x-x_i)/2)\,a_1(0,0)\,\langle X^{\delta}_{\alpha}(x,y_j),\phi^{\delta}(0,0)\rangle\\
&=\langle \phi^{\delta}(x,y_j)-\phi^{\delta}(0,0),\phi^{\delta}(0,0)\rangle
=\langle \phi^{\delta}(x,y_j),\phi^{\delta}(0,0)\rangle -1,
\end{align*}
and
\begin{align*}
&-((x-x_i)/2)\,a_1(0,0)\,\langle \phi^{\delta}(x,y_j),X^{\delta}_{\alpha}(0,0)\rangle\\
&=\langle \phi^{\delta}(x,y_j)-\phi^{\delta}(0,0),\phi^{\delta}(x,y_j)\rangle
=1-\langle \phi^{\delta}(x,y_j),\phi^{\delta}(0,0)\rangle.
\end{align*}
Hence, we have
$\langle X^{\delta}_{\alpha}(x,y_j),\phi^{\delta}(0,0)\rangle
=-\langle \phi^{\delta}(x,y_j),X^{\delta}_{\alpha}(0,0)\rangle$ by $a_1(0,0)\neq 0$. Furthermore, since
\begin{align*}
-((x-x_i)/2)\,a_1(0,0)\,(1+s^{\delta}_1(x,y_j))
=\langle \phi^{\delta}(x,y_j)-\phi^{\delta}(0,0),X^{\delta}_{\alpha}(0,0)\rangle
=\langle \phi^{\delta}(x,y_j),X^{\delta}_{\alpha}(0,0)\rangle,
\end{align*}
the first equation is obtained:
\begin{gather*}
\langle X^{\delta}_{\alpha}(x,y_j),\phi^{\delta}(0,0)\rangle
=-\langle \phi^{\delta}(x,y_j),X^{\delta}_{\alpha}(0,0)\rangle
=\textstyle\frac{x-x_i}{2} a_1(0,0) \left(1+s^{\delta}_1(x,y_j)\right).
\end{gather*}
In the case $a_1(0,0)=0$, the equation is also satisfied from \eqref{def:phidelta-x}--\eqref{def:Xbetadelta-x} by $\phi^{\delta}(x,y_j)=\phi^{\delta}(0,0)$.
In the same way, we can verify the equations for $\langle X^{\delta}_{\alpha}(x,y_j),\xi^{\delta}(0,0)\rangle$ and $\langle X^{\delta}_{\alpha}(x,y_j),X^{\delta}_{\beta}(0,0)\rangle$.

Next, from these equations verified above, we have
\begin{gather}\label{eq:Xalphadelta-x}
X^{\delta}_{\alpha}(x,y_j)
=s^{\delta}_1(x,y_j)X^{\delta}_{\alpha}(0,0)
+\textstyle\frac{x-x_i}{2} \left(1+s^{\delta}_1(x,y_j)\right)\left(a_1\phi^{\delta}-b_1\xi^{\delta}-c_1X^{\delta}_{\beta}\right)(0,0).
\end{gather}
Taking the norm of both sides of the equation, we obtain the equation for $s^{\delta}_1(x,y_j)$.
Then, the equation for $t^{\delta}_1(x,y_j)$ follows directly from the definition.

In consequence, we have completed the proof.
\end{proof}

The following lemma follows from \eqref{eq:Xalphadelta-x} and the definition of $t^{\delta}_1(x,y_j)$.
\begin{lemma}\label{lemma:Xalphadelta}
Suppose that the frame field $F^{\delta}(x,y_j)$ in \eqref{def:phidelta-x}--\eqref{def:Xbetadelta-x} is orthonormal at each point $(x,y_j)\in[x_i,x_{i+1}]\times \{y_j\}$.
Then, for $x\in[x_i,x_{i+1}]$, we have
\begin{gather*}
X^{\delta}_{\alpha}(0,0)+X^{\delta}_{\alpha}(x,y_j)
=t^{\delta}_1(x,y_j)\left(2X^{\delta}_{\alpha}(0,0)+(x-x_i)(a_1\phi^{\delta}-b_1\xi^{\delta}-c_1X^{\delta}_{\beta})(0,0)\right).
\end{gather*}
\end{lemma}

By Lemma \ref{lemma:Xalphadelta}, if the frame field $F^{\delta}(x,y_j)$ in \eqref{def:phidelta-x}--\eqref{def:Xbetadelta-x} is orthonormal at each point $(x,y_j)\in[x_i,x_{i+1}]\times \{y_j\}$, then we have
\begin{gather}\label{formula:Xalphadelta-x}
X^{\delta}_{\alpha}(x,y_j)-X^{\delta}_{\alpha}(0,0)
=(x-x_i)t^{\delta}_1(x,y_j)\left(
Y_1^{\delta}(0,0)
-\textstyle\frac{x-x_i}{2}\left((a_1^2+b_1^2+c_1^2)X^{\delta}_{\alpha}\right)(0,0)\right),
\end{gather}
where $Y_1^{\delta}(0,0):=(a_1\phi^{\delta}-b_1\xi^{\delta}-c_1X_{\beta}^{\delta})(0,0)$,
and
\begin{align}
\phi^{\delta}(x,y_j)-\phi^{\delta}(0,0)
&=\label{formula:phidelta-x}
-(x-x_i)t^{\delta}_1(x,y_j)a_1(0,0)\left(X^{\delta}_{\alpha}(0,0)+\textstyle\frac{x-x_i}{2}
Y_1^{\delta}(0,0)\right),\\
\xi^{\delta}(x,y_j)-\xi^{\delta}(0,0)
&=\label{formula:xidelta-x}
(x-x_i)t^{\delta}_1(x,y_j)b_1(0,0)\left(X^{\delta}_{\alpha}(0,0)+\textstyle\frac{x-x_i}{2}
Y_1^{\delta}(0,0)\right),\\
X^{\delta}_{\beta}(x,y_j)-X^{\delta}_{\beta}(0,0)
&=\label{formula:Xbetadelta-x}
(x-x_i)t^{\delta}_1(x,y_j)c_1(0,0)\left(X^{\delta}_{\alpha}(0,0)+\textstyle\frac{x-x_i}{2}
Y_1^{\delta}(0,0)\right).
\end{align}
Namely, with $\Omega_1(x,y)$ in \eqref{def:Omega12},
\begin{gather*}
F^{\delta}(x,y_j)-F^{\delta}(0)=(x-x_i)t^{\delta}_1(x,y_j)F^{\delta}(0)\left(\Omega_1(0)+\textstyle\frac{x-x_i}{2}(\Omega_1)^2(0)\right)
\end{gather*}
holds.
Conversely, we have the following theorem.
\begin{thm}\label{thm:Fdelta-x}
Let the functions $s^{\delta}_1(x,y_j)$ and $t^{\delta}_1(x,y_j)$ in \eqref{def:sdelta1} on $[x_i,x_{i+1}]\times \{y_j\}$ be given by Lemma \ref{lemma:s1t1}.
Let $F^{\delta}(x_i,y_j)(=F^{\delta}(0,0))$ be an orthonormal frame.
Then, the frame field $F^{\delta}(x,y_j)$ defined by \eqref{formula:Xalphadelta-x}--\eqref{formula:Xbetadelta-x} is orthonormal at each point $(x,y_j)\in [x_i,x_{i+1}]\times\{y_j\}$.
Furthermore, for a transformation $QF^{\delta}(x_i,y_j)$ of $F^{\delta}(x_i,y_j)$ by an orthogonal matrix $Q$, $F^{\delta}(x,y_j)$ in \eqref{formula:Xalphadelta-x}--\eqref{formula:Xbetadelta-x} changes into $QF^{\delta}(x,y_j)$.
In particular, for the unit matrix $\mathrm{Id}$, we have
\begin{gather*}
\|(Q-\mathrm{Id})F^{\delta}(x_i,y_j)\|
=\|(Q-\mathrm{Id})F^{\delta}(x_{i+1},y_j)\|
=\| Q-\mathrm{Id}\|.
\end{gather*}
\end{thm}
\begin{proof}
Note that the equations in \eqref{formula:Xalphadelta-x}--\eqref{formula:Xbetadelta-x} are induced from \eqref{def:phidelta-x}--\eqref{def:Xbetadelta-x} and \eqref{eq:Xalphadelta-x} (or Lemma \ref{lemma:Xalphadelta}).
Now, let the frame $F^{\delta}(0,0)(=F^{\delta}(x_i,y_j))$ be orthonormal.
The theorem follows from \eqref{def:phidelta-x}--\eqref{def:Xbetadelta-x} and \eqref{eq:Xalphadelta-x} by direct calculation as follows.
We firstly take the norm of the both sides in the equation of Lemma \ref{lemma:Xalphadelta}, and then we have
\begin{gather*}
\| X^{\delta}_{\alpha}(0,0)+X^{\delta}_{\alpha}(x,y_j)\|^2
=4t^{\delta}_1(x,y_j).
\end{gather*}
By the equation, we can show that all vector fields of $F^{\delta}(x,y_j)$ have unit norm: for example, by \eqref{eq:Xalphadelta-x} we have
\begin{align*}
\| X^{\delta}_{\alpha}(x,y_j)\|^2
&=(s^{\delta}_1(x,y_j))^2+(x-x_i)^2(t^{\delta}_1(x,y_j))^2(a_1^2+b_1^2+c_1^2)(0,0)\\
&=\frac{\left(1-((x-x_i)/2)^2(a_1^2+b_1^2+c_1^2)(0,0)\right)^2+(x-x_i)^2(a_1^2+b_1^2+c_1^2)(0,0)}{\left(1+((x-x_i)/2)^2(a_1^2+b_1^2+c_1^2)(0,0)\right)^2}=1.
\end{align*}
In the same way, we can verify that $\phi^{\delta}(x,y_j)$, $\xi^{\delta}(x,y_j)$ and $X^{\delta}_{\beta}(x,y_j)$ are also unit vectors.

Next, we have
\begin{gather*}
\langle X^{\delta}_{\alpha}(0,0)+X^{\delta}_{\alpha}(x,y_j), (a_1\phi^{\delta}-b_1\xi^{\delta}-c_1X^{\delta}_{\beta})(0,0)\rangle=(x-x_i)t^{\delta}_1(x,y_j)(a_1^2+b_1^2+c_1^2)(0,0)
\end{gather*}
by Lemma \ref{lemma:Xalphadelta}.
By the equation, we can show that all vector fields of $F^{\delta}(x,y_j)$ are orthogonal to each other: for example, by \eqref{def:phidelta-x}--\eqref{def:Xbetadelta-x} and \eqref{eq:Xalphadelta-x} we have
\begin{align*}
\langle X^{\delta}_{\alpha},\phi^{\delta}\rangle(x,y_j)
=\;&-((x-x_i)/2)s^{\delta}_1(x,y_j)(1+s^{\delta}_1(x,y_j))a_1(0,0)\\
&+(x-x_i)t^{\delta}_1(x,y_j)a_1(0,0)-((x-x_i)^3/2)(t^{\delta}_1(x,y_j))^2\left((a_1^2+b_1^2+c_1^2)a_1\right)(0,0)\\
=\;&-(x-x_i)(t^{\delta}_1(x,y_j))^2\left(1-((x-x_i)/2)^2(a_1^2+b_1^2+c_1^2)(0,0)\right)a_1(0,0)\\
&+(x-x_i)(t^{\delta}_1(x,y_j))^2\left(1+((x-x_i)/2)^2(a_1^2+b_1^2+c_1^2)(0,0)\right)a_1(0,0)\\
&-((x-x_i)^3/2)(t^{\delta}_1(x,y_j))^2\left((a_1^2+b_1^2+c_1^2)a_1\right)(0,0)=0.
\end{align*}
The other orthogonality for these vectors is also obtained in the same way.

The assertion for transformation $QF^{\delta}(x_i,y_j)$ of the initial condition follows directly from \eqref{formula:Xalphadelta-x}--\eqref{formula:Xbetadelta-x}. In consequence, we have verified the theorem.
\end{proof}

Step A-(2): For a given $F^{\delta}(0,0)$, we express the frame field $F^{\delta}(x_i,y)$ on $\{x_i\}\times[y_j,y_{j+1}]$ as a similar form to \eqref{def:phidelta-x}--\eqref{def:Xbetadelta-x}:
\begin{align}
\phi^{\delta}(x_i,y)-\phi^{\delta}(0,0)
&:=\label{def:phidelta-y}
-\textstyle\frac{y-y_j}{2}a_2(0,0)\left(X^{\delta}_{\beta}(0,0)+X^{\delta}_{\beta}(x_i,y)\right),\\
\xi^{\delta}(x_i,y)-\xi^{\delta}(0,0)
&:=\label{def:xidelta-y}
\textstyle\frac{y-y_j}{2}b_2(0,0)\left(X^{\delta}_{\beta}(0,0)+X^{\delta}_{\beta}(x_i,y)\right),\\
X^{\delta}_{\alpha}(x_i,y)-X^{\delta}_{\alpha}(0,0)
&:=\label{def:Xalphadelta-y}
\textstyle\frac{y-y_j}{2}c_2(0,0)\left(X^{\delta}_{\beta}(0,0)+X^{\delta}_{\alpha}(x_i,y)\right).
\end{align}
Let $s^{\delta}_2(x_i,y)$ and $t^{\delta}_2(x_i,y)$ be the functions on $\{x_i\}\times[y_j,y_{j+1}]$ defined by
\begin{gather}\label{def:sdelta2}
s^{\delta}_2(x_i,y):=\langle X^{\delta}_{\beta}(x_i,y),X^{\delta}_{\beta}(0,0)\rangle,\quad
t^{\delta}_2(x_i,y):=\left(1+s^{\delta}_2(x_i,y)\right)/2.
\end{gather}
Then, we have the following lemma and theorem in same way as in Step A-(1).
\begin{lemma}\label{lemma:s2t2}
Suppose that the frame field $F^{\delta}(x_i,y)$ in \eqref{def:phidelta-y}--\eqref{def:Xalphadelta-y} is orthonormal on $\{x_i\}\times[y_j,y_{j+1}]$. Then, we have
\begin{gather*}
s^{\delta}_2(x_i,y)=\frac{4-(y-y_j)^2(a_2^2+b_2^2+c_2^2)(0,0)}{4+(y-y_j)^2(a_2^2+b_2^2+c_2^2)(0,0)},\quad
t^{\delta}_2(x_i,y)=\frac{4}{4+(y-y_j)^2(a_2^2+b_2^2+c_2^2)(0,0)},\\
X^{\delta}_{\beta}(0,0)+X^{\delta}_{\beta}(x_i,y)=t^{\delta}_2(x_i,y)\left(2X^{\delta}_{\beta}(0,0)+(y-y_j)(a_2\phi^{\delta}-b_2\xi^{\delta}-c_2X^{\delta}_{\alpha})(0,0)\right).
\end{gather*}
\end{lemma}

By Lemma \ref{lemma:s2t2}, if the frame field $F^{\delta}(x_i,y)$ in \eqref{def:phidelta-y}--\eqref{def:Xalphadelta-y} is orthonormal at each point $(x_i,y)\in\{x_i\}\times [y_j,y_{j+1}]$, then we have
\begin{gather}\label{formula:Xbetadelta-y}
X^{\delta}_{\beta}(x_i,y)-X^{\delta}_{\beta}(0,0)
=(y-y_j)t^{\delta}_2(x_i,y)\left(
Y_2^{\delta}(0,0)-\textstyle\frac{y-y_j}{2}\left((a_2^2+b_2^2+c_2^2)X^{\delta}_{\beta}\right)(0,0)\right),
\end{gather}
where  $Y_2^{\delta}(0,0):=(a_2\phi^{\delta}-b_2\xi^{\delta}-c_2X_{\alpha}^{\delta})(0,0)$, and
\begin{align}
\phi^{\delta}(x_i,y)-\phi^{\delta}(0,0)
&=\label{formula:phidelta-y}
-(y-y_j)t^{\delta}_2(x_i,y)a_2(0,0)\left(X^{\delta}_{\beta}(0,0)+\textstyle\frac{y-y_j}{2}
Y_2^{\delta}(0,0)\right),\\
\xi^{\delta}(x_i,y)-\xi^{\delta}(0,0)
&=\label{formula:xidelta-y}
(y-y_j)t^{\delta}_2(x_i,y)b_2(0,0)\left(X^{\delta}_{\beta}(0,0)+\textstyle\frac{y-y_j}{2}
Y_2^{\delta}(0,0)\right),\\
X^{\delta}_{\alpha}(x_i,y)-X^{\delta}_{\alpha}(0,0)
&=\label{formula:Xalphadelta-y}
(y-y_j)t^{\delta}_2(x_i,y)c_2(0,0)\left(X^{\delta}_{\beta}(0,0)+\textstyle\frac{y-y_j}{2}Y_2^{\delta}(0,0)\right).
\end{align}
Namely, with $\Omega_2(x,y)$ in \eqref{def:Omega12},
\begin{gather*}
F^{\delta}(x_i,y)-F^{\delta}(0)=(y-y_j)t^{\delta}_2(x_i,y)F^{\delta}(0)\left(\Omega_2(0)+\textstyle\frac{y-y_j}{2}(\Omega_2)^2(0)\right)
\end{gather*}
holds.
Conversely, we have the following theorem.
\begin{thm}\label{thm:Fdelta-y}
Let the functions $s^{\delta}_2(x_i,y)$ and $t^{\delta}_2(x_i,y)$ in \eqref{def:sdelta2} on $\{x_i\}\times[y_j,y_{j+1}]$ be given in Lemma \ref{lemma:s2t2}.
Let $F^{\delta}(x_i,y_j)(=F^{\delta}(0,0))$ be an orthonormal frame.
Then, the frame field $F^{\delta}(x_i,y)$ defined by \eqref{formula:Xbetadelta-y}--\eqref{formula:Xalphadelta-y} is orthonormal at each point $(x_i,y)\in\{x_i\}\times [y_j,y_{j+1}]$.
Furthermore, for a transformation $QF^{\delta}(x_i,y_j)$ of $F^{\delta}(x_i,y_j)$ by an orthogonal matrix $Q$, $F^{\delta}(x_i,y)$ in \eqref{formula:Xbetadelta-y}--\eqref{formula:Xalphadelta-y} changes into $QF^{\delta}(x_i,y)$.
In particular, we have
\begin{gather*}
\|(Q-\mathrm{Id})F^{\delta}(x_i,y_j)\|
=\|(Q-\mathrm{Id})F^{\delta}(x_i,y_{j+1})\|
=\| Q-\mathrm{Id}\|.
\end{gather*}
\end{thm}

Step A-(3): For a sub-domain $\Delta_{i,j}$, the orthonormal frames $F^{\delta}(x_{i+1},y_j)$ and $F^{\delta}(x_i,y_{j+1})$ have been determined from $F^{\delta}(x_i,y_j)$.
Hence, we can construct the frames $F^{\delta}(x_{i+1},y)$ on $\{x_{i+1}\}\times[y_j,y_{j+1}]$ and $F^{\delta}(x,y_{j+1})$ on $[x_i,x_{i+1}]\times\{y_{j+1}\}$ by (2) and (1) of Step A, respectively.
Thus, we have two frames $F^{\delta}(x_{i+1},y_{j+1})$ at the point $P_{i+1,j+1}$, since $F^{\delta}(x_{i+1},y_{j+1})$ is defined for each path from $P_{i,j}$ to $P_{i+1,j+1}$.
We study the difference of these two frames and verify \eqref{ineq:norm-Fdeltas}. 

In this step, we simply write
\begin{gather*}
0:=(x_i,y_j),\quad
1:=(x_{i+1},y_j),\quad
2:=(x_i,y_{j+1}),\quad
3:=(x_{i+1},y_{j+1})
\end{gather*}
and denote by $F^{\underline\delta}(3)$ and $F^{\overline\delta}(3)$, respectively, the frames  determined by the paths $\underline{m}:0\rightarrow 1\rightarrow 3$ and $\overline{m}:0\rightarrow 2\rightarrow 3$.
Note that $t^{\delta}_1$ (resp.\ $t^{\delta}_2$) is defined on $[x_i,x_{i+1}]\times\{y_j\}$ and $[x_i,x_{i+1}]\times\{y_{j+1}\}$ (resp.\ on $\{x_i\}\times[y_j,y_{j+1}]$ and $\{x_{i+1}\}\times[y_j,y_{j+1}]$).
Then, we have
\begin{gather*}
t^{\delta}_1(1)=1/\left(1+(\delta/2)^2(a_1^2+b_1^2+c_1^2)(0)\right),\quad
t^{\delta}_2(3)=1/\left(1+(\delta/2)^2(a_2^2+b_2^2+c_2^2)(1)\right),\\
t^{\delta}_2(2)=1/\left(1+(\delta/2)^2(a_2^2+b_2^2+c_2^2)(0)\right),\quad
t^{\delta}_1(3)=1/\left(1+(\delta/2)^2(a_1^2+b_1^2+c_1^2)(2)\right),
\end{gather*}
and hence $t^{\delta}_i(k)=1+O(\delta^2)$ hold for the integers $i$ and $k$ above.
Furthermore, since $a_i$, $b_i$ and $c_i$ are analytic functions on $E$, we have $ t^{\delta}_1(3)-t^{\delta}_1(1)=O(\delta^3)$ and $t^{\delta}_2(3)-t^{\delta}_2(2)=O(\delta^3)$.

Now, we have
\begin{align*}
F^{\underline\delta}(3)-F^{\delta}(0)&=(F^{\underline\delta}(3)-F^{\delta}(1))+(F^{\delta}(1)-F^{\delta}(0)),\\
F^{\overline\delta}(3)-F^{\delta}(0)&=(F^{\overline\delta}(3)-F^{\delta}(2))+(F^{\delta}(2)-F^{\delta}(0)).
\end{align*}
We define $G_x$ and $G_y$ by
\begin{align*}
G_x:=(F^{\overline\delta}(3)-F^{\delta}(2))-(F^{\delta}(1)-F^{\delta}(0)),\quad
G_y:=(F^{\underline\delta}(3)-F^{\delta}(1))-(F^{\delta}(2)-F^{\delta}(0)).
\end{align*}
Then, we have only to verify that $G_x=G_y+O(\delta^3)$ holds.
Hence, we study the third degree for $\delta$ of $G_x-G_y$: in the asymptotic expansion $G_x-G_y\approx\sum_{k=0}^3p_k\delta^k$ for $\delta$, we show $p_k=0 \ (k=0,1,2)$.
In the estimate of each $G_x$ and $G_y$, the frames $F^{\overline\delta}(3)$ and $F^{\underline\delta}(3)$ are naturally distinguished by \eqref{formula:Xalphadelta-x}--\eqref{formula:Xbetadelta-x} and \eqref{formula:Xbetadelta-y}--\eqref{formula:Xalphadelta-y}, and hence we can write them as $F^{\delta}(3)$.

Now, for a function or a vector field $k(x,y)$, we denote
$[k]^{1}_{0}:=k(1)-k(0)$ and so on.
Let $Y_1^{\delta}:=a_1\phi^{\delta}-b_1\xi^{\delta}-c_1X_{\beta}^{\delta}$ and $Y_2^{\delta}:=a_2\phi^{\delta}-b_2\xi^{\delta}-c_2X_{\alpha}^{\delta}$.
\begin{lemma}
With the second degree for $\delta$ of $[Y_1^{\delta}]^2_0$ and $[Y_2^{\delta}]^1_0$, we have
\begin{align*}
[Y_1^{\delta}]^2_0
\approx\;& \delta\left(c_1c_2X_{\alpha}^{\delta}+(c_2)_xX_{\beta}^{\delta}\right)(0)
+(\delta^2/2)\left(((c_2)_x-(c_1)_y)Y_2^{\delta}-c_1(a_2^2+b_2^2-c_2^2)X_{\beta}^{\delta}\right)(0)\\
&+(\delta^2/2)\left((a_1)_{yy}\phi^{\delta}-(b_1)_{yy}\xi^{\delta}-(c_1)_{yy}X_{\beta}^{\delta}\right)(0),\\
[Y_2^{\delta}]^1_0
\approx\;& \delta\left((c_1)_yX_{\alpha}^{\delta}+c_1c_2X_{\beta}^{\delta}\right)(0)
+(\delta^2/2)\left(((c_1)_y-(c_2)_x)Y_1^{\delta}-c_2(a_1^2+b_1^2-c_1^2)X_{\alpha}^{\delta}\right)(0)\\
&+(\delta^2/2)\left((a_2)_{xx}\phi^{\delta}-(b_2)_{xx}\xi^{\delta}-(c_2)_{xx}X_{\alpha}^{\delta}\right)(0).
\end{align*}
\end{lemma}
\begin{proof}
We only prove the equation for $Y_1^{\delta}$, since the equation for $Y_2^{\delta}$ is obtained in the same way.
Now, for the element $[a_1\phi^{\delta}]^2_0$ of $[Y_1^{\delta}]^2_0$, we have
\begin{align*}
[a_1\phi^{\delta}]^2_0/\delta
&=\left(a_1(2)[\phi^{\delta}]^2_0+(a_1(2)-a_1(0))\phi^{\delta}(0)\right)/\delta\\
&\approx -(a_2(a_1+\delta(a_1)_y))(0)\left(X_{\beta}^{\delta}+(\delta/2)Y_2^{\delta}\right)(0)
+((a_1)_y+(\delta/2)(a_1)_{yy})(0)\phi^{\delta}(0)\\
&\approx \left(-a_1a_2X_{\beta}^{\delta}+a_2c_1\phi^{\delta}\right)(0)+(\delta/2)\left(-a_1a_2Y_2^{\delta}-2a_2^2c_1X_{\beta}^{\delta}+(a_1)_{yy}\phi^{\delta}\right)(0),
\end{align*}
by $(a_1)_y=a_2c_1$ in Lemma \ref{lemma:pde-abc}. In the same way, we have
\begin{gather*}
[b_1\xi^{\delta}]^2_0/\delta\approx\left(b_1b_2X_{\beta}^{\delta}+b_2c_1\xi^{\delta}\right)(0)+(\delta/2)\left(b_1b_2Y_2^{\delta}+2b_2^2c_1X_{\beta}^{\delta}+(b_1)_{yy}\xi^{\delta}\right)(0),\\
[c_1X_{\beta}^{\delta}]^2_0/\delta\approx\left(c_1Y_2^{\delta}+(c_1)_yX_{\beta}^{\delta}\right)(0)+(\delta/2)\left(2(c_1)_yY_2^{\delta}-c_1(a_2^2+b_2^2+c_2^2)X_{\beta}^{\delta}+(c_1)_{yy}X_{\beta}^{\delta}\right)(0).
\end{gather*}
Then, by $(c_2)_x+(c_1)_y+a_1a_2+b_1b_2=0$ in Lemma \ref{lemma:pde-abc}, we obtain the equation for $[Y_1^{\delta}]^2_0$.
\end{proof}
\begin{lemma}\label{lemma:delta3}
With the third degree for $\delta$ of $G_x-G_y$, we have
\begin{align*}
([\phi^{\delta}]^3_2-[\phi^{\delta}]^1_0)
&-([\phi^{\delta}]^3_1-[\phi^{\delta}]^2_0)\\
\approx\;& - (\delta^3/2)\left(c_1((a_2)_y-a_1c_2)X_{\alpha}^{\delta}-c_2((a_1)_x-a_2c_1)X_{\beta}^{\delta}\right)(0),\\
([\xi^{\delta}]^3_2-[\xi^{\delta}]^1_0)
&-([\xi^{\delta}]^3_1-[\xi^{\delta}]^2_0)\\
\approx\;& (\delta^3/2)\left(c_1((b_2)_y-b_1c_2)X_{\alpha}^{\delta}
-c_2((b_1)_x-b_2c_1)X_{\beta}^{\delta}\right)(0),\\
([X_{\alpha}^{\delta}]^3_2-[X_{\alpha}^{\delta}]^1_0)
&-([X_{\alpha}^{\delta}]^3_1-[X_{\alpha}^{\delta}]^2_0)\\
\approx\;& -(\delta^3/2)\left(-c_1((a_2)_y-a_1c_2)\phi^{\delta}+c_1((b_2)_y-b_1c_2)\xi^{\delta}\right)(0)\\
&-(\delta^3/2)\left((c_2)_{xx}+(c_1)_{yy}+c_1(a_2^2+b_2^2)+c_2(a_1^2+b_1^2)\right)(0)X_{\beta}^{\delta}(0),\\
([X_{\beta}^{\delta}]^3_2-[X_{\beta}^{\delta}]^1_0)
&-([X_{\beta}^{\delta}]^3_1-[X_{\beta}^{\delta}]^2_0)\\
\approx\;& (\delta^3/2)\left(-c_2((a_1)_x-a_2c_1)\phi^{\delta}+c_2((b_1)_x-b_2c_1)\xi^{\delta}\right)(0)\\
&+(\delta^3/2)\left((c_2)_{xx}+(c_1)_{yy}+c_1(a_2^2+b_2^2)+c_2(a_1^2+b_1^2)\right)(0)X_{\alpha}^{\delta}(0).
\end{align*}
\end{lemma} 
\begin{proof}
For the first equation, we have
\begin{align*}
([\phi^{\delta}]^3_2-[\phi^{\delta}]^1_0)/\delta
=\;&-[t_1^{\delta}]^3_1 \left(a_1(X_{\alpha}^{\delta}+(\delta/2)Y_1^{\delta})\right)(2)-t_1^{\delta}(1)\left[a_1(X_{\alpha}^{\delta}+(\delta/2)Y_1^{\delta})\right]^2_0\\
\approx\;& -\left[a_1(X_{\alpha}^{\delta}+(\delta/2)Y_1^{\delta})\right]^2_0\\
=\;& -a_1(2)\left[X_{\alpha}^{\delta}+(\delta/2)Y_1^{\delta}\right]^2_0-[a_1]^2_0 \left(X_{\alpha}^{\delta}+(\delta/2)Y_1^{\delta}\right)(0)\\
\approx\;& -\delta(a_1+\delta(a_1)_y)(0)\left(c_2X_{\beta}^{\delta}+(\delta/2)(c_2Y_2^{\delta}+c_1c_2X_{\alpha}^{\delta}+(c_2)_xX_{\beta}^{\delta})\right)(0)\\
&-\delta\left(((a_1)_y+(\delta/2)(a_1)_{yy})\,(X_{\alpha}^{\delta}+(\delta/2)Y_1^{\delta})\right)(0).
\end{align*}
Hence, we have
\begin{gather*}
\begin{split}
&([\phi^{\delta}]^3_2-[\phi^{\delta}]^1_0)/\delta^2\\
\approx
&-\left(a_2c_1X_{\alpha}^{\delta}+a_1c_2X_{\beta}^{\delta}\right)(0)
-(\delta/2)\left(a_1a_2(c_1+c_2)\phi^{\delta}-(a_1b_2c_2+a_2b_1c_1)\xi^{\delta}\right)(0)\\
&-(\delta/2)\left(\left(a_1c_2(c_1-c_2)+(a_1)_{yy}\right)X_{\alpha}^{\delta}
+\left(a_1(c_2)_x+2a_2c_1c_2-a_2c_1^2\right)X_{\beta}^{\delta}\right)(0).
\end{split}
\end{gather*}
In the same way, we have
\begin{gather*}
\begin{split}
&([\phi^{\delta}]^3_1-[\phi^{\delta}]^2_0)/\delta^2\\
\approx
&-\left(a_2c_1X_{\alpha}^{\delta}+a_1c_2X_{\beta}^{\delta}\right)(0)
-(\delta/2)\left(a_1a_2(c_1+c_2)\phi^{\delta}-(a_1b_2c_2+a_2b_1c_1)\xi^{\delta}\right)(0)\\
&-(\delta/2)\left(\left(-a_1c_2^2+2a_1c_1c_2+a_2(c_1)_y\right)X_{\alpha}^{\delta}
+\left(-a_2c_1^2+a_2c_1c_2+(a_2)_{xx}\right)X_{\beta}^{\delta}\right)(0).
\end{split}
\end{gather*}
From these equations, we have
\begin{align*}
&([\phi^{\delta}]^3_2-[\phi^{\delta}]^1_0)
-([\phi^{\delta}]^3_1-[\phi^{\delta}]^2_0)\\
&\approx -(\delta^3/2)\left(\left((a_1)_{yy}-a_2(c_1)_y-a_1c_1c_2\right)X_{\alpha}^{\delta}-\left((a_2)_{xx}-a_1(c_2)_x-a_2c_1c_2\right)X_{\beta}^{\delta}\right)(0).
\end{align*}
Then, by Lemma \ref{lemma:pde-abc}, we have
\begin{align*}
(a_1)_{yy}-a_2(c_1)_y-a_1c_1c_2=c_1\left((a_2)_y-a_1c_2\right),\quad
(a_2)_{xx}-a_1(c_2)_x-a_2c_1c_2=c_2\left((a_1)_x-a_2c_1\right).
\end{align*}
Hence, we obtain the first equation.
The second equation is also obtained in the same way.

Next, we prove the third equation. We have
\begin{align*}
& ([X^{\delta}_{\alpha}]^3_2-[X^{\delta}_{\alpha}]^1_0)/\delta
\approx \left[Y_1^{\delta}-(\delta/2)(a_1^2+b_1^2+c_1^2)X_{\alpha}^{\delta}\right]^2_0\\
&\approx [Y_1^{\delta}]^2_0-(\delta/2)(a_1^2+b_1^2+c_1^2)(2)[X_{\alpha}^{\delta}]^2_0-(\delta^2/2)(a_1^2+b_1^2+c_1^2)_y(0)X_{\alpha}^{\delta}(0)\\
&\approx [Y_1^{\delta}]^2_0-(\delta/2)(a_1^2+b_1^2+c_1^2)(0)[X_{\alpha}^{\delta}]^2_0-(\delta^2/2)(a_1^2+b_1^2+c_1^2)_y(0)X_{\alpha}^{\delta}(0)\\
&\approx \delta\left(c_1c_2X_{\alpha}^{\delta}+(c_2)_xX_{\beta}^{\delta}\right)(0)+(\delta^2/2)\left(((c_2)_x-(c_1)_y)Y_2^{\delta}-(a_1^2+b_1^2+c_1^2)_yX_{\alpha}^{\delta}\right)(0)\\
&+(\delta^2/2)\left(-\left(c_1(a_2^2+b_2^2-c_2^2)+c_2(a_1^2+b_1^2+c_1^2)\right)X_{\beta}^{\delta}
+(a_1)_{yy}\phi^{\delta}-(b_1)_{yy}\xi^{\delta}-(c_1)_{yy}X_{\beta}^{\delta}\right)(0).
\end{align*}
In the same way, we have
\begin{align*}
& ([X^{\delta}_{\alpha}]^3_1-[X^{\delta}_{\alpha}]^2_0)/\delta
\approx
\delta\left(c_1c_2X_{\alpha}^{\delta}+(c_2)_xX_{\beta}^{\delta}\right)(0)\\
&+(\delta^2/2)\left(c_1c_2Y_1^{\delta}+(c_2)_xY_2^{\delta}+(c_2(c_1)_y+2c_1(c_2)_x)X_{\alpha}^{\delta}
+(c_1c_2^2+(c_2)_{xx})X_{\beta}^{\delta}\right)(0).
\end{align*}
From these equations, we have
\begin{align*}
& ([X_{\alpha}^{\delta}]^3_2-[X_{\alpha}^{\delta}]^1_0)-([X_{\alpha}^{\delta}]^3_1-[X_{\alpha}^{\delta}]^2_0)\\
& \approx -(\delta^3/2)\left(c_1c_2(a_1\phi^{\delta}-b_1\xi^{\delta})+(c_1)_y(a_2\phi^{\delta}-b_2\xi^{\delta})-(a_1)_{yy}\phi^{\delta}+(b_1)_{yy}\xi^{\delta}+(c_1)_{yy}X_{\beta}^{\delta}\right)(0)\\
&-(\delta^3/2)\left(\left(2c_1(c_2)_x+(a_1^2+b_1^2+c_1^2)_y\right)X_{\alpha}^{\delta}+\left(c_1(a_2^2+b_2^2)+c_2(a_1^2+b_1^2)+(c_2)_{xx}\right)X_{\beta}^{\delta}\right)(0).
\end{align*}
Then, by Lemma \ref{lemma:pde-abc}, we have
\begin{gather*}
(b_1)_{yy}-b_2(c_1)_y-b_1c_1c_2=c_1\left((b_2)_y-b_1c_2\right),\\
2c_1(c_2)_x+2\left(a_1(a_1)_y+b_1(b_1)_y+c_1(c_1)_y\right)
=2c_1\left((c_2)_x+(c_1)_y+a_1a_2+b_1b_2\right)=0.
\end{gather*}
Hence, we obtain the third equation.
The last equation is also obtained in the same way.
\end{proof}
\begin{defn}[Choice of the constant $K$ in \eqref{ineq:norm-Fdeltas}]\label{def:K}
(1)  For the four equations of Lemma \ref{lemma:delta3}, we take norm of the coefficient vectors for $\delta^3$.
Then, let $K_1$ be the maximum of these four norms on the domain $E$.

(2) For the matrices $\Omega_1$ and $\Omega_2$ in \eqref{def:Omega12}, let
\begin{gather*}
K_2:=\Max_{p\in E}\left\{\|\Omega_i(p)\|,\ \|(\Omega_1)_x(p)\|/2,\ \|(\Omega_2)_y(p)\|/2\right\}.
\end{gather*}
Then, we define $K$ by $K:=\Max\{K_1,K_2\}+1$.
Here, we add $1$ to $\Max\{K_1,K_2\}$, since $(F^{\underline\delta_n}-F^{\overline\delta_n})(x_{i+1},y_{j+1})$ also includes the terms of degree $\delta_n^i \ (i>3)$, in (1).
\end{defn}

By Lemma \ref{lemma:delta3} and the definition of $K$, we have \eqref{ineq:norm-Fdeltas}:
\begin{thm}\label{thm:norm-Fdeltas}
For an integer $n$ and $(x_p,y_p)\in E$, we take a division of $E_p$.
With a sub-domain $\Delta^n_{i,j}$, suppose that an orthonormal frame $F^{\delta_n}(x_i,y_j)$ is determined.
Then, we have
\begin{gather*}
\| (F^{\underline \delta_n}-F^{\overline\delta_n})(x_{i+1,j+1})\|\leq K\delta_n^3,
\end{gather*}
if $n$ is sufficiently large.
\end{thm}

By Theorems \ref{thm:Fdelta-x}, \ref{thm:Fdelta-y} and \ref{thm:norm-Fdeltas}, the proof of Proposition \ref{prop:norm-Fdeltas} also has been completed.

For Step C: As in the proofs of Step A, we fix an integer $n$ and study only the case of $E=[x_0,x_e]\times [y_0,y_e]$, where $x_e=x_0+a$ and $y_e=y_0+a$.
Hence, we have $s=t=n$ and denote $\delta:=\delta_n=a/n$.
Let $F^0(x,y)$ be the solution to \eqref{eq:dF0} on $D$ under a given initial condition $F^0(x_0,y_0)$.
Let $\underline{m}$ be the path from $P_{0,0}$ to $P_{n,n}$ such that  
$\underline{m}: P_{0,0}\rightarrow \dots\rightarrow P_{n,0}\rightarrow \dots\rightarrow P_{n,n}=(x_e,y_e),$ and $F^{\delta}(x,y):=F^{\underline\delta}(x,y)$ be the orthonormal frame field on $\underline{m}$ determined from $F^{\delta}(x_0,y_0)=F^0(x_0,y_0)$ by Step A.
Firstly, note that we have
\begin{gather*}
\int_{\underline{m}}F^0(x,y)\Omega=\int_{\underline{m}}dF^0=\left[F^0\right]^{(x_n,x_n)}_{(x_0,y_0)},
\end{gather*}
and
\begin{align*}
\left[F^{\delta}\right]_{(x_i,y_0)}^{(x_{i+1},y_0)}
=\int_{(x_i,y_0)}^{(x_{i+1},y_0)}\frac{dF^{\delta}}{dx}(x,y_0)dx,\quad
\left[F^{\delta}\right]_{(x_n,y_j)}^{(x_n,y_{j+1)}}
=\int_{(x_n,y_j)}^{(x_n,y_{j+1)}}\frac{dF^{\delta}}{dy}(x_n,y)dy
\end{align*}
on each sub-interval of $\underline m$.
  
Now, we study the norm $\| (F^0-F^{\delta_n})(x_n,y_n)\|$.
For the second degree Taylor polynomials of both frames $F^0(x,y)$ and $F^{\delta}(x,y)$ at each point $(x_i,y_0)$ and $(x_n,y_j)$, we have the following lemma.
\begin{lemma}\label{lemma:taylor-F0-Fdelta}
For $x\in[x_i,x_{i+1}]$, the first degree Taylor polynomial of $\frac{d(F^0-F^{\delta})}{dx}{(x,y_0)}$ at $x=x_i$ is given by
\begin{gather*}
\begin{split}
\frac{d}{dx}(F^0-F^{\delta})_{(x,y_0)}
\approx\;& (x-x_i)F^0_{(x_i,y_0)}(\Omega_1)_x(x_i,y_0)\\
&+(F^0-F^{\delta})_{(x_i,y_0)}\left(\Omega_1(x_i,y_0)+(x-x_i)\Omega_1^2(x_i,y_0)\right).
\end{split}
\end{gather*}
For $y\in[y_j,y_{j+1}]$, the first degree Taylor polynomial of $\frac{d(F^0-F^{\delta})}{dy}{(x_n,y)}$ at $y=y_j$ is given by
\begin{gather*}
\begin{split}
\frac{d}{dy}(F^0-F^{\delta})_{(x_n,y)}
\approx\;& (y-y_j)F^0_{(x_n,y_j)}(\Omega_2)_y(x_n,y_j)\\
&+(F^0-F^{\delta})_{(x_n,y_j)}\left(\Omega_2(x_n,y_j)+(y-y_j)\Omega_2^2(x_n,y_j)\right).
\end{split}
\end{gather*}
\end{lemma}
\begin{proof}
Since $dF^0=F^0\Omega$ and $\partial^2 F^0/\partial x^2=F^0[\Omega_1^2+(\Omega_1)_x]$, we have the polynomial at $x=x_i$ of $d (F^0(x,y_0))/dx$:
\begin{gather*}
d (F^0(x,y_0))/dx \approx F^0(x_i,y_0)\left((\Omega_1)_{(x_i,y_0)}+(x-x_i)\left(\Omega_1^2+(\Omega_1)_x\right)_{(x_i,y_0)}\right)
\end{gather*}
for $x\in [x_i,x_{i+1}]$. 
Next, with the function $t^{\delta}_1(x,y_0)$ on $x\in [x_i,x_{i+1}]$, we have $t^{\delta}_1(x_i,y_0)=1$ and $(dt^{\delta}_1/dx)(x_i,y_0)=0$.
From these equations, we obtain the polynomial at $x=x_i$ of $d (F^{\delta}(x,y_0))/dx$:
\begin{gather*}
d (F^{\delta}(x,y_0))/dx \approx F^{\delta}(x_i,y_0)\left((\Omega_1)_{(x_i,y_0)}+(x-x_i)\Omega_1^2|_{(x_i,y_0)}\right)
\end{gather*}
for $x\in [x_i,x_{i+1}]$.
In the same way, we obtain the polynomials at $y=y_j$ of $d (F^0(x_n,y))/dy$ and $d (F^{\delta}(x_n,y))/dy$ for $y\in [y_j,y_{j+1}]$.
In consequence, we have verified the lemma.
\end{proof}

Now, we put $T^{\delta}(x,y):=(F^0-F^{\delta})(x,y)$.
Under the condition $T^{\delta}(x_0,y_0)=(F^0-F^{\delta})(x_0,y_0)=0$, we integrate $(dT^{\delta}/dx){(x,y_0)}$ from $x_i$ to $x_{i+1}$ in order of $i=0,1,\dots,n-1$.
Then, by Lemma \ref{lemma:taylor-F0-Fdelta} and $\Max \left\{\|\Omega_i(p)\|, \|(\Omega_1)_x(p)\|/2, \|(\Omega_2)_y(p)\|/2 \;|\; p\in E\right\}+1\leq  K$, we have
\begin{gather*}
\| T^{\delta}(x_1,y_0)\|\leq K\delta^2,\quad
\| T^{\delta}(x_{i+1},y_0)-T^{\delta}(x_i,y_0)\|\leq K\delta^2+\| T^{\delta}(x_i,y_0)\| K\delta \ \ \text{for} \ i\geq 1,
\end{gather*}
if $n$ is sufficiently large.
Hence, we have
\begin{gather*}
\| T^{\delta}(x_{i+1},y_0)\|+\delta \leq (\| T^{\delta}(x_{i},y_0)\|+\delta)(1+K\delta) \ \ \ \text{for} \ i\geq 1
\end{gather*}
by $\| T^{\delta}(x_{i+1},y_0)\|\leq K\delta^2+\| T^{\delta}(x_{i},y_0)\|(1+K\delta)$.
Therefore, we have
\begin{gather}\label{ineqs:norm-Tdelta}
\| T^{\delta}(x_n,y_0)\|+\delta \leq (\| T^{\delta}(x_1,y_0)\|+\delta)(1+K\delta)^{n-1}\leq \delta(1+K\delta)^n.
\end{gather}
Here, we have $(1+K\delta)^n<\exp(Ka).$

In the same way, by the integral of $(dT^{\delta}/dy){(x_n,y)}$ from $y_j$ to $y_{j+1}$, we have
\begin{gather*}
\| T^{\delta}(x_n,y_{j+1})-T^{\delta}(x_n,y_j)\|\leq K\delta^2+\| T^{\delta}(x_n,y_j)\| K\delta \ \ \ \text{for} \ j\geq 0,
\end{gather*}
if $n$ is sufficiently large.  
In consequence, we have
\begin{gather}\label{ineq:normT-K}
\| T^{\delta}(x_n,y_n)\|
\leq (\| T^{\delta}(x_n,y_0)\| +\delta)(1+K\delta)^n-\delta
\leq \delta \left(\exp(2Ka)-1\right)
\end{gather}
by \eqref{ineqs:norm-Tdelta}.
The inequality \eqref{ineq:normT-K} implies that $F^{\delta_n}(x_e,y_e)(=F^{\delta_n}(x_n,y_n))$ converges to $F^0(x_e,y_e)$ as $n$ tends to $\infty$.

In the argument above, we can replace $(x_e,y_e)$ with an arbitrary point $(x_p,y_p)\in E$, by Definition \ref{def:path}.
Thus, Step C has been verified by the above argument and Proposition \ref{prop:norm-Fdeltas}:
\begin{thm}\label{thm:Fdelta-F0}
Let $F^0(x_0,y_0)$ be an orthonormal frame at $(x_0,y_0)$ given arbitrarily.
Let $F^0(x,y)$ be the orthonormal frame field satisfying \eqref{eq:dF0} under the initial condition $F^0(x_0,y_0)$.
For an integer $n$ and $(x_p,y_p)\in E$, we take a division of $E_p$.
Let $\underline{m}$ and $\overline{m}$ be the two paths from $(x_0,y_0)$ to $(x_p,y_p)$ given by \eqref{def:mlower}--\eqref{def:mupper}.
Let $F^{\underline\delta_n}(x_p,y_p)$ and $F^{\overline\delta_n}(x_p,y_p)$ be the orthonormal frame determined from $\underline{m}$ and $\overline{m}$ by Step A.
Then, $F^{\underline\delta_n}(x_p,y_p)$ and $F^{\overline\delta_n}(x_p,y_p)$ converge to $F^0(x_p,y_p)$ uniformly for all $(x_p,y_p)\in E$, as $n$ tends to $\infty$.
\end{thm}

By the argument above, all Steps A, B and C have been verified: for an integer $n$ and $(x,y)\in E$, we have obtained two approximations $F^{\underline\delta_n}(x,y)$ and $F^{\overline\delta_n}(x,y)$ of $F^0(x,y)$.
In Theorem \ref{thm:Fdelta-F0}, note that any frame $F^{\delta_n}(x_p,y_p)$ determined by a path from $(x_0,y_0)$ to $(x_p,y_p)$ is also an approximation of $F^0(x_p,y_p)$ if $n$ is sufficiently large, by Proposition \ref{prop:norm-Fdeltas}.

Next, we construct an approximation of several coordinate curves in the curvature surface $f^0(x,y)$ on $D$.
Before the construction, we state a remark.
For an $x$-curve $f^0(x,y_0)$ with fixed $y_0> 0$ on an interval $I=[x_0,x_e]:=[y_0-a,y_0+a]$, where $a>0$ and $ y_0-a>0$, we fix a division of $I$ with equal length $\delta_n=a/n$ and make the frame field $F^{\delta_n}(x,y_0)$ on $I$ under the condition $F^{\delta_n}(x_0,y_0)=\mathrm{Id}$, by Theorem \ref{thm:Fdelta-F0}.
Then, for a vector
\begin{gather*}
{\bs{f}}^{\delta_n}(x,y_0):=(1+y_0^2)^{-1}\left(X^{\delta_n}_{\beta}-2(1+y^2)\xi^{\delta_n}+y\phi^{\delta_n}\right)(x,y_0)
\end{gather*}
on each sub-interval $[x_i,x_{i+1}]$, one might guess from Theorem \ref{thm:xcurve} if the curve
\begin{gather}\label{def:fdelta-x}
f^{\delta_n}(x,y_0):=\frac{1}{2\sqrt{5}}\left(\left[{\bs{f}}^{\delta_n}\right]_{(x_i,y_0)}^{(x,y_0)}+\sum_{k=1}^i\left[{\bs{f}}^{\delta_n}\right]_{(x_{k-1},y_0)}^{(x_k,y_0)}\right)
\ \ \ \text{for} \ x_i< x\leq x_{i+1},
\end{gather}
is an approximation of the $x$-curve (where we adopt the non-unit vector ${\bs{f}}^{\delta_n}(x,y_0)$ in Theorem \ref{thm:xcurve}).
However, it is not an approximation of the $x$-curve.
In fact, for \eqref{def:fdelta-x} we have ${\bs{f}}^{\delta_n}(x,y_0)-{\bs{f}}^{\delta_n}(y_0,y_0)\equiv \bs{0}$ for $y_0=x_n\leq x\leq x_{n+1}$, since $(y_0,y_0)$ is in the singularity $D\cap S_1$ of the metric $g_0$ on $D$ (or by \eqref{formula:Xalphadelta-x}--\eqref{formula:Xbetadelta-x}).
In consequence, it is necessary for the interval $[y_0,x_e]$ that we make the frame $F^{\delta_n}(x,y_0)$ and the vector ${\bs{f}}^{\delta_n}(x,y_0)$ in the direction that $x$ decreases, because $f^{\delta_n}(y_0,y_0)$ is determined as the limit of $f^{\delta_n}(p)$ for $p\in D\setminus S_1$.
The approximation $F^{\delta_n}(x,y_0)$ on $[x_{i-1},x_i]$ is also defined from $F^{\delta_n}(x_i,y_0)$ by \eqref{formula:Xalphadelta-x}--\eqref{formula:Xbetadelta-x} as $(x-x_i)\leq 0$, and then ${\bs{f}}^{\delta_n}(x,y_0)$ on $[x_{i-1},x_i]$ is determined by the frame $F^{\delta_n}(x,y_0)$ in the inverse direction.

\subsection{Approximation of $x$-curve $f^0(x,y_0)$}\label{subsec:approx-x}

Let $f^0(x,y_0)$ be an $x$-curve with $y_0> 0$ on $[y_0-a,y_0+a]$.
We divide the interval $[y_0-a,y_0+a]$ as above.
On the interval $[y_0-a,y_0]$, we make $F^{\delta_n}_1(x,y_0):=F^{\delta_n}(x,y_0)$ from $F^{\delta_n}_1(y_0-a,y_0)=\mathrm{Id}$ by \eqref{formula:Xalphadelta-x}--\eqref{formula:Xbetadelta-x} and determine $f^{\delta_n}_1(x,y_0):=f^{\delta_n}(x,y_0)$ by \eqref{def:fdelta-x}, which satisfies $f^{\delta_n}_1(y_0-a,y_0)=\bs{0}$: we denote by $(F^{\delta_n}_1,f^{\delta_n}_1)$ the pair $(F^{\delta_n}_1(y_0,y_0), f^{\delta_n}_1(y_0,y_0))$ obtained in this way.
Next, on the interval $[y_0,y_0+a]$, we make $F^{\delta_n}_2(x,y_0):=F^{\delta_n}(x,y_0)$ from $F^{\delta_n}_2(y_0+a,y_0)=\mathrm{Id}$ by \eqref{formula:Xalphadelta-x}--\eqref{formula:Xbetadelta-x} and determine ${\bs{f}}^{\delta_n}_2(x,y_0):={\bs{f}}^{\delta_n}(x,y_0)$ and $f^{\delta_n}_2(x,y_0):=f^{\delta_n}(x,y_0)$ in the inverse direction, which satisfies $f^{\delta_n}_2(y_0+a,y_0)=\bs{0}$: we denote by $(F^{\delta_n}_2,f^{\delta_n}_2)$ the pair $(F^{\delta_n}_2(y_0,y_0),f^{\delta_n}_2(y_0,y_0))$ obtained in this way.
Hence, the curve $f^{\delta_n}_2(x,y_0)-f^{\delta_n}_2$ satisfies $f^{\delta_n}_2(y_0,y_0)-f^{\delta_n}_2=\bs{0}$.
For the matrix $C$ such that $CF^{\delta_n}_2=F^{\delta_n}_1$, we define a curve $k^{\delta_n}(x,y_0)$ for $x\in[y_0,y_0+a]$ by $k^{\delta_n}(x,y_0):=C(f^{\delta_n}_2(x,y_0)-f^{\delta_n}_2)+f^{\delta_n}_1$.

Now, two curves $f^{\delta_n}_1(x,y_0)$ on $[y_0-a,y_0]$ and $k^{\delta_n}(x,y_0)$ on $[y_0,y_0+a]$ connect continuously and have the same frame $F^{\delta_n}_1$ at $x=y_0$.
The connected curve at $f^{\delta_n}_1$ is an approximation of the $x$-curve.
Then, since a vector $(yX^{\delta_n}_{\beta}-\phi^{\delta_n})(x,y_0)$ does not depend on $x$ and $\langle yX^{\delta_n}_{\beta}-\phi^{\delta_n},{\bs{f}}^{\delta_n}\rangle(x,y_0)=0$ holds, the curve is contained in a hyperplane $\R_{y_0}^3(\ni \bs{0})$ perpendicular to the vector $(yX^{\delta_n}_{\beta}-\phi^{\delta_n})(x,y_0)$ by $f^{\delta_n}(y_0-a,y_0)=\bs{0}$, and in particular the curve lies on a $2$-sphere of radius $(2\sqrt{5})^{-1}\sqrt{(5+4y_0^2)/(1+y_0^2)}$ in $\R_{y_0}^3$.

\noindent
\begin{figure}[H]
\renewcommand{\thesubfigure}{\alph{subfigure}}
\hfill
\begin{minipage}{0.45\linewidth}
\centering
\includegraphics[width=\linewidth,keepaspectratio]{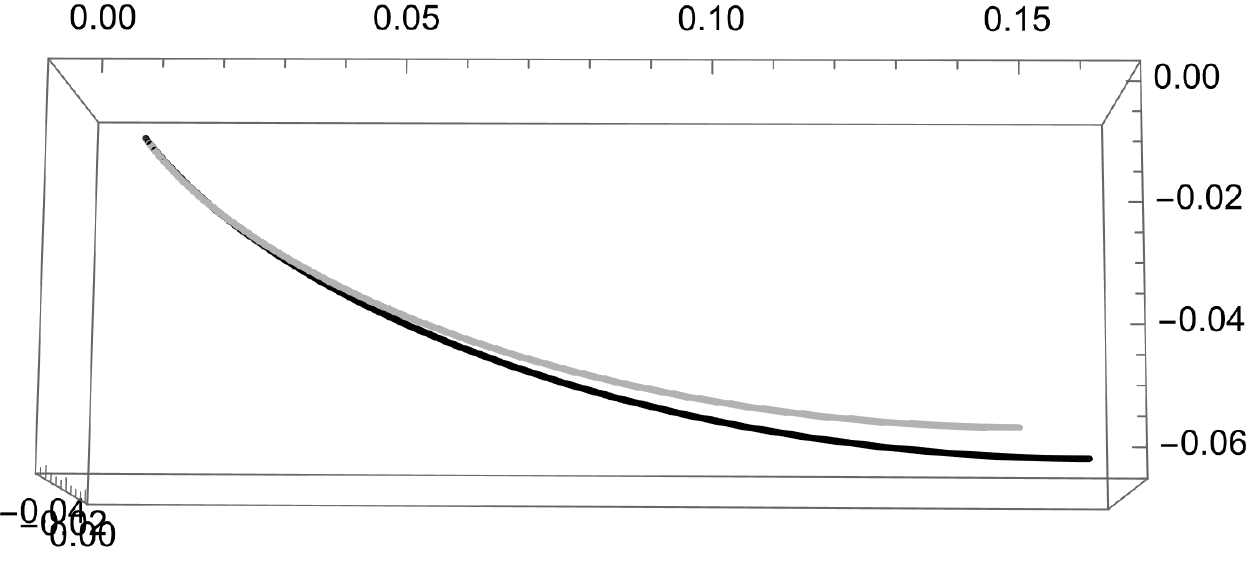}
\subcaption{$x$-curves when $n=5$}
\end{minipage}
\hfill
\begin{minipage}{0.45\linewidth}
\centering
\includegraphics[width=\linewidth,keepaspectratio]{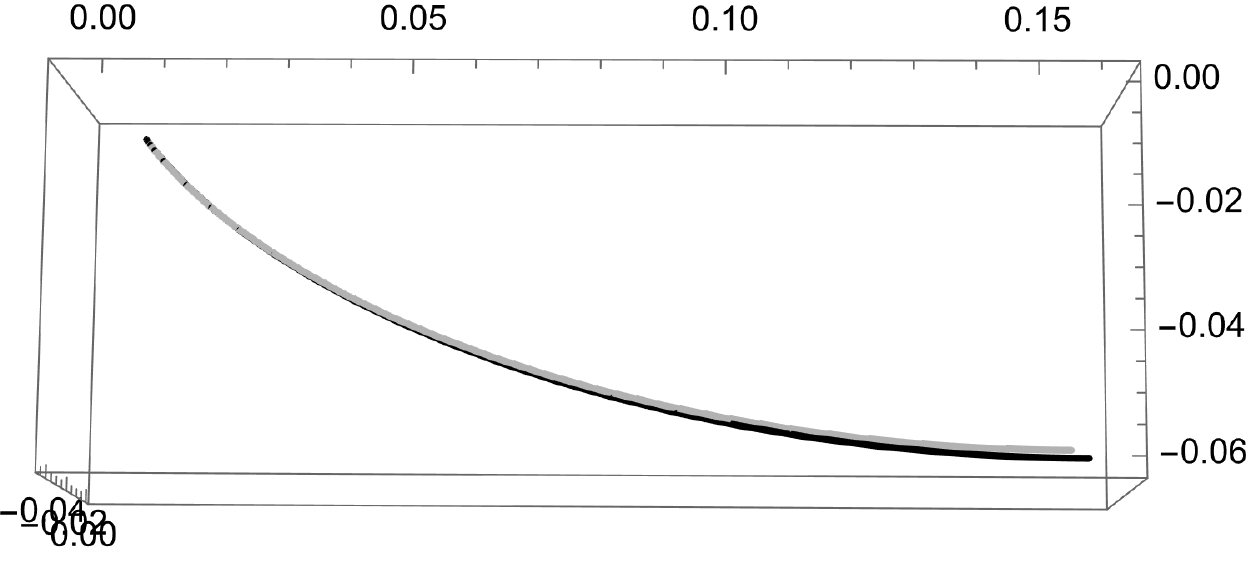}
\subcaption{$x$-curves when $n=20$}
\end{minipage}
\hfill
\caption{
\small
These show the differences between two $x$-curves $f^{\delta_n}(x,1)$ and ${\bar f}^{\delta_n}(x,1)$ on $[2,3]$ under the conditions $f^{\delta_n}(2,1) = {\bar f}^{\delta_n}(2,1)$ and $F^{\delta_n}(2,1) = {\bar F}^{\delta_n}(2,1)$: the $x$-curve $f^{\delta_n}(x,1)$ of black (resp.\ ${\bar f}^{\delta_n}(x,1)$ of gray) is constructed in the direction that $x$ increases (resp.\ decreases).
Here $F^{\delta_n}(x,1)$ (resp.\ ${\bar F}^{\delta_n}(x,1)$) is the frame field determining $f^{\delta_n}(x,1)$ (resp.\ ${\bar f}^{\delta_n}(x,1)$).
\normalsize}
\end{figure}

\subsection{Approximation of $y$-curve $f^0(x_0,y)$ with fixed $x_0(>0)$}\label{subsec:approx-y}

Let $f^0(x_0,y)$ be a $y$-curve for $y\in I=[-a,a]$, where $a>0$.
We take a division of $I$ with equal length $\delta_n:=a/n$.
On the interval $[-a,0]$, we make $F^{\delta_n}(x_0,y)$ from $F^{\delta_n}(x_0,-a)=\mathrm{Id}$ and determine a vector $\tilde{\bs{f}}{}^{\delta_n}(x_0,y)$ on each $[y_j,y_{j+1}]$ by
\begin{gather*}
\tilde{\bs{f}}{}^{\delta_n}(x_0,y):=(B_2\xi^{\delta_n}+C_2X^{\delta_n}_{\alpha})(x_0,y).
\end{gather*}
Then, the approximation $f^{\delta_n}(x_0,y)$ of $f^0(x_0,y) \ (-a\leq y\leq 0)$ is given by
\begin{gather*}
f^{\delta_n}(x_0,y):=
\frac{1}{(B_2^2+C_2^2)(x_0)}
\left(\left[\tilde{\bs{f}}{}^{\delta_n}\right]_{(x_0,y_j)}^{(x_0,y)}+\sum_{k=1}^j\left[\tilde{\bs{f}}{}^{\delta_n}\right]_{(x_0,y_{k-1})}^{(x_0,y_k)}\right) \ \ \ \text{for} \ y_j< y\leq y_{j+1},
\end{gather*}
where $ y_{j+1}\leq y_n=0$.
On the interval $[0,a]$, we make $F^{\delta_n}(x_0,y)$ from $F^{\delta_n}(x_0,a)=\mathrm{Id}$ and determine $\tilde{\bs{f}}{}^{\delta_n}(x_0,y)$ and $ f^{\delta_n}(x_0,y)$ in the inverse direction.
Then, we obtain the approximation of the $y$-curve by the connection of these two curves, in the same way as in the case of $x$-curves.
Then, since $(-B_2X^{\delta_n}_{\alpha}+C_2\xi^{\delta_n})(x_0,y)$ does not depend on $y$ and $\langle -B_2X^{\delta_n}_{\alpha}+C_2\xi^{\delta_n}, f^{\delta_n}\rangle(x_0,y)=0$ holds, the connected curve $f^{\delta_n}(x_0,y)$ is contained a hyperplane $\R_{x_0}^3$ perpendicular to the vector $(-B_2X^{\delta_n}_{\alpha}+C_2\xi^{\delta_n})(x_0,y)$, and in particular it lies on a $2$-sphere of radius $1/\sqrt{B_2^2+C_2^2}(x_0)$ in $\R_{x_0}^3$.
Furthermore, we have $\lim_{y\to\infty}f^{\delta_n}(x_0,y)=\lim_{y\to\infty}f^{\delta_n}(x_0,-y)$ as mentioned in Corollary \ref{cor:5.3}.
We transform the curve $f^{\delta_n}(x_0,y)$ to $\hat{f}^{\delta_n}(x_0,y)$ isometrically by $\hat{f}^{\delta_n}(x_0,y)={F^{\delta_n}(x_0,0)}^{-1}(f^{\delta_n}(x_0,y)-f^{\delta_n}(x_0,0))$. 
Then, $\hat{f}^{\delta_n}(x_0,0)(=\hat{f}^{\delta_n}(x_0,y_n))=\bs{0}$ and $\hat{F}^{\delta_n}(x_0,0)=\mathrm{Id}$ hold, and hence 
we have $(B\circ \hat{f}^{\delta_n})(x,-y)=\hat{f}^{\delta_n}(x,y)$ with respect to the reflection $B$ defined in Corollary \ref{cor:Af0}.

\subsection{Approximation of the curve $f^0(y,y)$ for $b\geq y\geq a(>0)$}

Let $F^0(x,y)$ be a solution to \eqref{eq:dF0} with $F^0(a,a)=\mathrm{Id}$.
Let us take a division $a=y_0<y_1<y_2<\cdots <y_n=b$ of equal length $\delta_n=(b-a)/n$ and a path $m_n: (a,a)=(y_0,y_0)\rightarrow(y_0,y_1)\rightarrow(y_1,y_1)\rightarrow(y_1,y_2)\rightarrow\cdots \rightarrow(y_{n-1},y_{n-1})\rightarrow(y_{n-1},y_n)\rightarrow(y_n,y_n)=(b,b)$.
Let $F^{\delta_n}(x,y)$ be an orthonormal frame on $m_n$ determined by $F^{\delta_n}(a,a)=\mathrm{Id}$ and the path $m_n$.
For $F^{\delta_n}(x,y)$, the vectors $\tilde{\bs{f}}{}^{\delta_n}(y_i,y)$ in \S \ref{subsec:approx-y} and ${\bs{f}}^{\delta_n}(x,y_{i+1})$ in \S \ref{subsec:approx-x} are determined  on each edge $\{y_i\}\times[y_i,y_{i+1}]$ and $[y_i,y_{i+1}]\times\{y_{i+1}\}$, respectively.
Then, since each lattice point $(y_i,y_i)$ for $y$-curve $(y_i,y)$ is not singular for the metric $g_0$ on $D$, the curve $f^{\delta_n}(x,y)$ determined from these vectors is an approximation of the curve $f^0(x,y)$ on $m_n$.
In consequence, we obtain a sequence $f^{\delta_n}(y_i,y_i) \ (i=0,1,\dots, n)$ of points, which approximates to the curve $f^0(y,y)$ as $n\rightarrow\infty$.

Next, we shall attach the $x$-curve $f^{\delta_n}(x,y_i)$ and $y$-curve $f^{\delta_n}(y_i,y)$ to each points $f^{\delta_n}(y_i,y_i)$.
Then, in order that the frames $F^{\delta_n}(x,y_i)$ on $x$-line and $F^{\delta_n}(y_i,y)$ on $y$-line satisfy \eqref{ineq:norm-Fdeltas}, these frames must be taken in the direction that $x$ and $y$ increase. We can make the frames $F^{\delta_n}(x,y_i)$ and $F^{\delta_n}(y_i,y)$ such that $F^{\delta_n}(y_i,y_i)$ is consistent with one obtained above, in similar way to \S \ref{subsec:approx-x}.
Thus, the $x$-curve $f^{\delta_n}(x,y_i) \ (x\leq y_i)$ and the $y$-curve $f^{\delta_n}(y_i,y)$ are determined as in \S \ref{subsec:approx-x} and \S \ref{subsec:approx-y}.

However, only for the $x$-curve on $[y_i,y_{i+1}]\times \{y_i\}$, we have $f^{\delta_n}(x,y_i)\equiv f^{\delta_n}(y_i,y_i)$ for $x\in[y_i,y_{i+1}]$, if we make $f^{\delta_n}(x,y_i)$ in the direction that $x$ increases.  
Therefore, for the $x$-curve on $[y_i,y_{i+1}]\times\{y_i\}$ we substitute $k^{\delta_n}(x,y_i)$ made by the method described in \S \ref{subsec:approx-x}, which approximates $f^0(x,y_i)$ on $[y_i,y_{i+1}]\times \{y_j\}$ in the inverse direction of $x$. For $x\geq y_{i+1}$ we make the $x$-curve $f^{\delta_n}(x,y_i)$ from the point $k^{\delta_n}(y_{i+1},y_i)$ and the frame $F^{\delta_n}(y_{i+1},y_i)$ determined above.

\end{document}